\documentclass[11pt]{amsart}
 \usepackage{amssymb,enumerate,mathtools}

 \theoremstyle{plain}
\newtheorem{theorem}{Theorem}[section]
 \newtheorem{cor}[theorem]{Corollary}
 \newtheorem{lem}[theorem]{Lemma}
  \newtheorem{lemma}[theorem]{Lemma}
 \newtheorem{rem}[theorem]{Remark}
 
  \newtheorem{defi}[theorem]{Definition}
   \newtheorem{definition}[theorem]{Definition}
 \newtheorem{theo}[theorem]{Theorem}
 \newtheorem{thm}[theorem]{Theorem}
  \newtheorem{pro}[theorem]{Proposition}

 \newtheorem*{ackn}{Acknowledgements}

  \theoremstyle{plain}
\newtheorem*{namedthm}{\namedthmname}
\newcounter{namedthm}

\makeatletter
\newenvironment{named}[1]
  {\def\namedthmname{#1}%
   \refstepcounter{namedthm}%
   \namedthm\def\@currentlabel{#1}}
  {\endnamedthm}

 \newcommand{\R}{\mathbb R}
 \newcommand{\bR}{\mathbb R}
 \newcommand{\Q}{\mathbb Q}
 \newcommand{\C}{\mathbb C}
 \newcommand{\N}{\mathbb N}

 \newcommand{\Pc}{\mathcal{P}}

 \newcommand{\e}{\varepsilon}
 \newcommand{\f}{\varphi}
 
 \newcommand{\p}{\psi}

 \newcommand \psh {plurisubharmonic }
 \newcommand{\PSH} {{\rm PSH}}
 \newcommand{\Tr} {{\rm Tr}}

 \newcommand \Sub {\Subset}

 \newcommand \mc   {\mathcal}

 \newcommand{\Ric}{{\rm Ric}}
 
  \newcommand{\Vol}{{\rm Vol}}
  \newcommand{\loc}{{\rm loc}}
  \newcommand{\capacity}{{\rm Cap}}

 \newcommand{\setdef}{\, ; \, }
 \numberwithin{equation}{section}
 
\usepackage{color}

\usepackage{hyperref}
\hypersetup{
    bookmarks=true,         
    unicode=false,          
    pdftoolbar=true,       
    pdfmenubar=true,       
    pdffitwindow=false,    
    pdfstartview={FitH},   
    pdftitle={Pluripotential K\"ahler-Ricci flows},    
    pdfauthor={V. Guedj, C.H. Lu, A. Zeriahi},     
    colorlinks=true,       % false: boxed links; true: colored links
   linkcolor=black,          % color of internal links
    citecolor=black,        % color of links to bibliography
    filecolor=black,      % color of file links
    urlcolor=black}           % color of external links

\begin{document}

%
% Author info

\author{Vincent Guedj}
\address{Vincent Guedj, Institut de Math\'ematiques de Toulouse  \\ 
Universit\'e de Toulouse, CNRS \\
UPS, 118 route de Narbonne \\
31062 Toulouse cedex 09, France}
\email{vincent.guedj@math.univ-toulouse.fr}
\urladdr{\href{https://www.math.univ-toulouse.fr/~guedj}{https://www.math.univ-toulouse.fr/~guedj/}}

\author{Chinh H.  Lu}
\address{Hoang-Chinh Lu, Laboratoire de Math\'ematiques d'Orsay,
 Univ. Paris-Sud,
 CNRS, Universit\'e Paris-Saclay,
  91405 Orsay, France}
\email{hoang-chinh.lu@math.u-psud.fr}
\urladdr{\href{https://www.math.u-psud.fr/~lu/}{https://www.math.u-psud.fr/~lu/}}

\author{Ahmed Zeriahi}
\address{Ahmed Zeriahi, Institut de Math\'ematiques de Toulouse,   \\ Universit\'e de Toulouse, CNRS \\
UPS, 118 route de Narbonne \\
31062 Toulouse cedex 09, France}
\email{ahmed.zeriahi@math.univ-toulouse.fr}
\urladdr{\href{https://www.math.univ-toulouse.fr/~zeriahi/}{https://www.math.univ-toulouse.fr/~zeriahi/}}

\thanks{The authors are partially supported by the ANR project GRACK}
  
%
% Paper metadata
%
\keywords{Parabolic Monge-Amp\`ere equation,  pluripotential solution, Perron envelope, K\"ahler-Ricci flow}

\subjclass[2010]{53C44, 32W20, 58J35}

\title[Pluripotential K\"ahler-Ricci flows]{Pluripotential K\"ahler-Ricci flows}
\date{\today}

 \begin{abstract}  
We develop   a parabolic pluripotential theory
on compact K\"ahler manifolds, defining and studying
weak solutions to degenerate parabolic complex Monge-Amp\`ere equations.
We provide
a parabolic analogue of the celebrated Bedford-Taylor theory
and apply it to the study of the K\"ahler-Ricci flow on varieties with log terminal singularities. 
\end{abstract}

 \maketitle

\tableofcontents

\section*{Introduction}

The Ricci flow, first introduced by Hamilton \cite{Ham82} is the equation
$$
\frac{\partial}{\partial t} g_{ij}=-2R_{ij},
$$
evolving a Riemannian metric by its Ricci curvature. If the Ricci flow starts from a K\"ahler metric -the underlying Riemannian manifold being complex K\"ahler-, the evolving metrics remain K\"ahler and the resulting PDE is called the K\"ahler-Ricci flow.

After the spectacular use of the Ricci flow by Perelman to settle the Poincar\'e  and Geometrization conjectures, 
it is expected that the K\"ahler-Ricci flow can be used similarly  to give a geometric classification of 
complex algebraic and K\"ahler manifolds, and produce canonical metrics at the same time.

Understanding the existence of canonical K\"ahler metrics
 on compact K\"ahler manifolds has been a central question in the last fourty years,
 following Yau's solution to the Calabi conjecture \cite{Yau78}.
The K\"ahler-Ricci flow provides a canonical deformation process towards such metrics,
as shown by the works of many authors (see e.g. \cite{Cao85,PSSW08,PSSW09,ST12,SW13,SSW13,CT15,TZ15,CS16,BBEGZ} and the references therein).

Writing locally $g_{ij}=\psi_{ij}=\partial_i\bar{\partial}_j\psi$,  it is classical that the K\"ahler-Ricci flow can be reduced to a nonlinear parabolic scalar equation in $\psi$, of the form
$$
 \det \left( \psi_{ij} \right) =e^{ \dot{\psi}_t + H(t,x)+ \lambda \psi_t },
$$
where $H$ is a smooth density, and $\lambda \in \R$ depends on $c_1(X)$.

The classification of complex algebraic manifolds
requires to work on singular varieties, as advocated by the Minimal Model Program.
Defining the K\"ahler-Ricci flow on midly singular projective varieties was undertaken 
by Song-Tian \cite{ST17}
and requires a theory of weak solutions for  degenerate parabolic complex Monge-Amp\`ere equations,
where $\psi$ is no longer smooth and $H$ can blow up.

A parabolic viscosity approach has been developed in \cite{EGZ16}.
It applies to the K\"ahler context, but requires the   densities to be continuous.
This enabled one  to study the
% long term 
behavior of the 
%(normalized) 
K\"ahler-Ricci flow on 
minimal models with positive Kodaira dimension and canonical singularities \cite{EGZ16b}.

While both the approach of Song-Tian and the viscosity one permit a good understanding of
the first singular situations encountered in the Minimal Model Program, 
one needs to extend these theories in order to treat  the fundamental case of
K\"ahler pairs with Kawamata log terminal (klt) singularities.
 {\it This is the main objective of the present work.}
 
\smallskip

From an analytic point of view klt singularities lead one to deal with densities that  may blow up, though belonging to $L^p$ for some exponent $p>1$ whose size is related
to the algebraic nature of the singularities.

%\smallskip
 
We develop in this article a parabolic pluripotential approach to the  complex Monge-Amp\`ere flows
\begin{equation}\label{eq:CMAF}  \tag{CMAF}  
 (\omega_t +i \partial \overline{\partial} \f_t)^n=e^{\dot{\f}_t + F (t,x,\f) } g(x) dV (x), 
\end{equation}
in $X_T :=]0,T[ \times X$, where $T \in ]0,+\infty]$ and
  \begin{itemize}
 \item $X$ is a compact K\"ahler $n$-dimensional manifold.
  \item $t \mapsto \omega(t,x)$ is a ${\mathcal C}^2$-family of  closed  semi-positive $(1,1)$-forms 
  %on $X$ 
  such that  $ \theta(x) \leq \omega_t(x) $,
  %\leq \Theta(x)  $
where $\theta$ is a closed semi-positive big form 
with
$$
-A \omega_t \leq \dot{\omega}_t \leq \frac{A}{t} \omega_t
\; \; \text{ and } \; \; 
\ddot{\omega}_t \leq A \omega_t
$$
for some fixed constant $A>0$;
  \item $(t,x,r) \mapsto F (t,x,r)$ is continuous in $[0,T[ \times X \times \R$, quasi-increasing  in $r$,
  locally uniformly Lipschitz and semi-convex  in $(t,r)$;
  \item $g \in L^p(X,dV)$, $p>1$, with $g>0$ almost everywhere;
  \item   $\f : [0,T[ \times X \rightarrow \R$ is the unknown function, with $\f_t: = \f (t,\cdot)$. 
   \end{itemize}
Here $d V$ is a fixed normalized volume form on $X$. 

\smallskip
 
 We introduce a notion of pluripotential solutions to such equations, 
 a parabolic analogue of the theory developed by Bedford and Taylor in their celebrated articles \cite{BT76,BT82}.
 
 We interpret the above parabolic equation on $X$ as a  second order PDE
 on the $(2n+1)$-dimensional manifold $X_T$ :
 \begin{itemize}
 \item the LHS becomes a positive Radon measure
 $(\omega_t +dd^c \f_t)^n \wedge dt$, which is well defined for  paths $t \mapsto \f_t$ of bounded
 $\omega_t$-psh functions \cite{BT82}, 
 \item the RHS  $e^{\dot{\f}_t + F (t,x,\f) } g(x) dV(x) \wedge dt$
 is a well-defined Radon measure if $t \mapsto \f_t(x)$ is (locally) uniformly Lipschitz.
 \end{itemize}
 It is useful in practice to allow the Lipschitz constant to blow up as $t$ approaches zero, so
we introduce the corresponding class $ \mathcal P (X_T, \omega)$ of {\it parabolic potentials} (see Definition \ref{defi:parabpot}).

We develop the local side of this theory in \cite{GLZparab1} 
by a direct approach, taking advantage of the euclidean structure of $\C^n$.
We approximate here  (\ref{eq:CMAF}) by smooth complex Monge-Amp\`ere flows and establish various a priori estimates
to prove our first main result :

\begin{named}{Theorem A}
	\label{main Thm A}
Let $\f_0$ be a bounded $\omega_0$-psh function.
There exists a parabolic potential $\f \in \mathcal P (X_T, \omega)$ such that
\begin{itemize}
\item $(t,x) \mapsto \f(t,x)$ is  locally  bounded in $[0,T[ \times X$;
\item $(t,x) \mapsto \f(t,x)$ is continuous   in $]0,T[ \times \rm{Amp}(\theta)$;
\item $t \mapsto \f_t$ is locally uniformly semi-concave in $]0,T[ \times X$;
\item $\f$ is a pluripotential solution to \eqref{eq:CMAF};
\item $\f_t \rightarrow \f_0$ as $t \rightarrow 0^+$ in $L^1(X)$ and pointwise.
\end{itemize}
\end{named}

Here $\rm{Amp}(\theta)$ denotes the ample locus of $\theta$, i.e. the largest Zariski open subset of $X$ where
the cohomology class of $\theta$ behaves like a K\"ahler class.

It turns out that $t \mapsto \f_t(x)-n(t \log t-t)+Ct$ is increasing for some fixed $C>0$.
The convergence at time zero is therefore rather strong (it is e.g. uniform in case 
$\f_0$ is continuous).

\smallskip

The semi concavity information of the solution $\f$ constructed in Theorem A is a crucial tool for approximation purpose (see Theorem \ref{thm:conv1}). We show that it is the unique pluripotential solution with such time-regularity, by establishing
the following comparison principle :

\begin{named}{Theorem B}  
	\label{main Thm B}
%	Assume the same assumptions as in \ref{main Thm A}. 
 If $\varphi\in \Pc(X_T,\omega)$ is a bounded pluripotential subsolution to \eqref{eq:CMAF} and $\psi \in \Pc(X_T,\omega)$ is a bounded pluripotential supersolution which is locally uniformly semi-concave in $t$ then  
	$$
	\varphi_0 \leq \psi_0 \Longrightarrow \varphi \leq \psi.
	$$
In particular there is a unique bounded pluripotential solution $\Phi(g,F,\omega_{t},\f_{0})$
to \eqref{eq:CMAF} which is locally uniformly semi-concave in $t$. 
%This is the upper envelope of all pluripotential subsolutions.  
\end{named}

This comparison principle also allows us to establish the following stability result which   generalizes
\cite[Theorem B]{GLZstability} :

\begin{named}{Theorem C}  
	\label{main Thm C}
Assume
\begin{itemize}
\item $(g_j)$ are  densities which converge to $g$ in $L^p$,
\item $F_j$ converges to $F$ with uniform constants;
%Lipschitz, semi-convexity, and quasi-increasing ;
\item   $\omega_{t,j}$ are smooth semi-positive forms smoothly converging to $\omega_t$,
\item $\f_{0,j}$ are bounded $\omega_{0,j}$-psh functions converging in $L^1(X,dV)$ to $\f_0$. 
\end{itemize}
Then  $\Phi(g_j,F_j,\omega_{t,j},\f_{0,j})$ locally uniformly converges to 
$\Phi(g,F,\omega_{t},\f_{0})$.
\end{named}

It is  delicate to compare pluripotential and viscosity concepts in general.
We refer the interested reader to \cite{GLZparab3} where we prove, when $g$ is continuous, 
that the viscosity solution constructed in \cite{EGZ16}
coincides with the pluripotential solution $\Phi(g,F,\omega_{t},\f_{0})$. 

\smallskip

The present pluripotential approach allows us to deal with non continuous data.
We can, in particular, define a
good notion of weak K\"ahler-Ricci flow on varieties with terminal singularities
(and more generally on k.l.t. pairs), as we explain in section \ref{sec:lt},
where we prove the following :

\begin{named}{Theorem D}
	\label{main Thm D}
	Let $(Y,\omega_0)$ be a compact $n$-dimensional K\"ahler variety
	with log terminal singularities and trivial first Chern class
	($\Q$-Calabi-Yau variety).
	
	Fix $S_0$ a positive closed current with bounded potentials,
 whose cohomology class is K\"ahler.
 The  K\"ahler-Ricci flow 
 $$
 \frac{\partial \omega_t}{\partial t}=-{\rm Ric}(\omega_t)
 $$
 exists for all times $t>0$, and deforms $S_0$ towards 
the unique Ricci flat K\"ahler-Einstein current 
 $\omega_{KE}$ cohomologous to $S_0$, as $t \rightarrow +\infty$.
\end{named}

This extends previous results of \cite{Cao85,Tsu88,TZ06}, avoiding any projectivity assumption on $X$
 \cite{ST17}, nor any restriction on the type of singularities \cite{EGZ16,EGZ16b}.
We refer the reader to section \ref{sec:lt} for  much more general and precise results.

 \subsection*{Assumptions on the data and notations} \label{subsect: ass data uniqueness}
 
 \subsubsection*{Assumptions on the manifold}
  In the whole article we let $X$ be a compact K\"ahler $n$-dimensional manifold.
  We   fix $T\in ]0,+\infty]$.
Except for section \ref{sec:lt} we are mainly concerned with finite time intervals, i.e. $T<+\infty$, and we implicitly assume
that our data are possibly defined in a slightly larger time interval, i.e. on
$(0,T+\e)$ for some $\e>0$.

We let $X_T$ denote the $(2n+1)$-dimensional manifold $X_T=]0,T[ \times X$ with parabolic boundary
$$
\partial X_T:=\{0\} \times X.
$$

We fix $\theta$ a smooth closed semipositive $(1,1)$-form  whose cohomology class is big,
i.e. contains a (singular) positive closed current of bidegree $(1,1)$ which dominates a K\"ahler form.
% (i.e. $\int_X \theta^n>0$) 
We let $\Omega$ denote the ample locus of $\theta$,
$$
\Omega:={\rm Amp}(\theta),
$$
which is a non empty Zariski open subset of $X$.

\subsubsection*{Assumptions on the forms}
   
  We assume throughout the article that
 $(\omega_t)_{t\in [0,T[}$ is a $\mathcal{C}^2$-smooth family of closed semipositive $(1,1)$-forms on $X$ satisfying
 \begin{equation*}
	%\label{eq:omegatbound}
	\theta \leq \omega_t 
\end{equation*}
for all $t\in [0,T[$.
For finite times we can also assume without loss of generality that $\omega_t \leq \Theta$
for some  K\"ahler form $\Theta$.

By the end of Section \ref{sec:apriori} we need to assume that 
$t \mapsto \omega_t$ moreover satisfies
 \begin{equation*}
 	%\label{eq:omegatconcave2}
 	\ddot{\omega}_t \leq A\omega_t ,  
 \end{equation*}
 and 
 \begin{equation}
 	\label{eq:omegatLip}
 -A \omega_t \leq  	\dot{\omega}_t \leq A \omega_t, 
 \end{equation}
for some constant $A>0$.
The lower bound in \eqref{eq:omegatLip} is equivalent to the fact that $t \mapsto e^{+At} \omega_t$ is increasing.
 In particular
$$
\omega_{t+s} \geq e^{-As} \omega_t \geq (1-As) \omega_t, \ s>0.
$$
The latter will be used on several occasions in the sequel.

\subsubsection*{Assumptions on the densities}

 We assume throughout the article that
\begin{itemize}
%\item $\varphi_0$ is a bounded $\omega_0$-psh function on $X$. 
\item $dV$ is a fixed volume form on $X$; 
\item $0\leq g \in L^p(X,dV)$ for some $p>1$, and ${\rm Vol}(\{g=0\})=0$,
\item $(t,x,r) \mapsto F(t,x,r)$ is a continuous function on $[0,T[ \times X\times \mathbb{R}$; 
\item $r \mapsto F(\cdot,\cdot,r)$ is uniformly quasi-increasing, i.e. there exists a constant  $\lambda_F\geq 0$ such that
 for every $(t,x) \in [0,T[ \times X$, the function 
\begin{equation}
	\label{eq:Fincreasing}
r \mapsto  F(t,x,r)+ \lambda_F r \ \text{is increasing in }   \mathbb{R}. 
\end{equation}
\item $(t,r) \mapsto F(t,\cdot,r)$  is locally uniformly  Lipschitz, i.e. for all $J\Subset [0,T[ \times \mathbb{R}$ there is $\kappa_J>0$ such that
 for every $x\in X$, $(t,r)$, $(t',r') \in J$,  
\begin{equation}
	\label{eq:FLip}
	|F(t,x,r) -F(t',x,r')| \leq \kappa_J( |t-t'| + |r-r'|); 
\end{equation}
 \item $(t,r) \mapsto F(t,x,r)$  is locally uniformly  semi-convex, i.e. for every compact $J\Subset [0,T[\times \mathbb{R}$ there exists   $C_J>0$ such that
 for every $x\in X$,  
\begin{equation} \label{eq:FconvIntro}
(t,r) \mapsto  F(t,x,r)+ C_J (t^2+r^2) \ \text{is convex in} \ J. 
\end{equation}
 \end{itemize}
 
 Note that if $F$ is ${\mathcal C}^2$-smooth then the local conditions \eqref{eq:FLip} and \eqref{eq:FconvIntro}  are automatically satisfied, while \eqref{eq:Fincreasing} is a global assumption.

 \subsubsection*{Invariance properties of the set of assumptions}
 
We check in section \ref{sec:checkgeom} that the above conditions 
 are   satisfied for the parabolic equations that describe the evolution of
 the normalized (as well as the non-normalized) K\"ahler-Ricci flow on
 a mildly singular K\"ahler variety.

 The family of parabolic complex Monge-Amp\`ere equations we consider
  enjoy several useful invariance properties. We refer the reader to section \ref{sec:invariance}
  for more details.

\subsection*{Organization of the paper} 
We describe the class of potentials we are using in 
Section \ref{sec:recap} and define parabolic complex Monge-Amp\`ere operators in 
Section \ref{sec:operator}.
%and study basic convergence results for the latter.
We establish fundamental a priori estimates in Section \ref{sec:apriori},
which are then used to prove \ref{main Thm A} in Section \ref{sec:existence}.
We study uniqueness and stability of pluripotential solutions in Section \ref{sec: uniqueness}, establishing  \ref{main Thm B}
and \ref{main Thm C}.
In Section \ref{sec:lt} we use these tools to study the long term behavior of the
normalized K\"ahler-Ricci flow on varieties with log terminal singularities and non-negative Kodaira dimension,
proving \ref{main Thm D} and several other convergence results.

\begin{ackn} 
This work is a natural continuation of \cite{EGZ16,EGZ16b}. 
We thank Philippe Eyssidieux for many useful discussions.
\end{ackn}

\section{Parabolic potentials and Monge-Amp\`ere operators}

\subsection{Families of quasi-plurisubharmonic functions} \label{sec:recap}

\subsubsection{Compactness properties}

 Recall that a function $u:X \rightarrow [- \infty , + \infty[$ is $\omega_t$-plurisubharmonic
 ($\omega_t$-psh for short), if it is locally given as the sum of a smooth and a 
 plurisubharmonic function
 %, in particular $u \in L_{loc}^1(X)$, 
 and the current
 $$
 \omega_t+dd^c u \geq 0
 $$
is positive on $X$. Here $d=\partial+\overline{\partial}$
and  $d^c=i(\overline{\partial}-\partial)$ are both real operators.
%Note that $dd^c=2i \partial \overline{\partial}$.
 
 \begin{defi} \label{defi:parabpot}
The set of parabolic potentials  $ \mathcal P (X_T, \omega)$ is the set of functions  
$ \f :  ]0,T[ \times  X \longrightarrow [- \infty , + \infty[$  such that
\begin{itemize}
%\item $\f : ]0,T[ \times X \longrightarrow [- \infty, + \infty[$ is upper semi-continuous, 
 \item $x \mapsto \f (t,x)$ is $\omega_t$-plurisubharmonic on $X$, for all $t  \in ]0,T[$,
 \item  $\f$ is locally  uniformly  Lipschitz in $]0,T[$.
\end{itemize}  
\end{defi}

The last condition means that for any compact subset $J \subset ]0,T[$ there exists  $\kappa = \kappa_J(\f)> 0$
  such that 
 \begin{equation} \label{eq:Lip}
  \f (t,x)  \leq  \f (s,x)  + \kappa \vert t - s\vert,
\text{  for all } s, t \in J
\text{ and } x \in {X}.
 \end{equation}

We  say that a family $\Phi \subset \mathcal P (X_T, \omega)$ is locally uniformly Lipschitz in $]0,T[$ if the inequality 
(\ref{eq:Lip}) is satisfied for all $\f \in \Phi$ with a uniform constant $\kappa = \kappa (J,\Phi) > 0$ 
which only depends  on $J$ and $\Phi$.

\smallskip

A parabolic potential $\varphi \in \Pc(X_T,\omega)$ can be extended as an upper semicontinuous function on $[0,T[\times X$ with $\omega_t$-psh slices. 

\begin{pro}
	\label{pro: usc}
	Assume $\f_0$ is $\omega_0$-psh and 
	$\f \in \Pc(X_T,\omega)$ satisfies $\f_t {\rightarrow} \f_0$ in $L^1$, as $t \rightarrow 0$.
 Then the extension $\f: [0,T[ \times X \rightarrow [-\infty,+\infty[$ is upper semi-continuous. 
\end{pro}

 \begin{proof}
 It is classical that for all $x \in X$, $\f_0(x)=\limsup_{y \rightarrow x} \limsup_{t \rightarrow 0} \f_t(y)$.
 It therefore suffices to prove the following more general result : 
 assume that $u\in \Pc(X_T,\omega)$  is  bounded from above near $t=0$ and define 
	$$
	u_0(x) := \limsup_{y\to x} \left(\limsup_{t\to 0} u_t(y)\right).
	$$
 Then the extension $u: [0,T[ \times X \rightarrow [-\infty,+\infty[$ is upper semicontinuous.

 The upper semi-continuity inside $X_T$ follows from the semi-continuity in space and Lipschitz regularity in time.
 Assume that $(t_j,x_j)$ is a sequence in $X_T$ converging to $(0,x_0)$ with $x_0\in X$. We want to prove that $$\limsup_j u(t_j,x_j) \leq u_0(x_0).$$
 	 Since the problem is local we can assume that the functions $u_{t_j}$ are psh and negative in a neighborhood $B\subset \mathbb{C}^n$ of $x_0$. 
 	 
 	 Since $u_0$ is psh in $B$ there exists $r>0$ such that $B(x_0,2r) \subset B$ and 
 	 \begin{equation}\label{eq: basic psh}
 	 	\frac{1}{\Vol(B(x_0,r))}\int_{B(x_0,r)} u_0(z)dV(x) \leq u_0(x_0) +\varepsilon.
 	 \end{equation}
Fix $\delta \in ]0,r[$. For $j$ large enough,  $x_j\in B(x_0,\delta)$ hence $B(x_0,r) \subset B(x_j,r+\delta)$ and since $u_{t_j}\leq 0$ in $B$ we have, 
\begin{flalign*}
	u(t_j,x_j)&  \leq \frac{1}{\Vol(B(x_j,r+\delta))}\int_{B(x_j,r+\delta)} u(t_j,x)dV(x)\\
	&  \leq \frac{1}{\Vol(B(x_j,r+\delta))}\int_{B(x_0,r)} u(t_j,x)dV(x)\\
	& = \frac{\Vol(B(x_0,r))}{\Vol(B(x_j,r+\delta))}  \frac{1}{\Vol(B(x_0,r))} \int_{B(x_0,r)} u(t_j,x)dV(x).
\end{flalign*} 
Since $\limsup_{j}u_{t_j}(x) \leq u_0(x)$, for all $x\in X$,  letting $j\to +\infty$ and using \eqref{eq: basic psh}  we obtain 
$$\limsup_{j}u(t_j,x_j) \leq  \frac{\Vol(B(z_0,r))}{\Vol(B(z,r+\delta))}  (u_0(x_0)+ \varepsilon).$$
Now, we first let $\delta\to 0$ and then $\varepsilon\to 0$ to obtain the result.   
 \end{proof}

We next prove a compactness result for this class of functions.
 
\begin{thm} \label{thm:Montel}   
Let $(\f_j) \subset \mathcal P (X_T,\omega)$ be a sequence which 
\begin{itemize}
\item is locally uniformly bounded from above in $X_T$;
\item is locally uniformly Lipschitz in $]0,T[$;
\item does not converge locally uniformly to $- \infty$ in $X_T$. 
\end{itemize}

Then  $(\f_j)$ is bounded in $L^1_{\loc} (X_T)$ and there exists a subsequence which converges to some function 
$\f \in  \mathcal P (X_T)$ in the $L^1_{\loc} (X_T)$-topology.

If  $(\f_j)$ converges weakly (in the sense of distributions) to $\f$ in $X_T$, then it converges in $L^p_{\loc} (X_T)$  for all $p \geq 1$. 
%and $(\f_j)$ converges to $\f$ in $W^{1,1}_{loc} (X_T)$ ?.
\end{thm}

The classes $L^p$ are here defined with respect to the $(2 n + 1)$-dimensional  Lebesgue measure 
 associated to a fixed volume form $d t \wedge d V$. For convenience we normalize $dV$ so that $\int_X dV=1$.

 \begin{proof}
  The proof of this result is local in nature and follows closely the classical proof of the analogous result for 
  quasi-plurisubharmonic functions, once we have a substitute for the sub-mean value inequality.
  
  We can thus assume here that $X=\Omega \subset \C^n$ is a bounded strictly pseudoconvex domain.
The Poincar\'e lemma insures that $\omega_t=dd^c \rho_t$ for a
family of plurisubharmonic functions $\rho_t$ which is Lipschitz in $t$.
Changing  $\f_t$ in $\f_t+\rho_t$, we reduce further to the case when $\omega_t=0$.
The corresponding compactness and convergence properties 
have then been obtained in \cite{GLZparab1}.
 \end{proof}
 
 \begin{cor}
The class $\mathcal P (X_T,\omega)$ is a  subset of $ L^p_{\loc} (X_T)$ for all $1 \leq p$, and the inclusions
$\mathcal P (X_T,\omega) \hookrightarrow  L^p_{\loc} (X_T)$ are continuous.
\end{cor}

The topologies induced by the classes $L^p$ are thus all equivalent when restricted to the class $\mathcal P (X_T,\omega)$.

 \subsubsection{Slices and time-derivatives}
 
We now estimate the $L^1$-norm on slices. 

 \begin{lem} \label{lem:L1Slice-L1} 
Fix $u, v \in \mathcal P (X_T,\omega)$ and  $0 < T_0 < T_1  < T$. 
Then  
$$
 \Vert u (t,\cdot) - v (t,\cdot)\Vert_{L^1 (X)} \leq 2  M \max \left\{\Vert u - v\Vert_{L^1 (X_T)}^{1 \slash2}, \Vert u - v\Vert_{L^1 (X_T)} \right\},
$$
for all $T_0 \leq t \leq T_1$, 
 where 
 $
 M :=\max \{\sqrt{ \kappa}, (T - T_1)^{- 1}\},
 $
 and $\kappa$ is the uniform Lipschitz constant of $u-v$ in $[T_0,T]$.
 \end{lem}
 
 This lemma expresses in a quantitative way the following fact : for functions in $\mathcal P (X_T,\omega) $, the convergence in $L^1 (X_T)$ 
 implies the local uniform convergence of their slices in $L^1 (X)$ : if $(u_j) \subset {\mathcal P}(X_T,\omega)$ converges to $u$
in $L^1(X_T)$ and is locally uniformly Lipschitz in $]0,T[$,
% with uniform Lipschitz constants in arbitrary  time intervals $[\e,T'] \subset (0,T)$,
then $u_j(t,\cdot)$ converges to $u(t,\cdot)$ in $L^1(X)$ for each slice $t$.
 
 \begin{proof} 
 The proof is identical to the corresponding one in the local context, we refer the reader to \cite{GLZparab1}.
\end{proof}

 Fix $\mu$ a (finite) Borel measure on $X$, and let $\ell$ denote the Lebesgue measure on $\R^+$.
% The upper and lower  partial derivatives of $\f$ with respect to $t$ are defined by  
% $$
% \partial_t^+ \f (t,x) := \limsup_{s \searrow 0^+} \frac{\f (t + s,x) - \f (t,x)}{s} 
% $$
% and 
% $$
% \partial_t^- \f (t,x) :=\liminf_{s \nearrow 0^-} \frac{\f (t + s,x) - \f (t,x)}{s}.
% $$
% When $\partial_t^+ \f (t,x)=\partial_t^- \f (t,x)$, we say that $\partial_t \f (t,x)$ exists.
% 

\begin{lem} \label{lem: existence of time derivative}
Fix $\f \in \mathcal P (X_T,\omega)$. 
Then $\partial_t \f (t,x)$ exists for all $(t,x) \notin E$, where
 $E \subset X_T$ is $\ell \otimes \mu$-negligible.

 In particular  $\partial_t \f \in L_{\loc}^{\infty} (X_T)$ 
 and for any continuous function 
 $h \in {\mathcal C}^0(\R,\R)$, \, $ h (\partial_t \f) \, \ell \otimes \mu$ is a well defined Borel measure on $X_T$.  
\end{lem}

%Note that the assumption $\f (t,\cdot) \in L^2 (X,\mu)$ is automatically satisfied when 
%$\mu=f dV$ is absolutely continuous with respect to Lebesgue measure, with density
%$f \in L^p$, $p>1$, as follows from H\"older inequality.

\begin{proof}
 The proof is identical to the corresponding one in the local context, we refer the reader to \cite{GLZparab1}.
\end{proof}  

%The previous lemma  shows that $\partial_t^+ \f = \partial_t^- \f $, $\ell \otimes \mu$-almost everywhere in $X_T$. 
%These thus define a function which we denote by $\dot \f \in L_{\loc}^{\infty} (X_T)$.
 When $\f$ is semi-convex or semi-concave in $t$, we can improve  this result.
 
 \begin{defi}
 We say that  $\f : X_T \longrightarrow \R $ is  uniformly semi-concave in $]0,T[$ if for any compact $J \Subset ]0,T[$, there exists  
 $\kappa = \kappa (J,\f) > 0$ such that for all $x \in X$, the function  $t \longmapsto \f (t,x) - \kappa t^ 2$ is concave in $J$. 
 \end{defi}
 
 The definition of uniformly semi-convex functions is analogous.
 Note that such functions are automatically locally uniformly Lipschitz.
 
 \begin{lem} 
 %\label{lem:sconv-Lip}  % Lemma 1.6
 Let $\f : X_T \longrightarrow \R $ be a continuous function which is uniformly semi-convex 
 %or semi-concave
 in $]0,T[$. Then 
 $$
 \partial_t^+ \f (t,x) = \lim_{s \to 0^+} \frac{\f (t + s,x) - \f (t,x)}{s} 
 $$ 
 is upper semi-continuous in $X_T$, while 
 $$
 \partial_t^- \f (t,x):= \lim_{s \to 0^-} \frac{\f (t + s,x) - \f (t,x)}{s} 
 $$
 is lower semi-continuous in $X_T$.
 In particular, $\partial_t^+ \f$ and $\partial_t^- \f$ coincide and are continuous $\ell \otimes \mu$-almost everywhere in $X_T$.
 \end{lem}
 
 \begin{proof} 
  The proof is identical to the corresponding one in the local context, we refer the reader to \cite{GLZparab1}.
 \end{proof}

\subsubsection{Topology on ${\mathcal P}(X_T,\omega)$}

 We introduce a natural complete metrizable topology on the convex set $\mathcal P (X_T,\omega)$.

We first consider a partial Sobolev space $W^{(1,0), \infty}_{\loc} (X_T)$ : this is the set of functions $u \in L^1_{\loc} (X_T)$ 
whose  partial time derivative  (in the sense of distribution) satisfies  $\dot u = \partial_t u \in L^{\infty}_{\loc} (X_T)$. 
It follows from Lemma~\ref{lem: existence of time derivative} that
  $$
  \mathcal P (X_T,\omega)  \subset W^{(1,0), \infty}_{\loc} (X_T).
  $$
  
The local uniform Lipschitz constant of $\f \in \mathcal P (X_T,\omega) $  on a  compact subset $J \Subset ]0,T[$ is given by
  $$
  {\sup_{t, s \in J, s \neq t} \sup_{x \in X}}^* \frac{\vert \f (s,x) - \f (t,x)\vert}{\vert s - t\vert} = \Vert \dot \f \Vert_{L^{\infty} (J \times X)},
  $$
  where $\sup^*$ is the essential sup with respect to a volume form $d V$ on $X$.

   We can therefore  consider the following semi-norms on  $W^{(1,0), \infty}_{\loc} (X_T)$:
given a compact subset $J \Sub ]0,T[$ and $u \in W^{(1,0), \infty}_{\loc} (X_T)$, we set
$$
 \rho_{J} (u) := \Vert \dot \f \Vert_{L^{\infty} (J \times X)}
 + \int_J \int_X \vert u (t,x) \vert d V (x) d t.
$$

\begin{pro}
The space $ W^{(1,0), \infty}_{\loc} (X_T)$ endowed with the semi-norms $( \rho_{J})$
is a complete metrizable space and  $\mathcal P (X_T,\omega) $ is a closed subset.
\end{pro}

\subsection{Parabolic complex Monge-Amp\`ere operators}  \label{sec:operator}

As explained in the introduction, we assume in this section (without loss of generality) that
$\theta \leq \omega_t \leq \Theta$, where
$\theta$ is a semi-positive and big $(1,1)$-form and $\Theta$ is a K\"ahler form.

\subsubsection{Parabolic Chern-Levine-Nirenberg inequalities}

We assume here that $\f \in \mathcal P (X_T,\omega) \cap L^{\infty}_{\loc} (X_T)$. For all $t \in ]0,T[$, the  function 
$$
X \ni x  \mapsto \f_t(x) = \f (t,x) \in \R
$$
 is $\omega_t$-psh and bounded, hence  $(\omega_t+dd^c \varphi_t)^n$ is well defined as a positive Borel measure on $X$ as follows from the works of Bedford and Taylor \cite{BT76,BT82}. 
% $$
%  {\rm MA} (\f_t) := (\omega_t + dd^c \f_t)^n
%  $$ 
%  is well defined in the sense of Bedford and Taylor \cite{BT82}.
  
Since $0 \leq \omega_t \leq \Theta$ for  $0 \leq t \leq T$, 
the  positive Borel measures $ ( \omega_t+ dd^c \f_t)^n$ have uniformly bounded masses on $X$ :
  $$
  \int_X  (\omega_t+ dd^c \f_t)^n \leq \int_X (\Theta + dd^c \f_t)^n \leq \int_X \Theta^n.
  $$
These can be considered, alternatively, as a family of currents of degree $2n$ on the real $(2n+1)$-dimensional manifold 
 $X_T= ]0,T[ \times X $.  It follows from Bedord-Taylor's convergence theorem \cite{BT76,BT82} that  $t \longmapsto  (\omega_t+dd^c \varphi_t)^n$ is continuous 
 as a map from $]0,T[$ to the space $\mathcal M (X)$ of positive Radon measures on $X$ endowed with the weak$^*$-topology. More generally we have

 \begin{lem}  \label{lem:weak-cont}  
 Fix $\f \in \mathcal P (X_T,\omega)\cap L^{\infty}_{\loc} (X_T)$ and   
 $\chi $  a continuous test function in $X_T$. The function
$
 t \longmapsto \int_{X} \chi (t,\cdot)  (\omega_t+ dd^c \f_t)^n,
$
is  continuous  in  $]0,T[$ and bounded, with
$$
\sup_{0 < t < T} \left \vert \int_{X} \chi (t,\cdot)  (\omega_t+ dd^c \varphi_t)^n \right \vert \leq   (\max_{X_T} \vert \chi \vert) \int_X \Theta^n.
$$  
 \end{lem}
 
 More generally if $\chi$ is upper semi-continuous (resp. lower semi-continuous, resp. Borel) on $X_T$, then so is the function $t \mapsto \int_X \chi(t,\cdot) (\omega_t+dd^c \varphi_t)^n$.

 \begin{proof}
 Fix a continuous test function $\chi$ on $X_T$ and fix a compact interval $J\Subset ]0,T[$ such that $J\times X$ contains the support of $\chi$. % We can reduce to the case 
% when the support of $\chi$ is contained in a product set
%  $J \times K \subset ]0,T[ \times X$, where $J \Subset ]0,T[$ and $K \Subset X$ 
%  are compact sets and 
%$\omega_t=dd^c \rho_t$, where $\rho_t \in \mathcal{C}^{\infty}(L)$ is a local potential of $\omega_t$ in a \nbd $L$ of $K$ such that $t\mapsto \rho_t$ is continuous.  

Fix $t_0\in ]0,T[$. The Lipschitz property of $\varphi$ ensures that $\varphi_t$ uniformly converges on $X$ to $\varphi_{t_0}$ as $t\to t_0$. The continuity of $t\mapsto \omega_t$ and Bedford-Taylor's convergence theorem then ensure that $(\omega_t+dd^c \varphi_t)^n$ converges to $(\omega_{t_0}+dd^c \varphi_{t_0})^n$ as $t\to t_0$. Since $\chi_t$ uniformly converges on $X$ to $\chi_{t_0}$, the first statement follows. The second statement follows from the fact that $\int_X (\omega_t+dd^c \varphi_t)^n \leq \int_X \Theta^n$, for all $t\in ]0,T[$. 
\end{proof}

 \begin{defi}   \label{def:LHS}
 Let $\f \in \mathcal P (X_T,\omega)\cap L^{\infty}_{\loc} (X_T)$. 
The map
 \begin{equation} \label{eq:current}
\chi \mapsto  \int_{X_T} \chi d t\wedge  (\omega_t+dd^c \varphi_t)^n  := \int_0^T d t \left(\int_{X} \chi (t,\cdot)  (\omega_t+dd^c \varphi_t)^n \right).
 \end{equation}
defines a $(2 n + 1)$-current on $X_T$ denoted by $ d t \wedge  (\omega_t+dd^c \varphi_t)^n$, which
can be  identified with  a positive Radon measure on $X_T$.
 \end{defi}
 
That   (\ref{eq:current})  is well defined  for continuous test  (or Borel) functions $\chi$ follows from Lemma \ref{lem:weak-cont}.
 %insures that  (\ref{eq:current}) is still valid for any Borel function in $X_T$. 
The operator can also be defined by approximation  in the spirit of Bedford and Taylor convergence results \cite{BT76,BT82} :

 \begin{pro}   
 Fix $\f \in \mathcal P (X_T,\omega)\cap L^{\infty}_{\loc} (X_T)$ and let 
 $\f_j$ be a monotone sequence of  functions $(\f_j)$ in $\mathcal P (X_T,\omega)\cap L^{\infty}_{\loc} (X_T)$  
converging to $\f$ almost everywhere in $X_T$. Then 
$$
d t \wedge  (\omega_t+ dd^c\f^j_t)^n \to d t \wedge  (\omega_t+dd^c \varphi_t)^n,
$$
in the  sense of measures on $X_T$. 
\end{pro}

\begin{proof} 
Let $\chi$ be a continuous test function in $X_T$. By definition, for all $j$, we have
$$
 \int_{X_T} \chi d t \wedge  {\rm MA} (\f^j) := \int_0^T d t\left(\int_{X} \chi (t,\cdot)  {\rm MA} ( \f^j_t) \right).
$$

We can apply Bedford and Taylor convergence theorems \cite{BT82} to conclude that, for all $t \in ]0,T[$, 
$$
\int_X \chi(t,\cdot) (\omega_t+dd^c \varphi_t^j)^n  \to   \int_X \chi(t,\cdot) (\omega_t+dd^c \varphi_t)^n.
$$ 
Since  $\int_X \chi (t,\cdot)  (\omega_t+ dd^c \f^j_t)^n$ is  a uniformly bounded   (Lemma \ref{lem:weak-cont}), the conclusion follows from Lebesgue convergence theorem. 
\end{proof}

It is classical that one can then define similarly mixed parabolic Monge-Amp\`ere operators
$$
dt \wedge (\omega_t+dd^c \f_t^1) \wedge \cdots \wedge (\omega_t+dd^c \f_t^n)
$$ 
whenever $\f^1,\ldots,\f^n \in \mathcal P (X_T,\omega)\cap L^{\infty}_{\loc} (X_T)$.
We note, for later use, the following stronger version of Chern-Levine-Nirenberg inequalities:

\begin{pro}
Assume $\f^1,\ldots,\f^n \in \mathcal P (X_T,\omega)\cap L^{\infty}_{\loc} (X_T)$ and
$\p \in \mathcal P (X_T,\omega)$. Then, for all $J \Subset ]0,T[$,
\begin{multline*}
	%\label{eq: CLN parabolique}
	\int_{J \times X}   |\psi| dt \wedge (\omega_t+dd^c \f_t^1)  \wedge \cdots \wedge (\omega_t+dd^c \f_t^n) \leq \\
	 \leq   \Vol(\Theta)  \int_J \left(|\sup_X \psi_t| + \sum_{j=1}^n{\rm osc} (\f^j_t)\right) dt 
	 +\int_{J\times X} |\psi| dt \wedge \Theta^n.
\end{multline*}
In particular, $\p \in L^1_{\loc} (X_T, dt \wedge (\omega_t+dd^c \f_t^1) \wedge \cdots \wedge (\omega_t+dd^c \f_t^n))$. 
\end{pro} 

\begin{proof}
Fix $\psi \in \mathcal{P}(X_T,\omega)$ and $J\Subset ]0,T[$. Setting $\psi_t= \tilde{\psi}_t + \sup_X \psi_t$ and using the triangle inequality we can write
\begin{multline*}
	\int_{J\times X} |\psi| dt  \wedge (\omega_t+dd^c \f_t^1) \wedge \cdots \wedge (\omega_t+dd^c \f_t^n)  \leq \\\leq \int_{J\times X} |\tilde{\psi}| dt  \wedge (\omega_t+dd^c \f_t^1) \wedge \cdots \wedge (\omega_t+dd^c \f_t^n) 
	 + \Vol(\Theta) \int_J |\sup_X \psi_t|dt. 
\end{multline*}
We can thus assume that $\sup_X \psi_t =0$ for all $t\in J$.   
A series of integration by parts as in \cite[Corollary 3.3]{GZ05} yields 
\begin{flalign*}
\int_X |\psi_t|  (\omega_t+dd^c \f_t^1) \wedge \cdots \wedge (\omega_t+dd^c \f_t^n) &\leq \sum_{j=1}^n {\rm osc}_X (\f^j_t)   \Vol(\omega_t) + \int_X |\psi_t| \omega_t^n\\
& \leq \Vol(\Theta)\sum_{j=1}^n {\rm osc}_X (\f^j_t)  +  \int_X |\psi_t| \Theta^n. 
\end{flalign*}
Integrating on $J$ yields the desired estimate.
% and $J\Subset ]0,T[$. Then for each $t \in J, x\in X$ there exist a neighborhood $V(t,x) \subset X_T$ and a constant $C(t,x)$ such that 
%$|\psi(t,x)|\leq C(t,x)$, for all $(t,x) \in V(t,x)$. The compactness of $J\times X$ ensures that there exists a constant $C>0$ such that $|\psi(t,x)| \leq C$, for all $(t,x) \in J\times X$. Now, if $\f^1,\ldots,\f^n \in \mathcal P (X_T,\omega)\cap L^{\infty}_{\loc} (X_T)$ then 
%$$
%\int_{J \times X} |\psi| dt \wedge MA(\f^1,...,\f^n) \leq C \int_J \int_X (\omega_t)^n dt\leq C \int_J \int_X \Theta^n \leq CT\, \int_X \Theta^n.  
%$$

\end{proof}

\subsubsection{Convergence results}

 \begin{defi} 
 A family $\Phi \subset \mc P (X_T,\omega)$ is  uniformly semi-concave in  $]0,T[$, if for any compact subset $J \Subset ]0,T[$, 
 there exists a constant  $\kappa = \kappa (J,\Phi) > 0$ such that any $\f \in \Phi$  is uniformly $\kappa$-concave in $J$.
 \end{defi}

 Fix $\mu$ a  Borel measure on $X$ and let $\ell$ denote the Lebesgue measure on $\R$. 
 
\begin{thm} \label{thm:conv1} 
Let $(f_j)$ be  a sequence of positive functions which 
converge to  $f$ in $L^1 (X_T,\ell \otimes \mu)$. 
 Let $(\f^j)$ be  a sequence of functions in ${\mathcal P} (X_T,\omega)$ which
 \begin{itemize}
 \item converge $\ell \otimes \mu$-almost everywhere in $X_T$ 
 to a function $\f \in  {\mathcal P} (X_T,\omega)$;
 \item   is uniformly semi-concave in $]0,T[$. 
 \end{itemize}
   
  Then $\lim_{j \to + \infty} \dot{\f}^j (t,x) = \dot{\f} (t,x)$ for $\ell \otimes \mu$-almost any $(t,x)  \in X_T$, and 
 $$
  h (\dot{\f}^j) \, f_j \, \ell \otimes \mu \to h (\dot{\f}) \, f \, \ell \otimes \mu, 
  %\, \, \, \, \mathrm{as} \, \, \, j \to + \infty,
 $$
 in the weak sense of Radon measures on $X_T$,
 for all $h \in {\mathcal C}^0(\R,\R)$.
\end{thm}

\begin{proof}
 The proof is identical to the corresponding one in the local context, we refer the reader to \cite{GLZparab1}.
\end{proof}

\section{A priori estimates} \label{sec:apriori}

In this section we assume that
$\f(t,x)=\f_t(x)$ is a  smooth $\omega_t$-psh solution to \eqref{eq:CMAF},  where 
 $t \mapsto \omega_t$ is a smooth family of K\"ahler forms, and $F,g$ are smooth with $g$ positive. 
 Our aim is to establish various a priori estimates that will allow us to construct weak solutions
 to the corresponding degenerate equations.
 
 For convenience we will also assume that $\theta$ is K\"ahler, $\varphi_0$ is smooth and strictly $\omega_0$-psh. 
It follows however from \cite{EGZ09,GZ13,Dat1} that
all the  a priori bounds below remain valid when, $\theta$ is semipositive and big,  $\varphi_0$ is merely $\omega_0$-psh and bounded.
 
We will make various extra assumptions, depending on the a priori estimates that we are interested in.

 \subsection{Controlling the oscillation of $\f_t$}  \label{subsect: osc}
 
Recall that $\theta(x) \leq \omega_t(x) \leq \Theta(x)$,
 where $\theta, \Theta$ are K\"ahler forms.
 We let $V_1$ (resp. $V_2$) denote the volume of $\{\theta\}$ (resp. $\{\Theta\}$),
 $$
 V_1=\int_X \theta^n
 \text{ and } 
 V_2=\int_X \Theta^n.
 $$
 We fix $c_1,c_2 \in \R$ normalizing constants such that $V_i=e^{c_i} \mu(X)$, where $\mu=gdV$.
 It follows from \cite{Kol98,EGZ09} that there exists $\rho_1$ a bounded $\theta$-psh function (respectively $\rho_2$ a bounded
 $\Theta$-psh function) such that 
 $$
 (\theta+dd^c \rho_1)^n=e^{c_1} \mu 
 \text{ and }
  (\Theta+dd^c \rho_2)^n=e^{c_2} \mu.
 $$
 The functions $\rho_1,\rho_2$ are moreover unique once normalized by
 $$
 \sup_X \rho_1=\inf_X \rho_2=0.
 $$
 Note that in proving existence of solutions the form $\theta$ will no longer be K\"ahler but merely semipositive and big. But the $L^{\infty}$ bound on $\rho_1$ remains uniform thanks to \cite[Proposition 2.6]{EGZ09}. 
 
 \begin{pro} \label{pro:bdd} % Proposition 1.15
 The following uniform a priori bound on  $\f_t$ holds :
 $$
 |\varphi_t(x)|  \leq C_0:= C \left ( e^{\lambda_F T} + \frac{e^{\lambda_F T}-1}{\lambda_F} \right),
 $$
 where $C$ is the following uniform constant 
 $$
 C= \sup_{X_T} |F(t,x,0)| + (\lambda_F+1)  \sup_X( |\rho_1| + |\rho_2| ) + \sup_X |\varphi_0| + \max(-c_1,c_2). 
 $$
% \text{ and }
% $
% C_2=\max \left\{ -\inf_{X_T} F(t,x,\sup_X \f_0)-c_2; 0 \right\}.
% $
 \end{pro}
 
 Recall that $\lambda_F\geq 0$ is a constant such that, for all $(t,x)\in X_T$,  the function $r\mapsto F(t,x,r)+\lambda_F r$ is increasing on $\mathbb{R}$.
 
 \begin{proof}
 Set, for $t\in \mathbb{R}$,
 $$
 \gamma(t):= \sup_X |\varphi_0| e^{\lambda_F t} + \frac{C(e^{\lambda_F t }-1)}{\lambda_F},
 $$
  where $C$ is as in the statement of the proposition.
 A direct computation shows that $\gamma(0) =\sup_X |\varphi_0|$ and $\gamma'(t) -\lambda_F \gamma(t) = C$. 
 
 Set, for $(t,x) \in X_T$, $u(t,x): =\rho_1(x)-\gamma(t)$.
 Observe that $u_t$ is $\theta$-psh hence $\omega_t$-psh, that $u_0 \leq \f_0$
 and 
 \begin{eqnarray*}
 (\omega_t+dd^c u_t)^n  \geq  (\theta+dd^c \rho_1)^n=e^{c_1} \mu \geq  e^{\dot{u}_t +F(t,x,u_t)} \mu.
 \end{eqnarray*}
 The last inequality follows from our choice of $C$: since 
 $r \mapsto F(\cdot,\cdot,r)+\lambda_F r$ is increasing and $\rho_1 \leq 0, u_t\leq 0$, we obtain
 \begin{eqnarray*}
 	F(t,x,u_t(x))+ \dot{u}_t &=& F(t,x,u_t(x)) -\gamma'(t)\\
  	&\leq & F(t,x,0) -\lambda_F (\rho_1 -\gamma(t)) -\gamma'(t)  \\
  	&\leq & F(t,x,0) + \lambda_F |\rho_1| -C  \\
 	&\leq & c_1.  
 \end{eqnarray*}
 It thus follows from the maximum principle that $\varphi \geq u$ on $X_T$.  
 
  Set now $v(t,x)=\rho_2(x)+\gamma(t)$, $(t,x) \in X_T$. 
 We let the reader check similarly  that $v_t$ is $\Theta$-psh, it satisfies
 $$
 (\Theta+dd^c v_t)^n \leq  e^{\partial_t v(t,x) + F(t,x,v_t)} gdV,
 $$
 and $v_0\geq \varphi_0$.  Now $\varphi_t$ is a subsolution
  to this new parabolic equation since $\omega_t\leq \Theta$. It  follows therefore from the maximum principle that $\varphi \leq v$ on $X_T$, and the desired estimates follow.
 \end{proof}
 
 The following construction of subbarrier will be useful in showing that the 
 pluripotential solution to \eqref{eq: MAF as measures} has the right value at $t=0$. 
 
 \begin{pro}
 	\label{pr:sbarrier}
 	%Set $t_0 := \min(1,T/2)$. 
 	For all  $0 \leq t \leq 1$,
 	$$
 	\varphi_t \geq (1-t) e^{-At} \varphi_0 + t \rho_1 + n(t\log t -t) - C\frac{e^{\lambda_F t}-1}{\lambda_F},
 	$$
 	where $C$ is the following uniform constant,
 	$$
 	C:= \sup_{X_T} F(t,x,0) + (A+\lambda_F+1) \left (\sup_X |\varphi_0|  + \sup_X |\rho_1| + n\right) -c_1.
 	$$ 
 %	where  
 	%$A$ is a positive constant such that $\dot{\omega}_t \geq -A\omega_t$, for all $(t,x) \in ]0,T[ \times \Omega$,  and 
 	%$$
 	%C= -nc_1 + (A+1) \sup_X |\varphi_0| + \inf_X \rho_1 + \sup_{(t,x)\in [0,1] \times X} F(t,x, \sup_X |\varphi_0|).
 %	$$
 \end{pro}
 
% The constant $C_2$ depends explicitly on the data as the proof will show.
 
 \begin{proof}
 Recall that $A$ denotes a positive constant such that $\dot{\omega}_t \geq -A\omega_t$ for all $t \in ]0,T[$.
In particular $\omega_t \geq e^{-At} \omega_0$ and $\omega_t \geq \theta$.  Recall also that $\lambda_F\geq 0$ is a constant so that $r\mapsto F(t,x,r)+\lambda_F r$ is increasing in $\mathbb{R}$ for all $(t,x) \in X_T$. 

Consider the function
 	$$
 	u_t(x) := (1-t) e^{-At} \varphi_0 + t \rho_1 + n(t\log t -t) - C \frac{e^{\lambda_F t} -1}{\lambda_F}, 
 	$$
 	where $C$ is the uniform constant defined in the proposition.
 
 Using that  $\omega_t \geq e^{-At}\omega_0$ and $\omega_t \geq \theta$, we have 
 	 \begin{flalign*}
 	 	(\omega_t +dd^c u_t)^n & = \left ( (1-t)\omega_t + (1-t) e^{-At} dd^c \varphi_0 + 
 	 	 t(\omega_t+dd^c \rho_1)  \right )^n \\
 	 	& \geq   t^n(\omega_t +dd^c \rho_1)^n\\
 	 	&  \geq t^n e^{c_1}gdV. 
 	 \end{flalign*}
 	 Since $u_t\leq  0$ and $r\mapsto F(t,x,r)+\lambda_F r$ is increasing, a direct computation yields
 	   \begin{flalign*}
 	 	\dot{u}_t + F(t,x,u_t) & = n\log t + \rho_1 + e^{-At}(A(1-t)+1) (-\varphi_0) -C e^{\lambda_F t} \\
 	 	&+ F(t,x,u_t)+\lambda_F u_t -\lambda_F u_t \\
 	 	&\leq n\log t + (A+1) \sup_X |\varphi_0| + \sup_{X_T} F(t,x,0)\\
 	 	&  + \lambda_F \left (\sup_X |\varphi_0| + \sup_X |\rho_1| +n \right ) - C  \\
 	 	&\leq n\log t + c_1. 
 	 \end{flalign*}
 	It thus follows that $u_t$ is a subsolution to \eqref{eq:CMAF} with $u_0\leq \varphi_0$. The desired estimate follows from the classical maximum principle. 
% 	 This is achieved if  for all $(t,x) \in ]0,t_0[ \times  X$,
% 	 $$
% 	 n\log t + nc_1  \geq \frac{\partial}{\partial t} (1-t) e^{-At} \varphi_0 + \rho_1 + F(t,x,u_t(x)) -C.
% 	 $$
 	% Since $F$ is increasing in $r$ and $u_t \leq \sup_X \varphi_0$ it suffices to have that 
 	 %$$C+ nc_1 \geq (A+1) \sup_X |\varphi_0| + \inf_X \rho_1 + \sup_{[0,1] \times X} F(t,x, \sup_X |\varphi_0|).$$
% 	 The choice of $C$ follows, as $F, u_t$ and $\rho_1$ are bounded 
 \end{proof}

 \subsection{Controlling the average}
 
 We establish the following control on the average of $\varphi_t$ which will be useful in proving convergence at zero. 
 
 \begin{pro}
 	\label{prop: average}
 	Set $\mu=gdV$.
 	The following  bound holds
 	\[
 	\int_X \varphi_t \, d\mu \leq \int_X \varphi_0 \, d\mu  + Ct, 
 	\]
 	where $C_3$ is the following uniform constant 
 	$$
 	C:= -\mu(X) \log (\mu(X)/ V_2) -\inf_{X_T\times [-C_0,C_0]} F(t,x,r) \mu(X),
 	$$
 	and $C_0$ is the uniform constant defined in Proposition \ref{pro:bdd}.
  \end{pro}
  
 \begin{proof}
Set $-C':= \inf_{X_T \times [-C_0,C_0]} F(t,x,r) >-\infty$. It follows from the flow equation  that 
\[
\int_X  e^{\dot{\varphi}_t-C'} d\mu  \leq \int_X \omega_t^n \leq V_2. 
\] 	
On the other hand it  follows from Jensen's inequality that 
\[
\int_X e^{\dot{\varphi}_t} \frac{d\mu }{\mu(X)}  \geq \exp \left ( \int_X \dot{\varphi}_t \frac{d\mu }{\mu(X)}\right).
\]
Combining these two estimates we arrive at 
\[
 \int_X \dot{\varphi}_t gdV \leq C:= C' \mu(X) +  \mu(X) \log  V_2 - \mu(X) \log  \mu(X). 
\]
The function $t\mapsto \int_X \varphi_t d\mu  -Ct$ is therefore non-increasing, hence 
\[
\int_X \varphi_t d\mu \leq  \int_X \varphi_0 d\mu+ Ct.
\]

 \end{proof}

\subsection{Lipschitz control in time} \label{sec:lipapriori}
 
 We now establish an a priori bound which will allow us to show that
 the solutions $\f_t$ to degenerate complex Monge-Amp\`ere flows
 are  locally uniformly Lipschitz in time, away from zero.
 
 %The a priori bound  requires a technical assumption on (the cohomological classes of) the forms $\omega_t$.
  For the convenience of the reader we first state and prove our theorem in the simpler case
 when $t \mapsto \omega_t$ is affine and $r\mapsto F(t,x,r)$ is increasing. A more technical statement follows, together with its proof.

 \subsubsection{Affine dependence on time}

\begin{thm} \label{thm: apriori-est C1 particular} % Theorem 2.1  
Assume $t \mapsto \omega_t=\omega_0+t \chi$ is affine and 
$r \mapsto F(\cdot,\cdot,r)$ is increasing. Then for all $(t,x) \in X_T$,
$$
n \log t -C \leq \dot{\f}_t(x) \leq \frac{C}{t},
$$
where $C$ depends explicitly on $T,  ||\partial F/\partial r||_{L^{\infty}}, ||\partial F/\partial t||_{L^{\infty}}$, $\|g\|_p$, and  $C_0$. %and ${\rm Osc}_X \f_0$. the constant $C_0$ already depends on $\sup_X |\varphi_0|$
\end{thm}

Here and below $C_0$ denotes the constant from Proposition \ref{pro:bdd}, and the Lipschitz constants 
$||\partial F/\partial r||_{L^{\infty}}, ||\partial F/\partial t||_{L^{\infty}}$
are computed on  $ X_T   \times [-C_0,C_0]$. 

\begin{proof}
For notational convenience we set $\mu =gdV$. We first establish the bound from above.
Consider
$$
H(t,x)=t \dot{\f}_t(x)-(\f_t-\f_0)-Bt,
$$
where
$$
B=n+1-\inf_{ X_T   \times [-C_0,C_0]} \left[ t \frac{\partial F}{\partial t}(t,x,r) \right].
$$

Set $S_t :=\omega_t+dd^c \f_t$ and observe that
$$
\frac{\partial H}{\partial t} =t \ddot{\f}_t -B,
$$
with
$$
\dot{\f}_t=\log \left( S_t^n/ \mu \right)-F(t,x,\f_t)
$$
hence
$$
\ddot{\f}_t=\Delta_{S_t} (\dot{\f}_t) 
+{\rm tr}_{S_t}(\dot{\omega}_t)
-\frac{\partial F}{\partial t}(x,t,\f_t)-\dot{\f}_t \frac{\partial F}{\partial r}(x,t,\f_t),
$$
where
$$
\Delta_{S_t} f:=n \frac{dd^c f \wedge S_t^{n-1}}{S_t^n}
\; \text{ and } \; 
{\rm tr}_{S_t}(\eta):=n \frac{\eta \wedge S_t^{n-1}}{S_t^n}.
$$

On the other hand
$$
\Delta_{S_t}(H)=t \Delta_{S_t}(\dot{\f}_t)-n+ {\rm tr}_{S_t}(\omega_t+dd^c \f_0),
$$
therefore
{\small
$$
\left( \frac{\partial}{\partial t} -{\Delta}_{S_t} \right)(H)=\left\{ 
-t\frac{\partial F}{\partial t}-t\dot{\f}_t \frac{\partial F}{\partial r}+n-B \right\}
-{\rm tr}_{S_t}(S_0)+{\rm tr}_{S_t}(\omega_0+t \dot{\omega}_t-\omega_t).
$$
}

The assumption that $t \mapsto \omega_t$ is affine insures ${\rm tr}_{S_t}(\omega_0+t \dot{\omega}_t-\omega_t)=0$,
while our choice of $B$ yields
$$
\left( \frac{\partial}{\partial t} -{\Delta}_{S_t} \right)(H) \leq -1-t\dot{\f}_t \frac{\partial F}{\partial r}(t,x,\f_t).
$$

If $H$ realizes its maximum $H_{max}$ along $(t=0)$, we obtain
$$
H(t,x) \leq H_{max}=\sup_{x \in X} H(0,x)=0
$$
which yields the desired upper-bound for $\dot{\f}_t$.

If $H$ realizes its maximum $H_{max}$ at some point $(t_0,x_0)$ with $t_0 >0$, then
$$
0 \leq \left( \frac{\partial}{\partial t} -{\Delta}_{S_t} \right)(H)(t_0,x_0) 
$$
hence
$$
t_0\dot{\f}_{t_0}(x_0) \frac{\partial F}{\partial r}(t_0,x_0,\f_{t_0}(x_0)) \leq -1<0.
$$
Since $\frac{\partial F}{\partial r} \geq 0$ and $t_0 >0$, we infer 
$
\dot{\f}_{t_0} (x_0) <0.
$
hence
$$
H_{max} \leq -(\f_{t_0}-\f_0)(x_0) \leq C,
$$
where the last inequality follows from Proposition \ref{pro:bdd}.
This yields again the desired upper-bound.

\medskip

We now take care of the lower bound.
We first deal with the particular case when
$$
\mu(x)=h(x) \theta^n(x),
$$
where $h \geq 0$ is a bounded density. Fix $D>>1$ so large that all our quantities 
are well defined and under control on $[0,T+1/D] \times X$. Observe that
$$
\chi+D\omega_t=D \omega_{t+1/D} \geq D \theta.
$$

We set 
$$
G(t,x)=\dot{\f}_t(x)+D \f_t(x)-n \log t,
$$ 
and compute
\begin{eqnarray*}
\left( \frac{\partial}{\partial t} -{\Delta}_{S_t} \right)(G)
&=&{\rm tr}_{S_t}(\chi+D \omega_t)-Dn -\frac{\partial F}{\partial t} +\left[D -\frac{\partial F}{\partial r} \right]\dot{\f}_t -\frac{n}{t} \\
&\geq & \frac{f_t^{-1/n}}{C_1}-C_2-\frac{n}{t},
\end{eqnarray*}
where $f_t=e^{\dot{\f}_t}$. We have used here
\begin{eqnarray*}
{\rm tr}_{S_t}(\chi+D \omega_t) &\geq&  D {\rm tr}_{S_t}(\theta) \geq nD \left( \frac{\theta^n}{S_t^n} \right)^{1/n} \\
&=& nD f_t^{-1/n} \left( \frac{\theta^n(x)}{e^{F(t,x,\f_t)} \mu(x)} \right)^{1/n} \geq \frac{f_t^{-1/n}}{C_1'},
\end{eqnarray*}
(the last inequality uses our assumption that $\mu=h \theta^n$ with $h$ bounded), the
fact that $\frac{\partial F}{\partial t} \leq c$, $\frac{\partial F}{\partial r} \leq c'$, and
the inequality
$$
\dot{\f}_t=\log f_t \geq -n \e f_t^{-1/n}-n C_{\e},
$$
valid for $\e>0$ arbitrarily small (since $\e x >\log x-C_\e$ for $x>0$).
 
The function $G$ attains its minimum on $]0,T]\times X$ at a point $(t_0,x_0)$ with $t_0>0$. At this point we
therefore have a control on the density $f_t$, namely
$$
f_{t_0}^{-1/n}(x_0)  \leq C_1C_2+ \frac{nC_2}{t_0}
$$
hence
$$
\dot{\f}_{t_0}(x_0)=\log f_{t_0}(x_0) \geq -n \log \left[ C_1C_2+ \frac{nC_2}{t_0} \right].
$$
We infer
$$
G(t_0,x_0) \geq D \f_{t_0}(x_0) -n \log \left[ C_1C_2t_0+ nC_2 \right] \geq -C_3,
$$
using Proposition \ref{pro:bdd}. The desired lower bound follows.
  
  We now get rid of the extra assumption made on $\mu$.
  We fix as earlier $\rho_1$ a smooth $\theta$-psh function such that 
  $$
  (\theta+dd^c \rho_1)^n=e^{c_1} \mu
  $$
  and $\sup_X \rho_1=0$. 
  %We can assume also that $\theta$ is K\"ahler (and hence $\rho_1$ is smooth thanks to Yau's celebrated result \cite{Yau78}) but   $\sup_X |\rho_1|$  does not depend on the lower bound of $\theta$ as follows from \cite[Proposition 2.6]{EGZ09}. 
  We set
  $$
  \tilde{\omega}_t=\omega_t+dd^c \rho_1, \;
   \;
  \tilde{\f}_t:=\f_t-\rho_1 \in {\rm PSH}(X,\tilde{\omega}_t)
  $$
  and
  $$
  \tilde{F}(t,x,r)=F(t,x,r+\rho_1(x)).
  $$
  Observe that $\dot{\tilde{\f}}_t=\dot{\f}_t$, $\partial_t \tilde{\omega}_t=\partial_t {\omega}_t$ and
  $$
   \tilde{\omega}_t \geq \tilde{\theta}=\theta+dd^c \rho_1.
  $$
  Moreover $\tilde{F}$ has the same Lipschitz constant (in $t$ and $r$) as  that of $F$ and
  $$
  (\tilde{\omega}_t+dd^c \tilde{\f}_t)^n=e^{\dot{\tilde{\f}}_t+\tilde{F}(t,x,\tilde{\f}_t)} \mu(x)
  $$
  with $\tilde{\theta}^n=e^{c_1} \mu$, hence $\mu=\tilde{h} \tilde{\theta}^n$ with bounded density $\tilde{h}=e^{-c_1}$. 
  We can thus use the same reasoning as above to conclude.
\end{proof}

\subsubsection{Refining the hypotheses}

  We now establish similar uniform bounds on $\dot{\f}_t$ under less restrictive assumptions on 
  $t \mapsto \omega_t$ and $r \mapsto F(\cdot, \cdot,r)$. Recall that $r\mapsto F(t,x,r)+ \lambda_F r$ is increasing on $\mathbb{R}$. 

\begin{thm}  \label{thm: apriori-est C1 general}
Assume $\dot{\omega}_t \geq -A \omega_t$.
%$t\in ]0,T[$, 
Then for all $(t,x) \in X_T$,
$$
n \log t -C \leq \dot{\f}_t(x).
$$

If there exists $A\geq 0$ with $\dot{\omega}_t \leq A \omega_t$, then for all $(t,x) \in X_T$,
$$
\dot{\f}_t(x) \leq \frac{C}{t}.
$$
\end{thm}

Here again $C$ depends explicitly on $T,  ||\partial F/\partial r||_{L^{\infty}}, ||\partial F/\partial t||_{L^{\infty}}$, $||g||_{L^p}$ 
and  $C_0$ (defined in Proposition \ref{pro:bdd}). The norms $ ||\partial F/\partial r||_{L^{\infty}}$, $||\partial F/\partial t||_{L^{\infty}}$ are computed on $[0,T[ \times X \times [-C_0,C_0]$. 

\begin{proof}
%We establish the following slightly more general  upper bound : 
%let $h:\R^+ \rightarrow \R^+$ be a ${\mathcal C}^2$-smooth function such that $h(0)=0$, $h(t)>0$ for $t>0$,
%and  $t \mapsto \frac{\omega_t}{h(t)}$ is decreasing. 
%We are going to show that for all $(t,x) \in X_T$
%$$
%\dot{\f}_t(x) \leq \frac{C}{h(t)}.
%$$
%
%The previous affine setting is recovered by taking $h(t)=t$, while the statement of the lemma follows with $h(t)=e^{At}-1$.
Consider
$$
H(t,x)=t \dot{\varphi}_t(x) -C \varphi_t,
$$
where $C:= (A+\lambda_F)T+2$. We  let the reader check that
$$
\left( \frac{\partial}{\partial t} -{\Delta}_{S_t} \right)(H) \leq C' - \left [ t \partial_r F + C-1 \right] \dot{\f}_t. 
$$
The upper bound then follows just as  in the proof of Theorem \ref{thm: apriori-est C1 particular}.
%if we choose $B$ large enough so that
%$$
%-h'' \f_t-B-nh'-h \frac{\partial F}{\partial t} \leq -1.
%$$
%The hypothesis on $t \mapsto \omega_t/h(t)$ is used to get rid of the trace term,
%$$
%{\rm tr}_{S_t}(h \dot{\omega}_t-h'\omega_t) \leq 0
%\; \text{ since  } \; 
%h \dot{\omega}_t-h'\omega_t \leq 0. 
%$$
%The rest of the proof is identical to
%that of Theorem \ref{thm: apriori-est C1 particular}.

\smallskip

We now establish the lower bound.
% by following the same line as in the proof of Theorem \ref{thm: apriori-est C1 particular} with one small modification. 
Consider, for $(t,x) \in  ]0,T] \times  X$,
$$
G(t,x) := \dot{\varphi}_t + A(2\varphi_t  - \rho_1) - n\log t,
$$
where $\rho_1\in \PSH(X,\theta)$ is the unique normalized solution to 
$$
(\theta+dd^c \rho_1)^n=e^{c_1}gdV.
$$ 
Using the same notations as in the proof of Theorem \ref{thm: apriori-est C1 particular} we obtain
\begin{flalign*}
	\Delta_{S_t} G & = \Delta_{S_t} \dot{\varphi}_t + A \left ( \Tr_{S_t}(2\omega_t+2dd^c \varphi_t - \omega_t - (\omega_t+dd^c \rho_1) \right )\\
	& \leq \Delta_{S_t} \dot{\varphi}_t + 2n A- A \Tr_{S_t}(\omega_t)  - \Tr_{S_t}(\theta+dd^c \rho_1)\\
	& \leq \Delta_{S_t} \dot{\varphi}_t + 2n A + \Tr_{S_t} (\dot{\omega}_t) - n e^{(c_1-\dot{\varphi}_t -F(t,\cdot, \varphi_t))/n}\\
	& \leq \Tr_{S_t} (\dot{\omega}_t +dd^c \dot{\varphi}_t) + 2n A  - \frac{ f_t^{-1/n}}{C_1}. 
\end{flalign*}
It thus follows that 
\begin{flalign*}
&\left( \frac{\partial}{\partial t} -{\Delta}_{S_t} \right)(G)
= \ddot{\varphi}_t  -2A \dot{\varphi}_t  - \frac{n}{t} -  {\Delta}_{S_t} G \\
&\geq  - \frac{\partial F}{\partial t}  -  \left (2 A + \frac{\partial F}{\partial r} \right) \dot{\varphi}_t   +  \frac{f_t^{-1/n}}{C_1} - \frac{n}{t} - 2n A.
\end{flalign*}
%
%The crucial term ${\rm tr}_{S_t}(\chi+A \omega_t)$ becomes 
%$$
%{\rm tr}_{S_t}(\dot{\omega}_t+(A+1) \omega_t)  \geq  {\rm tr}_{S_t}(\theta) \geq \frac{f_t^{1/n}}{C_1''}
%$$
We can then conclude  as in the proof of Theorem \ref{thm: apriori-est C1 particular}.
\end{proof} 

\begin{rem}
	The lower bound for $\dot{\varphi}_t$ ensures that
	$$
	\varphi_t \geq \varphi_0 + n(t\log t -t) -Ct,
	$$
	which is a similar lower bound than the one provided by  Proposition \ref{pr:sbarrier}.
\end{rem}

\subsection{Semi-concavity in time} 
 
Our goal in this section is to establish that
$\omega_t$-psh solutions to \eqref{eq:CMAF} are $\kappa$-concave in time away from zero, with a
 uniform a priori constant $\kappa$.

\subsubsection{A particular case}

\begin{thm} \label{thm: apriori-est C2 particular} % Theorem 3.1
Assume that $t \mapsto \omega_t$ is affine and $r \mapsto F(\cdot,\cdot,r)$ is convex, increasing.
 Let $\f_t$ be a smooth solution to \eqref{eq:CMAF}. Then
 there exists $C>0$ such that
$$
\ddot{\f}_t(x) \leq \frac{C}{t}
\; \;
\text{ for all }
(t,x) \in X_T,
$$
where $C$ depends explicitly on 
$T,  ||\partial F/\partial r||_{L^{\infty}}, ||\partial F/\partial t||_{L^{\infty}},||\partial^2 F/\partial r \partial t||_{L^{\infty}}$, 
$||\partial^2 F/\partial t^2||_{L^{\infty}}$, $\|g\|_p$
and  $C_0$.
\end{thm}
Recall that $C_0$ is an upper bound for $|\varphi_t|$ established in Proposition  \ref{pro:bdd}, and the norms on the partial derivatives of $F$ are computed on $X_T\times [-C_0,C_0]$. 

We will establish a similar (though less precise) control under less restrictive assumptions on $t \mapsto \omega_t$ and $F$.
We postpone this to the next subsection, as the a priori estimates are already quite involved.

\begin{proof}
Set $\omega_t=\omega+t \chi$ so that $\dot{\omega}_t=\chi$.
 Writing
 $$ 
\dot{\f}_t=\log \left[ (\omega_t +dd^c \f_t)^n/ g(x)dV(x) \right] - F (t,x,\f),
 $$
 we differentiate in time to obtain
 $$
 \ddot{\f}_t=\Delta_{S_t} (\dot{\f}_t) 
+{\rm tr}_{S_t}(\dot{\omega}_t)
-\frac{\partial F}{\partial t}(t,x,\f_t)-\dot{\f}_t \frac{\partial F}{\partial r}(t,x,\f_t),
$$
where $S_t=\omega_t+dd^c \f_t$ ,
$$
\Delta_{S_t} f:=n \frac{dd^c f \wedge S_t^{n-1}}{S_t^n}
\; \text{ and } \; 
{\rm tr}_{S_t}(\eta):=n \frac{\eta \wedge S_t^{n-1}}{S_t^n}.
$$

It follows from the Lipschitz a priori estimate (Theorem \ref{thm: apriori-est C1 particular}) that
$$
-C \leq t\dot{\f}_t \frac{\partial F}{\partial r}(t,x,\f_t) \leq C
$$ 
is uniformly bounded on $X_T$, hence
$$
t   \ddot{\f}_t=t \, {\rm tr}_{S_t}(\chi+dd^c \dot{\f}_t)+O(1).
$$

Differentiating again yields
{\small
 \begin{eqnarray*}
 \lefteqn{
 \dddot{\f}_t=\Delta_{S_t} (\ddot{\f}_t) 
-n^2\left(\frac{(\chi+dd^c \dot{\f}_t) \wedge S_t^{n-1}}{S_t^n} \right)^2
+n(n-1) \frac{(\chi+dd^c \dot{\f}_t)^2 \wedge S_t^{n-2}}{S_t^n}
} \\
&&
-\frac{\partial^2 F}{\partial t^2}(x,t,\f_t)
-2\dot{\f}_t \frac{\partial^2 F}{\partial r \partial t}(x,t,\f_t)
-\ddot{\f}_t \frac{\partial F}{\partial r}(x,t,\f_t)
-(\dot{\f}_t)^2 \frac{\partial^2 F}{\partial r^2}(x,t,\f_t).
\end{eqnarray*}
}

Set $H(t,x)=t \ddot{\f}_t-Bt$ where $B>0$. It follows from
the Lipschitz control $t |\dot{\f}_t| \leq C$ and Lemma \ref{lem:inegdiff} below that
$$
\left( \frac{\partial }{\partial t}-\Delta_{S_t} \right)H
\leq \left[1-t\frac{\partial F}{\partial r} \right]\ddot{\f}_t
-nt\left(\frac{(\chi+dd^c \dot{\f}_t) \wedge S_t^{n-1}}{S_t^n} \right)^2
$$
if we choose $B>0$ so large that
$$
t\frac{\partial^2 F}{\partial t^2}(t,x,\f_t)
+2t\dot{\f}_t \frac{\partial^2 F}{\partial r \partial t}(t,x,\f_t)+B \geq 0.
$$
We  use here the simplifying assumption that $r \mapsto F(\cdot,\cdot,r)$ is convex
so that $-(\dot{\f}_t)^2 \frac{\partial^2 F}{\partial r^2}(t,x,\f_t) \leq 0$.
We will remove this assumption in the next subsection.

Let $(t_0,x_0) \in X_T$ be a point at which $H$ realizes its maximum.  
If $t_0=0$ then $H \leq 0$ hence $\ddot{\f}_t \leq B$ and we are done. 
If $t_0>0$, then 
$
0 \leq \left( \frac{\partial }{\partial t}-\Delta_{S_t} \right)H
$
at the point $(t_0,x_0)$ thus for $(t,x)=(t_0,x_0)$, 
$$
\frac{1}{n} \left(t \, {\rm tr}_{S_t}(\chi+dd^c \dot{\f}_t)\right)^2 
\leq 
\left[1-t\frac{\partial F}{\partial r} \right] t \ddot{\f}_t
$$
with
\begin{eqnarray*}
t\ddot{\f}_t&=& t \, {\rm tr}_{S_t}(\chi+dd^c \dot{\f}_t)
-t\frac{\partial F}{\partial t}(t,x,\f_t)-t\dot{\f}_t \frac{\partial F}{\partial r}(t,x,\f_t) \\
&=& t \, {\rm tr}_{S_t}(\chi+dd^c \dot{\f}_t)+O(1).
\end{eqnarray*}

It follows that $t_0 \ddot{\f}_{t_0}(x_0)$ is uniformly bounded from above,
hence so is $H \leq C$. Thus $t \ddot{\f}_t \leq Bt+C \leq C'$ on $X_T$.
\end{proof}

We have used the following differential inequality which is probably well known. We include a proof for the reader's convenience.

\begin{lem} \label{lem:inegdiff}
Assume $n \geq 2$.
 Let $\omega$ be a K\"ahler form and let $\eta$ be a closed $(1,1)$-differential form.
 Then
 $$
 \frac{\eta^2 \wedge \omega^{n-2}}{\omega^n} \leq 
 \left(  \frac{\eta \wedge \omega^{n-1}}{\omega^n} \right)^2.
 $$
\end{lem}

\begin{proof}
 This is a pointwise inequality, hence it reduces to linear algebra.
 Since $\omega$ is a K\"ahler form, we can assume that $\omega(x)$ is
 the euclidean K\"ahler metric. Perturbing $\eta(x)$ if necessary, we can also
 make a change of local coordinates so that $\eta(x)$ is given by a diagonal
 matrix with diagonal entries $\lambda_1,\ldots,\lambda_n$.
We infer
 $$
  \frac{\eta^2 \wedge \omega^{n-2}}{\omega^n}(x)=\frac{1}{n(n-1)}\sum_{\alpha \neq \beta} \lambda_{\alpha} \lambda_{\beta}
 $$
 while 
 $$
  \left(  \frac{\eta \wedge \omega^{n-1}}{\omega^n} \right)^2(x)=
  \left(\frac{1}{n} \sum_{\alpha} \lambda_{\alpha} \right)^2. 
 $$
 The desired inequality follows from the following elementary computation
 $$
 \left (\frac{1}{n}\sum_{\alpha=1}^n \lambda_{\alpha} \right )^2 - \frac{1}{n(n-1)} \sum_{\alpha\neq \beta} \lambda_{\alpha}\lambda_{\beta} = \frac{1}{n^2(n-1)} \sum_{\alpha<\beta} (\lambda_{\alpha}-\lambda_{\beta})^2\geq 0. 
 $$
\end{proof}

\subsubsection{More general bounds}

 We assume in this subsection that there exists a constant $A>0$ such that
 $\forall t\in ]0,T[$
 \begin{equation}
 	\label{eq:omegatconcave}
 	 -A \omega_t \leq 	\dot{\omega}_t \leq +A \omega_t
 	\; \; \ \text{and}\ 	\; \; 
\ddot{\omega}_t \leq A \omega_t .
 \end{equation}

 We also assume that $(t,r) \mapsto F(t,x,r)$  is uniformly  semi-convex, i.e.  there exists a  constant $C_F>0$ 
 such that for every $x\in X$, the function 
\begin{equation}
	\label{eq:Fconv}
(t,r) \mapsto  F(t,x,r)+ C_F (t^2+r^2) \ \text{is convex in} \ [0,T[ \times [-C_0,C_0].
\end{equation}

\begin{thm} \label{thm:apriori-est} % Theorem 3.1
Assume that $\omega_t$ and $F(\cdot,\cdot,r)$ are as above.
 Let $\f_t$ be a solution of the above parabolic Monge-Amp\`ere equation. Then
 there exists $C>0$ such that
$$
\ddot{\f}_t(x) \leq \frac{C}{t^2}
\; \;
\text{ for all }
(t,x) \in X_T,
$$
where $C$ depends explicitly on 
$A, T, C_F,\lambda_F$,  $||\partial F/\partial r||_{L^{\infty}}, ||\partial F/\partial t||_{L^{\infty}}$, 
$\|g\|_p$, and $C_0$.
\end{thm}

Here $C_0$ is the constant given in Proposition \ref{pro:bdd} and the norms $||\partial F/\partial r||_{L^{\infty}}$, 
$||\partial F/\partial t||_{L^{\infty}}$ are computed on $X_T\times [-C_0,C_0]$.

\begin{proof}
In the proof below we use $C$ to denote various constants under control. 
Set $\alpha_t:= \dot{\omega}_t +dd^c \dot{\varphi}_t$, $S_t:= \omega_t +dd^c \varphi_t$, and for $h \in \mathcal{C}^{\infty}(X,\mathbb{R})$, 
$$\Delta_t h: = \Tr_t (dd^c h):= \Tr_{S_t} (dd^c h) = n \frac{dd^c h \wedge S_t^{n-1}}{S_t^n}.$$ 
 
 Writing
 $$ 
\dot{\f}_t=\log \left[ (\omega_t +dd^c \f_t)^n/ g(x)dV(x) \right] - F (t,x,\f),
 $$
 we differentiate twice in time to obtain, as in  the proof of Theorem \ref{thm: apriori-est C2 particular}, that
\begin{equation}
	\label{eq:ddot0}
	t   \ddot{\f}_t=t \, \Tr_t \alpha_t -t \partial_t F -t \dot{\varphi}_t\partial_r F = t \, \Tr_t \alpha_t+ O(1), 
\end{equation}
where we use the uniform bound $t |\dot{\f}_t| \leq C$ (thanks to Theorem \ref{thm: apriori-est C1 general}), and
\begin{flalign*}
	 \dddot{\f}_t & =\Tr_t (\dot{\alpha}_t)  + n(n-1) \frac{\alpha_t^2 \wedge S_t^{n-2}}{ S_t^n} - n^2\left( \frac{\alpha_t\wedge S_t^{n-1}}{S_t^n} \right )^2 - \ddot{\f}_t \frac{\partial F}{\partial r}(t,x,\f_t)  \\
	 & -\frac{\partial^2 F}{\partial t^2}(t,x,\f_t)
-2\dot{\f}_t \frac{\partial^2 F}{\partial r \partial t}(t,x,\f_t)
-(\dot{\f}_t)^2 \frac{\partial^2 F}{\partial r^2}(t,x,\f_t)\\
& \leq \Tr_t (\dot{\alpha}_t)  - \frac{1}{n} (\Tr_{t} \alpha_t)^2  -\ddot{\f}_t \frac{\partial F}{\partial r}(t,x,\f_t)  + C_F((\dot{\varphi}_t)^2 +1),
\end{flalign*}
using the convexity condition \eqref{eq:Fconv}  and Lemma  \ref{lem:inegdiff}.
The Lipschitz control $t |\dot{\f}_t| \leq C$ (provided by Theorem \ref{thm: apriori-est C1 general}) yields
\begin{equation}
	\label{eq:ddot1}
	 t^2\dddot{\f}_t  \leq  t^2\Tr_t \dot{\alpha}_t -t^2n^{-1} [\Tr_t (\alpha_t)]^2 - t^2 \ddot{\f}_t 
	 \frac{\partial F}{\partial r}(t,x,\f_t) +C. 
\end{equation}

Set $H(t,x)=t^2\ddot{\f}_t-ATt\varphi_t$. It follows from \eqref{eq:omegatconcave} and a direct computation that 
\begin{equation}
	\label{eq:ddot2}
	\Delta_t H \geq t^2 \Tr_t \dot{\alpha}_t  + (ATt-At^2)\Tr_t(\omega_t)  - AnTt \geq t^2 \Tr_t \dot{\alpha}_t - C.
\end{equation}
 It follows therefore from \eqref{eq:ddot1} and \eqref{eq:ddot2} that  
$$
\left( \frac{\partial }{\partial t}-\Delta_{t} \right)H
\leq t \ddot{\varphi}_t (2- t\partial_r F) - t^2n^{-1} (\Tr_t \alpha_t)^2  +C.
$$
%if we choose $B>0$ so large that
%$$
%t\frac{\partial^2 F}{\partial t^2}(t,x,\f_t)
%+2t\dot{\f}_t \frac{\partial^2 F}{\partial r \partial t}(t,x,\f_t)+B \geq 0.
%$$
%We also use here the simplifying assumption that $r \mapsto F(\cdot,\cdot,r)$ is convex
%so that $-(\dot{\f}_t)^2 \frac{\partial^2 F}{\partial r^2}(t,x,\f_t) \leq 0$.
%It is likely that this restriction is unnecessary.

Let $(t_0,x_0) \in X_T$ be a point at which the function $H$ realizes its maximum.  
If $t_0=0$ then $H \leq 0$ hence $t^2\ddot{\f}_t \leq C$ and we are done. 
If $t_0>0$, then 
$
0 \leq \left( \frac{\partial }{\partial t}-\Delta_{S_t} \right)H
$
at the point $(t_0,x_0)$ thus for $(t,x)=(t_0,x_0)$, 
$$
 t \ddot{\varphi}_t (t\partial_r F-2) + t^2 n^{-1} (\Tr_t \alpha_t)^2  \leq C. 
$$
Using \eqref{eq:ddot0} we conclude that $t^2 \ddot{\f}_t \leq C$ on $X_T$, finishing the proof.
\end{proof}

\subsection{Conclusion}

\subsubsection{The estimates}

We  summarize here the   a priori estimates we have obtained so far.
We assume that the forms and densities are {\it smooth} and satisfy the uniform
bounds listed in the introduction, involving the constants $A,p,\lambda_F,C_F$.

\begin{thm} \label{thm:sumup}
There exists $C_0,C_1,C_2>0$ such that
for all $(t,x) \in X_T$,
\begin{enumerate}
\item $-C_0  \leq \f_t(x) \leq C_0$;
\item $n\log t -C_1 \leq \dot{\f}_t(x) \leq C_1/t$;
\item $\ddot{\f}_t(x) \leq C_2/t^2$;
\end{enumerate}
\noindent where the $C_j$'s depend on $A,B,p,C_F,\lambda_F$ and
\begin{itemize}
\item $C_0$ explicitly depends on $T,\theta,\Theta,\inf_X \f_0, \sup_X \f_0$
and $\sup_{X_T} |F(t,x,0)|$;
\item $C_1$ explicitly depends on $C_0,T,||\partial F/\partial r||_{L^{\infty}}, ||\partial F/\partial t||_{L^{\infty}}$
and $||g||_{L^p}$;
\item $C_2$ explicitly depends on $C_0,C_1,T$.
\end{itemize}
\end{thm}

The norms $||\partial F/\partial r||_{L^{\infty}}, ||\partial F/\partial t||_{L^{\infty}}$ are computed on $X_T \times [-C_0,C_0]$.

\subsubsection{Convergence of semi-concave functions}

It is useful to know when a sequence of $\omega_t$-psh functions is uniformly semi-concave.
It allows one to obtain the convergence of the associated parabolic Monge-Amp\`ere operators
as the following result shows:

\begin{thm} \label{thm:conv2}   
Let $g_j(t,x)$ be a family of $L^1 (X_T)$-densities such that $g_j \rightarrow g$ in $L^1 (X_T)$.
Let $F_j(t,x,r)$ be continuous densities which uniformly converge towards $F$.
Let $\f_j(t,x)$  be a family of  $\omega_t$-psh functions such that
\begin{itemize}
\item $(\f_j)$ is uniformly bounded;
\item $\ddot{\f}_j \leq C/t^2$ for some uniform constant $C>0$.
\end{itemize}

Then there exists a bounded function $\f \in \mathcal P (X_T,\omega)$ such that, up to extracting and relabelling, 
$\f_j \rightarrow \f$ in $L^1_{\loc} (X_T)$ and
$$
e^{\dot{\f}_j+F_j(t,x,\f_j(t,x))} g_j(t,x) \, d t \wedge dV(x)  \longrightarrow e^{\dot{\f}+F(t,x,\f(t,x))} g(t,x) \, d t \wedge dV(x), 
$$
in the weak sense of Radon measures on $X_T$.
\end{thm}

%Here $L^1 (X_T)$ is the Lebesgue space with respect to the Lebesgue measure   on $X_T$ associated to the volume form $d t \wedge d V$.

\begin{proof} 
Since $(\f_j)$ is bounded in $L^2 (X_T)$, it is weakly compact. Extracting and relabelling, we 
assume that $(\f_j)$ weakly converges   to   $\f \in L^2 (X_T)$. 

Fix a compact sub-interval $J\Subset ]0,T[$. There exists a constant $C=C_J>0$ such that the functions $t\mapsto \varphi_j(t,x) -Ct^2$ are concave in $J$ for all $x\in X$ fixed. The same propertie holds for the limiting function $\varphi(t,x)$ by letting $j\to +\infty$. 
For  $t$ fixed, the functions $x \mapsto \f_j(t,x)$ are $\omega_t$-psh and uniformly bounded, hence 
 $x \mapsto \f(t,x)$ is $\omega_t$-psh and uniformly bounded in $X_T$.

It follows from Theorem \ref{thm:Montel}  that $\f_j \to \f$ in $L^1_{\loc} (X_T)$ and $\f_j (t,x) \to \f (t,x)$ 
almost everywhere in $X_T$ with respect to the Lebesgue measure.
The conclusion follows by applying Theorem~\ref{thm:conv1}.
\end{proof}

 \section{Existence and properties of sub/super/solutions}  \label{sec:existence}
 
 From now on we assume that $t \mapsto \omega_t$ and the denstities $g,F$ satisfy the conditions listed in the introduction. 
 
For bounded parabolic potentials $\varphi \in \mathcal{P}(X_T) \cap L^{\infty}(X_T)$, 
 the equation \eqref{eq:CMAF} should be understood in the sense of measures on $]0,T[ \times X$: 
 \begin{equation}\tag{CMAF}
 	\label{eq: MAF as measures}
 	(\omega_t +dd^c \varphi_t)^n \wedge dt =e^{\dot{\varphi}_t + F(t,x,\varphi_t)}g(x) dV(x)\wedge dt. 
 \end{equation}
 It follows from Definition \ref{def:LHS} that the left hand side is a well defined   Radon measure, while
Lemma \ref{lem: existence of time derivative} ensures that so is the right hand side.

 \subsection{Stability estimates}

We establish in this section  uniform  $L^{\infty}$-$L^1$ stability estimates 
needed in the proof of the existence theorem.

\begin{pro} \label{pro:stab} % Proposition 5.3
Fix $0 \leq g_1, g_2 \in L^ p (X)$ with $p > 1$, and $0 < T_0 < T_1 < T$.
Assume $\f^1, \f^ 2 \in {\mathcal P} (X_T,\omega) \cap {\mathcal C}^{\infty} (X_T)$ both satisify
$$
 d t \wedge (\omega_t + dd^c \f^i)^n  = e^{\dot{\f}^i (t,\cdot) + F_i (t, \cdot, \f^i)} g_i  d t \wedge d V .
$$
Then for all  $(t,x) \in [T_0,T_1] \times  X$, 
\begin{equation} \label{eq:stab}
   \vert\f^1 (t,x) -  \f^2 (t,x)\vert  \leq B  \Vert \f^1 - \f^2 \Vert_{L^1 (X_T)}^{\alpha},
\end{equation}
where  $0< \alpha=\alpha (n,p)$    
while  $0< B $  depends   on $T_0, T_1,\theta,\Theta$, 
and upper bounds for  $\Vert g^i \Vert_{L^p (X)}$, $\Vert F_i\Vert_{L^{\infty}(X_T)}$,
and $\Vert \partial_t F_i\Vert_{L^{\infty}(X_T)}, \Vert \partial_r F_i\Vert_{L^{\infty}(X_T)}$.
\end{pro}

\begin{proof}
We  are going to use the stability results of \cite{EGZ09,GZ12}. 
These rely on important estimates which we recall for the convenience of the reader.
The uniform bounds $\theta \leq \omega_t \leq \Theta$  and  \cite[Lemma 2.2]{EGZ08} show
that there exists a uniform constant $A_1 > 0$ such that 
$$
\Vol (K) \leq A_1 \capacity_{\omega_t} (K)^2,
$$
for all  $t \in [0,T]$ and all compact sets $K \subset X$, where
$$
\capacity_{\omega_t} (K):=\sup \left\{ \int_K (\omega_t+dd^c u)^n  \setdef u \in \PSH(X,\omega_t) 
\text{ with } 0 \leq u \leq 1 \right\}
$$
is the Monge-Amp\`ere capacity associated to the form $\omega_t$.

Fix $T_0 < T_1 < T$ and consider the densities 
 $$
 f^i_t :=  e^{\dot{\f}^i (t,\cdot) + F_i (t, x, \f^i)} g_i (x), \, i = 1, 2.
 $$
 It follows from Theorem~\ref{thm: apriori-est C1 general} that $t \dot{\f}^i (t,x)$ is uniformly bounded by $C_1 $,
 while  Proposition~\ref{pro:bdd} ensures that the $\f^i$'s are uniformly bounded.
The $L^p$ norms of the densities $f^i_t$ are thus uniformly bounded from above  by 
$$
 A_2 := e^{C_1 \slash T_0 +C_2} (\Vert g_0 \Vert_{L^p (X)} + \Vert g_1 \Vert_{L^p (X)})
 $$
  when $t \in [T_0,T_1]$. It follows therefore from \cite[Proposition 3.3]{EGZ09} that
 $$
 \max_X \vert \f^1 (t,\cdot) - \f^2 (t,\cdot)\vert \leq C \Vert \f^1 (t,\cdot) - \f^2 (t,\cdot)\Vert_{L^1 (X)}^{\gamma},
 $$
 for all  $t \in [T_0,T_1]$,
 where 
 %$C = C (M,A_1,A_2) > 0 $  only depends  on   $M, A_1, A_2$ and 
 $\gamma \in ]0, 1[$ only depends on $p, n$.
  Lemma~\ref{lem:L1Slice-L1} yields
$$
 \Vert \f^1 (t,\cdot) - \f^2 (t,\cdot)\Vert_{L^1 (X)} \leq A \max 
 \{  \Vert \f^1 - \f^2 \Vert_{L^1(X_T)},  \Vert \f^1 - \f^2 \Vert_{L^1(X_T)}^{1 \slash 2} \},
$$
where $ A := 2 \max \{ \sqrt{ \kappa}, (T - T_1)^{- 1}\}.$

The proof is completed by combining the last two inequalities.
\end{proof}

 In practice, this proposition yields the following useful information: 
 
 \begin{cor}
 Assume $||g_j||_{L^p},||F_j||_{L^{\infty}},||\partial_t F_j||_{L^{\infty}}$ and $||\partial_r F_j||_{L^{\infty}}$ are uniformly bounded.
 If a sequence $(\f^j)$ of solutions to $\eqref{eq:CMAF}_{F_j,g_j}$ converges
 in $L^1(X_T)$ to   $\f$, then it uniformly converges 
 on compact subsets of $]0,T[ \times X$.  
 \end{cor}

\subsection{The Cauchy problem}
We are now in position to prove \ref{main Thm A} of the introduction.

\begin{defi}
	A parabolic potential $\varphi\in \Pc(X_T,\omega)$ is a pluripotential solution (sub/super solution respectively) to \eqref{eq: MAF as measures} with initial values $\varphi_0 \in \PSH(X,\omega_0)\cap L^{\infty}(X)$ if $\varphi$ satisfies \eqref{eq: MAF as measures} (or the inequality $\geq/ \leq$ respectively) in the sense of measures on $X_T$ and $\varphi_t\to \varphi_0$ in $L^1(X)$ as $t\to 0^+$. 
\end{defi}

%\subsubsection{Upper semi-continuity}

%A bounded parabolic potential 
%$\f\in \Pc(X_T)$ is a pluripotential solution to  \eqref{eq:CMAF} if
%it satisfies 
%$$
% (\omega +dd^c \f_t)^n  \wedge dt= e^{\dot{\f}_t + F (t,x,\f_t) } g(x) \, dV(x) \wedge dt,
%$$
%in the sense of measures on $X_T:=]0,T[\times X$. The data $(\omega,g,F)$ is as in the introduction.

%\begin{defi}
%A bounded parabolic potential  $\varphi \in \Pc(X_T)$
% is a solution to
%the Cauchy problem for  \eqref{CMAF} with initial data $\f_0$ if  \marginpar{A revoir}
%\begin{itemize}
%\item $\f$  is a solution of the equation \eqref{CMAF} on  $X_T$,
%\item and $\f_t \rightarrow \f_0$ in $L^1(X)$, 
%as $t \to 0$.
%\end{itemize}
%\end{defi}

%The pluripotential solution that we are going to construct below is upper semicontinuous on $[0,T[ \times X$, thanks to  Proposition \ref{pro: usc}.

%\subsubsection{Existence result}

%Fix  $\f_0$  a bounded $\omega_0$-psh function in $X$. We now show that the Cauchy
%problem for the equation \eqref{CMAF} with initial data $\f_0$ admits at least a solution :

 \begin{thm} \label{thm:existence} 
Assume that  $\f_0$  a bounded $\omega_0$-psh function in $X$, and $(\omega,F,g)$ is as in the introduction. There exists  $\f \in \mathcal P (X_T,\omega)$  solving  \eqref{eq: MAF as measures} such that for all $0 <T'<T$,
 \begin{itemize}
\item $(t,x) \mapsto \f(t,x)$ is uniformly bounded in $]0,T'] \times X$;
\item $(t,x) \mapsto \f(t,x)$ is continuous in $]0,T[ \times \Omega$;
\item $t \mapsto \f(t,x) -n(t \log t -t) +C_1t$ is increasing on $]0,T'[$ for some $C_1>0$;
\item $t \mapsto \f(t,x) +C_2 \log t$ is concave on $]0,T'[$  for some $C_2>0$;
%for each compact subinterval $J \Subset ]0,T[$ there exists a constant $C_J>0$ such that $t \mapsto \varphi (t,x) -C_Jt^2$ is concave in $J$;
\item $\f_t \rightarrow \f_0$ as $t \rightarrow 0^+$ in $L^1(X)$ and pointwise. 
\end{itemize}
 \end{thm}
 
 Recall that $\Omega$ is the ample locus of $\theta$.
 The solution we provide is in particular locally uniformly semi-concave in $t \in ]0,T[$.
  We will study the uniqueness issue in the next section.
 
 \begin{proof}  
 Fix $0<T'<T$. We prove the existence of a solution on $]0,T'[\times X$. The uniqueness result (Corollary \ref{cor:unique1}) then ensures that a solution exists in $]0,T[\times X$. 
 
 We  approximate 
 \begin{itemize}
 	\item  $g$ by smooth densities $g_j > 0$ in $L^p(X)$; 
 	\item   $F$ by smooth densities $F_j$ with uniform constants $\kappa_{F_j}, C_{F_j}, \lambda_{F_j}$ ;
 	\item $\varphi_0$ from above on $X$ by smooth $(\omega_0+2^{-j}\Omega)$-psh functions $\varphi_{0,j}$. 
 \end{itemize}
It is well-known (see e.g. \cite{TosAFST}) that there exists a unique smooth solution $\f^j \in \mathcal P (X_T,\omega)$ 
to  $\eqref{eq:CMAF}_{F_j,g_j}$, i.e. 
\begin{equation} \label{eq:approxsol}
dt \wedge (\omega_t + dd^c \f^j_t)^n = e^{\dot{\f}^j(t,x)+F_j(t,x,\f^j(t,x))} g_j d t \wedge dV(x). 
\end{equation}
 
It follows from Theorem \ref{thm:sumup} that the $\f^j$'s are
uniformly bounded and  the derivatives $\ddot{\f}^j$ are locally uniformly bounded from above
in $X_T$. Extracting and relabelling 
%it follows therefore from 
Theorem~\ref{thm:conv2} ensures that
 there exists  $\f \in \mathcal P  (X_T) \cap L^{\infty}_{\loc}  (X_T) $ such that  $\f^j \rightarrow \f$ in $L^1 (X_T)$ and
$$
e^{\dot{\f}^j(t,x)+F_j(t,x,\f^j(t,x))} g_j(t,x) d t \wedge dV(x) \longrightarrow e^{\dot{\f}(t,x)+F(t,x,\f(t,x))} g(x) d t \wedge dV(x),
$$
in the sense of currents on $X_T$.

We claim that $\f^j \to \f$ locally uniformly in $X_T$. This follows indeed from the stability estimates established in Proposition \ref{pro:stab} above.
 Fix $0 < T_0 < T_1  <  T$. Since the densities $g_j$ have uniform $L^p$ norms, Theorem \ref{thm: apriori-est C1 general} ensures that the sequence $(\dot{\f}^j)$ is uniformly bounded in $[T_0,T] \times X$. By \eqref{eq:stab}, for all $j, k$ large enough, $t \in [T_0,T_1]$, $x \in X$, we have
$$
 \vert  \f^j (t,x) - \f^k (t,x)  \vert \leq C \Vert \f^j - \f^k \Vert_{L^1 (X_T)}^{\alpha},
$$
where $C > 0$ and $0 < \alpha < 1$ are uniform constants which do not depend on $j, k,$ and $t \in [T_0,T_1]$.
This proves our claim. 

Therefore  $ d t \wedge (\omega_t + dd^c \f^j_t)^n \to d t \wedge (\omega_t+dd^c \varphi_t)^n$ in the sense of measures on $X_T$,
hence $\varphi$ solves \eqref{eq: MAF as measures}.
%$$
% d t \wedge (\omega_t+dd^c \f_t)^n  = e^{\dot{\f} +F(t,x,\f_t(x))} g (x)  d t \wedge dV(x),
%$$
%in the  sense of measures on $X_T$, i.e. $\f$ is a pluripotential solution to\eqref{eq:CMAF}.  

One shows similarly that  $\f$ is uniformly semi-concave in $]0,T[$ :
 the densities $g_j$ in (\ref{eq:approxsol}) are uniformly bounded in $L^p (X)$, hence Theorem~\ref{thm:apriori-est} 
insures the existence of a uniform constant $C > 0$ such that 
$$
 \ddot \varphi^j (t,x) \leq C/ t^2
$$
for all $j \in \N, (t,x) \in X_T$. 
Thus, for each compact subinterval $J\Subset ]0,T[$ there exists a constant $C_J>0$ such that the functions $t \longmapsto \varphi^j (t) - C_J t^2 $ are concave  in $J$, and the same property holds for $\varphi$ by letting   $j \to \infty$.

The continuity of $\f$ on $]0,T[\times \Omega$ follows from the elliptic theory as will be shown in Proposition \ref{pro: continuity of solution elliptic EGZ} below.
The lower bound
$\dot{\f}_t \geq n \log t -C_1$, provided by Theorem \ref{thm: apriori-est C1 general}, ensures that
$t \mapsto \f_t-n(t \log t-t)+C_1 t $ is increasing, hence any cluster point (in $L^1$-topology) of $\varphi_t$ (as $t\to 0^+$) is greater than $\varphi_0$. 
 %Observe that  $\omega_t=\omega_0+t \chi   \geq 0$
%for all $t \in [0,T]$ hence $\chi \geq -T^{-1} \omega_0$
%and $\omega_t \geq (1-t/T) \omega_0$.
%We can assume without loss  of generality that 
%$T<1/2$ and 
%$\omega_t \geq (1-t) \omega_0$, since we are interested in the behavior near zero.
%
%It follows from the uniform semi-concavity of $\varphi_{t}^j$ in $t$ that the function $t\mapsto \varphi_t^j - C(t\log t -t)$ is concave in $t\in ]0,T[$.  Hence, for $t\in ]0,1[$ we have, after letting $j\to +\infty$,
% \[
% \varphi_t-C(t\log t -t) \geq (1-t)\varphi_0 + t (\varphi_1 + C). 
% \]
% From this we get a lower bound for $\varphi_t$ and see that any cluster point of $(\varphi_t)$ lies above $\varphi_0$. 
% 
 On the other hand, it follows from Proposition \ref{prop: average} and Lemma \ref{lem:L1Slice-L1} that 
 \[
 \int_X \varphi_t gdV \leq \int_X \varphi_0 gdV + Ct,
 \]
 for a uniform constant $C>0$. Let $u_0$ be any cluster point of $(\varphi_t)$ as $t\to 0^+$. Then as explained above, $u_0 \geq \varphi_0$. On the other hand, the average control above ensures   that
 \[
 \int_X u_0 g dV \leq \int_X \varphi_0 g dV.
 \]
 Since the set $\{g=0\}$ has Lebesgue measure zero we infer $u_0= \varphi_0$ almost everywhere, hence everywhere. 
  \end{proof}
  
  \begin{rem}
  Proposition \ref{pro: usc} ensures that the pluripotential solution constructed above is upper semi-continuous on $[0,T[\times X$. 
The functions 
 $\f_t$ quasi-decrease to $\f_0$ as $t \searrow 0$. The convergence at time zero is thus quite strong : if $\f_0$ is continuous, it follows for instance that the convergence is uniform
 (for non continuous initial $\f_0$, there is convergence in capacity). 	
  \end{rem}

\begin{rem}  
The way the density is allowed to vanish is crucial.
% (we require here that its bounded below by $e^w$ with $w$ quasi-psh).
\ref{main Thm A} does not hold for an arbitrary  density $g \geq 0$ :
if $g$ vanishes in a non empty open set $D \subset X$ then  \eqref{eq: MAF as measures} has no solution with initial  value 
$\f_0$ unless $\f_0$ is a maximal $\omega_0$-psh function in $D$.
Indeed the complex Monge-Amp\`ere operator is continuous for the convergence in capacity,
%of bounded psh functions,
so $(\omega_t+dd^c \f_t)^n=0$ would converge to $(\omega_0+dd^c \f_0)^n=0$ in $D$.
 \end{rem}

\subsection{Invariance properties of the set of assumptions} \label{sec:invariance}
 
%We check in section \ref{sec:checkgeom} that our conditions 
% are   satisfied for the parabolic equations that describe the evolution of
% the normalized (as well as the non-normalized) K\"ahler-Ricci flow on
% a mildly singular K\"ahler variety.
 
% \smallskip
 
 The family of parabolic complex Monge-Amp\`ere equations we consider
 $$
  (\omega_t +dd^c \f_t)^n=e^{\dot{\f}_t + F (t,x,\f) } g(x) dV (x), 
 $$
  has several invariance properties, as we now briefly explain.
 
 \subsubsection{Translations}
 
 We can replace $\f_t(x)$ by $\p_t(x)=\f_t(x)+C(t)$ without changing the Monge-Amp\`ere term, while the
 density $F$ is modified into
 $$
 \tilde{F}(t,x,r)=F(t,x,r-C(t))-C'(t).
 $$
 We let the reader check that $\tilde{F}$ satisfies the same set of assumptions as $F$.
 
 More generally we can  replace $\omega_t$ by $\eta_t=\omega_t-dd^c \rho_t$, changing
 $\f_t(x)$ in $\f_t(x)+\rho_t(x)$. The density $g$ remains unchanged while the new density $F$ is
 $$
  \tilde{F}(t,x,r)=F(t,x,r-\rho(t,x))-\partial_t \rho(t,x).
  $$

  \subsubsection{Scaling}
  
  A more involved transformation consists in scaling in space and renormalizing in time, so that the equation keeps the same shape.
  Namely we replace
  $\omega_t$ by $\gamma(s) \omega_{t(s)}$
  as well as $\f_t(x)$ by $\p_s(x)=\gamma(s) \f_{t(s)}(x)$, where
  $s \mapsto \gamma(s)>0$ is smooth and positive, $t(0)=0$ and $t'(s)=1/\gamma(s)$, so that 
 $$
 \partial_s \p_s=\frac{\gamma'(s)}{\gamma(s)} \p_s+\partial_t \f_{t(s)}.
 $$
 The density $g$ remains unchanged while the density $F$ is transformed into
 $$
 \tilde{F}(s,x,R)=F\left(t(s),x,r(s,R) \right)+n \log \gamma(s)-\frac{\gamma'(s)}{\gamma(s)} R,
$$
where $r(s,R)=\frac{R}{\gamma(s)}$.  

%We impose one more condition on $\gamma$ 
%\begin{equation}
%	\label{eq: transformation 0,T 0 S}
%\int_0^{+\infty} \frac{ds}{\gamma(s)}>T,
%\end{equation}
%After dilation in the time variable, we can  ensure that there exists $S>0$ with $t(S)=T$. 
A classical example of such a transformation is when $\gamma(s)=e^s$ and $t(s)=1-e^{-s}$, allowing one to pass from the
K\"ahler-Ricci flow to the {\it normalized} K\"ahler-Ricci flow.

We let the reader check that $ \tilde{F}$ remains quasi-increasing in $R$ and locally uniformly Lipschitz in $(s,R)$. 
It is slightly more involved to keep track of the semi-convexity property :

\begin{lemma} \label{lem: invariance property}
	The function $(s,R) \mapsto \tilde{F}(s,x,R)$ is locally  uniformly semi-convex in $(s,R)$.
\end{lemma}

\begin{proof}
Fix $0<S_0<S$, $T_0=t(S_0)$, and a compact interval $J \Subset \mathbb{R}$. We want to prove that $(s,R) \mapsto \tilde{F}(s,R)$ is semi-convex in $[0,S_0] \times J$.  We omit in the sequel the dependence on $x$ as it is not affected by  the transformation.

We can assume that $F$ is smooth and proceed by approximation.
The goal is to prove that the Hessian matrix $H(s,R)$ of $(s,R) \mapsto \tilde{F}(s,R)$ satisfies 
$$
H(s,R) + C I_2 \geq 0,
$$
where $I_2$ is the identity matrix in $M_2(\mathbb{R})$, and the constant $C$ is under control. 
Increasing   $C$ we can also assume that $F$ is convex in $[0,S_0] \times J$. 
Recall that $F$ is Lipschitz on $[0,T_0] \times J$ and $s\mapsto \gamma(s)>0$ is smooth. Using this we can write 
\begin{flalign*}
\frac{\partial^2 \tilde{F}}{\partial s^2} &=  \frac{\partial^2 F}{\partial t^2 } \left ( \frac{\partial t}{\partial s} \right )^2  + 2 \frac{\partial^2 F}{\partial t \partial r} \frac{\partial r}{\partial s} \frac{\partial t}{\partial s}+  \frac{\partial^2 F}{\partial r^2 } \left ( \frac{\partial r}{\partial s} \right )^2 + O(1),\\
\frac{\partial^2 \tilde{F}}{\partial R^2} &= \frac{\partial^2 F}{\partial r^2 } \left ( \frac{\partial r}{\partial R} \right )^2  + O(1),\\
\frac{\partial^2 \tilde{F}}{\partial R \partial s} &=  \frac{\partial^2 F}{\partial r^2 } \frac{\partial r}{\partial R} \frac{\partial r}{\partial s}+  \frac{\partial^2 F}{\partial r \partial t} \frac{\partial r}{\partial R} \frac{\partial t}{\partial s} +  O(1). % \frac{\partial^2 F}{\partial t^2 } \frac{\partial t}{\partial R} \frac{\partial t}{\partial s} +  \frac{\partial^2 F}{\partial r^2 } \frac{\partial r}{\partial R} \frac{\partial r}{\partial s}.
\end{flalign*}
It remains to check that 
$$
	\left[ {\begin{array}{cc}
   a & b \\
   b & c \\
  \end{array} } \right] \geq 0,
$$
where

\begin{flalign*}
a &=  \frac{\partial^2 F}{\partial t^2 } \left ( \frac{\partial t}{\partial s} \right )^2  + 2 \frac{\partial^2 F}{\partial t \partial r} \frac{\partial r}{\partial s} \frac{\partial t}{\partial s}+  \frac{\partial^2 F}{\partial r^2 } \left ( \frac{\partial r}{\partial s} \right )^2,\\
c &= \frac{\partial^2 F}{\partial r^2 } \left ( \frac{\partial r}{\partial R} \right )^2, \\
b &=  \frac{\partial^2 F}{\partial r^2 } \frac{\partial r}{\partial R} \frac{\partial r}{\partial s}+  \frac{\partial^2 F}{\partial r \partial t} \frac{\partial r}{\partial R} \frac{\partial t}{\partial s}. % \frac{\partial^2 F}{\partial t^2 } \frac{\partial t}{\partial R} \frac{\partial t}{\partial s} +  \frac{\partial^2 F}{\partial r^2 } \frac{\partial r}{\partial R} \frac{\partial r}{\partial s}.
\end{flalign*}
The convexity of $F$ and a direct computation ensure that $a,c\geq 0$ and $ac-b^2\geq 0$.  
\end{proof}

We note, for later use, that such a transformation allows one to reduce to the case when 
$r \mapsto F(\cdot,\cdot,r)$ is increasing :

\begin{lemma}
	\label{lem quasi-increasing vs increasing}
	Assume that $r \mapsto F(\cdot,r)$ is quasi-increasing  with $\lambda_F >0$. 
	Consider $\gamma: s \in [0,S[ \mapsto 1-\lambda_F s \in [0,T[$, 
	where $S<\lambda_F^{-1}$ is defined by $\int_0^S (1-\lambda_F r)^{-1} dr=T$.
	The function $R \mapsto \tilde{F}(s,R)$ is increasing, 
	  $\forall s\in [0,S[$.
\end{lemma}

\begin{proof}
	The function $\tilde{F}$ is given, for $(s,R) \in [0,S[ \times \mathbb{R}$, by
	$$
	\tilde{F}(s,R) = F(t(s), R/\gamma(s)) + n\log \gamma(s) + \frac{\lambda R}{ \gamma(s)}.
	$$ 
Using that $r\mapsto F(t,x,r)+ \lambda r$ is increasing, it is straightforward to check that $\tilde{F}$ is increasing in $R$. 
\end{proof}

 \subsection{Pluripotential sub/supersolutions}
 
  %We consider the following complex Monge-Amp\`ere flow on $X$ 
% \begin{equation}
 %	\label{eq: MAF uniqueness}
 %	(\omega_t +dd^c \varphi_t)^n =e^{\dot{\varphi}_t + F(t,x,\varphi)}g(x) dV(x). 
% \end{equation}

 \subsubsection{Definitions}
 
 Our plan is to establish a  pluripotential parabolic comparison principle. The latter is easier to obtain under an extra regularity
assumption in the time variable, so we  introduce the following terminology for convenience :
 
 \begin{definition}
 	A parabolic potential $u\in \mathcal{P}(X_T,\omega)$ is called of class $\mathcal{C}^{1/0}$  
 	 if for every $t\in ]0,T[$ fixed,   $\partial_t u(t,x)$ exists and is continuous in $\Omega$. 
 \end{definition}

 \begin{definition}
 	A parabolic potential $\f \in \mathcal{P}(X_T) \cap L^{\infty}(X_T)$ is called a pluripotential subsolution 
 	of \eqref{eq: MAF as measures} if 
 	$$
 	(\omega_t +dd^c \varphi_t)^n \wedge dt  \geq e^{\dot{\varphi}_t + F(t,x,\varphi)}g(x) dV(x)\wedge dt
 	$$
 	holds in the sense of measures in $]0,T[ \times X$.  
 	
 	Similarly  a parabolic potential $\f\in \mathcal{P}(X_T) \cap L^{\infty}(X_T)$ is called a pluripotential supersolution 
 	of \eqref{eq: MAF as measures} if 
 	$$
 	(\omega_t +dd^c \varphi_t)^n \wedge dt \leq e^{\dot{\varphi}_t + F(t,x,\varphi)}g(x) dV(x)\wedge dt
 	$$
 	holds in the sense of measures in $]0,T[ \times X$.  
 \end{definition}
 
 In many cases one can interpret these notions by considering a family of inequalities on slices :

 \begin{lemma} \label{lem: sub super solution pointwise}
 	Fix $u\in \Pc(X_T) \cap L^{\infty}(X_T)$.
 	
 1) If $u$	is a pluripotential subsolution of \eqref{eq: MAF as measures} such that
 $\partial_t^{+} u$ exists and  is lower semi-continuous in $t\in ]0,T[$, then for all $t \in ]0,T[$,
 $$
 (\omega_t+dd^c u_t)^n \geq e^{\partial^+_t u + F} gdV
 \text{  in the sense of measures in } X.
 $$

 2) If $u$	is a pluripotential supersolution of \eqref{eq: MAF as measures} such that
 $\partial_t^{-} u$ exists and  is upper semi-continuous in $t\in ]0,T[$, then for all $t \in ]0,T[$, 
 $$
 (\omega_t+dd^c u_t)^n \leq e^{\partial^-_t u + F} gdV
  \text{ in the sense of measures in } X.
 $$
 \end{lemma}
 
 \begin{proof}
 	We will prove the result for subsolutions. The corresponding result for supersolutions follows similarly. 
 	 Assume that the right-derivative $\partial^+_t u$ exists for all $(t,x)\in X_T$ and  is lower semi-continuous in $t$ 
 	 for $x$ fixed. It follows from \cite[Proposition 3.2]{GLZparab1} that for almost every $t\in ]0,T[$,
 	\[
 	(\omega_t+dd^c u_t)^n \geq e^{\partial^+_t u + F} gdV,
 	\]
 	in the sense of measures on $X$. 
 	Any $t \in ]0,T[$ can be approximated by a sequence $(t_j)_{j\in \mathbb{N}}$ for which the inequality above holds. The limiting inequality follows from the lower semi-continuity of $\partial_t^+ u(t,x)$ in $t$ and Fatou's lemma.  
 \end{proof}

\subsubsection{Properties of supersolutions}

We use  properties of solutions to  complex  Monge-Amp\`ere equations to show that
parabolic supersolutions automatically have continuity properties.

\begin{pro}
	\label{pro: continuity of solution elliptic EGZ}
	Assume that  $\psi \in \mathcal P (X_T) \cap L^{\infty} (X_T)$ is a supersolution to \eqref{eq:CMAF}. Then $\psi$ is continuous in $]0,T[ \times \Omega$.
\end{pro}

\begin{proof}
Fix $0 < T_0 < T_1 < T$. For almost every $t \in ]0,T[$ we have
$$
(\omega_t + dd^c \psi_t)^n \leq e^{\dot{\psi} (t,\cdot) + F (t,\cdot,\psi_t)} g d V
$$
in the weak sense on $X$.

Since $\psi$ is locally uniformly Lipschitz in $t$ and $F$ is bounded,   there exists  $M > 0$ 
such that $\dot{\psi} (t,\cdot) + F (t,\cdot,\psi_t) \leq M$ 
for almost any $t \in [T_0,T_1]$. Thus 
$$
(\omega_t + dd^c \psi_t)^n \leq e^{M} g d V,
$$
%in the weak sense on $X$, 
for almost every $t \in [T_0,T_1]$.
By (weak) continuity (in $t$)  of the LHS, it follows that this inequality  actually 
holds for any $t\in [T_0,T_1]$.

The elliptic theory (see e.g. \cite[Theorem 12.23]{GZ16})
implies that $\psi_t$ is continuous in $\Omega$ for any $t \in [T_0,T_1]$.
Since $\psi$ is uniformly Lipschitz in $[T_0,T_1]$ it follows that $\psi$ is continuous in $[T_0,T_1] \times \Omega$.
Indeed let $\kappa$ be the uniform Lipshitz constant of $\psi$ on $[T_0,T_1]$. Then for any $s, t \in [T_0,T_1]$ and $x, y \in \Omega$ we have
\begin{eqnarray*}
\vert \psi (s,x) - \psi (t,y)\vert & \leq & \vert \psi (s,x) - \psi (t,x)\vert + \vert \psi (t,x) - \psi (t,y)\vert \\
&\leq & \kappa \vert s - t\vert + \vert \psi (t,x) - \psi (t,y)\vert,
\end{eqnarray*}
which implies the continuity of $\psi$ in $[T_0,T_1] \times \Omega$.
%Since $0 < T_0 < T_1 < T$ was arbitrary, the conclusion follows.
\end{proof}

Supersolutions admit uniform bound from below :

\begin{pro}
 	\label{prop: estimate for super solutions}
 Assume that $\psi \in \Pc(X_T) \cap L^{\infty}_{\loc}(X_T)$ is a  pluripotential supersolution to \eqref{eq: MAF as measures}
 which is locally uniformly semi-concave in $t$. 
There exists   $C>0, t_0>0$ such that for all $(t,x) \in ]0,t_0] \times  X$,
 $$
 \psi(t,x) \geq (1-t)e^{-At} \psi_0(x) + C(t\log t -t).
 $$ 
 \end{pro}
 
Here $A$ is a positive constant such that %defined in \eqref{eq:omegatLip}: 
$$
-A \omega_t \leq \dot{\omega}_t \leq A \omega_t, \ \text{for all}\ t\in ]0,T[.
$$
This   implies in particular that $\omega_{t+s} \geq e^{-At} \omega_s$, for all $t,s>0$ with $t+s<T$.

\begin{proof}
Set $M:=M_{\psi}:= \sup_{X_{T/2}} |\psi|$.  
The Lipschitz condition on $F$  ensures that there exists a constant $\kappa=\kappa_F$ such that, for all $t,t' \in [0,T/2]$, $x \in X$, $r \in [-M,M]$, 
$$
|F(t,x,r) -F(t',x,r)| \leq \kappa|t-t'|. 
$$
Set $t_0:= \min(1,T/4, 1/2\lambda)$.  As observed above, $\omega_{t+s} \geq e^{-At} \omega_s$ for all $s\in ]0,t_0/2]$ and $t\in ]0,t_0[$. 
	Fix $s\in ]0,t_0]$ and consider, for $t \in ]0,t_0]$,
	$$
	u_s(t,x) := (1-t)e^{-At} \psi_s(x) + t\rho + C (t\log t -t),
	$$
	$$
	v_s(t,x):= \psi(t+s,x) +2\kappa ts,
	$$
where $\rho$ is a $\theta$-psh function on $X$, normalized by $\sup_X \rho =0$, which solves $(\theta+dd^c \rho)^n = e^{c_1}gdV$ with a normalization constant $c_1$, and $C$ is a positive constant to be specified later. 

The existence and boundedness of $\rho$ follows from \cite{EGZ09}.  Observe  that $u_s$ is of class $\mathcal{C}^{1/0}$ in $t$ and for each $t\in ]0,t_0]$ fixed, $u_s(t,\cdot)$ is continuous in $\Omega$ (see Proposition \ref{pro: continuity of solution elliptic EGZ}). 

A direct computation shows that, for $t\in ]0,t_0]$ ,
\begin{flalign*}
	(\omega_{t+s} +dd^c u_s)^n & \geq   \left ((1-t) \omega_{t+s}  + t \omega_{t+s} +dd^c ((1-t) e^{-At} \psi_s) + t dd^c \rho \right )^n\\
	&\geq \left( (1-t)e^{-At} (\omega_s + dd^c \psi_s) + t(\theta+dd^c \rho) \right )^n\\
	& \geq  t^n g dV.  
\end{flalign*}
In the second line above we have used $\omega_{t+s} \geq e^{-At} \omega_s$ while in the last line we have used $\omega_s+dd^c \psi_s \geq 0$. Thus, since $\psi_s$ is uniformly bounded,  by choosing $C>0$ large enough (depending on $M_F$) we obtain
$$
(\omega_{t+s} +dd^c u_s) ^n \geq e^{\partial_t u_s(t,\cdot) + F(t,\cdot,u_s(t,\cdot))} gdV. 
$$
It is also clear from the definition that $u_s(t,\cdot)$ converge in $L^1(X,dV)$ to $u_s(0,\cdot)=\psi_s$ as $t\to 0^+$.  

On the other hand, since $\psi$ is a supersolution to \eqref{eq:CMAF}, by Lemma \ref{lem: sub super solution pointwise} we have 
\begin{eqnarray*}
(\omega_{t+s} +dd^c v_s)^n & \leq e^{\partial_t^- \psi(t+s,\cdot) + F(t+s,\cdot,\psi(t+s,\cdot))}gdV\\
& \leq e^{\partial_t^- v_s(t,\cdot) -2\kappa s + F(t+s,\cdot,\psi(t+s,\cdot))} gdV.
\end{eqnarray*}
The Lipschitz condition on $F$ ensures that, for all $t,s\in [0,t_0], x\in X$,
$$
F(t+s,x, \psi(t+s,x)) \leq F(t,x,\psi(t+s,x)) + \kappa s. 
$$
The quasi-increasing property of $F$ ensures that 
\begin{flalign*}
F(t,x,\psi(t+s,x))  & = F(t,x, v_s(t,x) -2\kappa ts)\\
&   \leq 	F(t,x,v_s(t,x)) + 2 \kappa \lambda ts. 
\end{flalign*}
Thus for $t\leq t_0 \leq 1/2\lambda$, 
$$
(\omega_{t+s} +dd^c v_s)^n \leq e^{\partial_t v_s(t,\cdot) + F(t,\cdot,v_s(t,\cdot))} gdV.
$$

It follows from Proposition \ref{pro: continuity of solution elliptic EGZ} that $v_s$ is continuous on $[0,t_0]\times \Omega$ and $v_s(0,\cdot)= \psi_s$. We can thus apply Proposition \ref{thm: weak parabolic comparison principle v1} below
 and obtain $u_s\leq v_s$ on $]0,t_0[ \times X$. 
 Letting $s\to 0$ we  obtain that  for all $(t,x) \in ]0,t_0] \times X$,
$$
(1-t) e^{-At} \psi_0(x) + t\rho + C(t\log t -t) \leq \psi(t,x).
$$
The result follows since $\rho$ is bounded. 
\end{proof}

\subsubsection{Regularization of subsolutions}

  We introduce a  regularization process for  subsolutions. 
 % Assume that our data $(X_T, g,F,\omega_t)$ is as in Section \ref{subsect: ass data uniqueness} and we 
Fix $0<T'<T$ and $\varepsilon_0>0$ such that $(1+\varepsilon_0)T'<T$. 
    It follows from \eqref{eq:omegatLip} that there exists  $A_1>0$ such that 
    for all $t\in ]0,T'[$ and $ s\in [1-\varepsilon_0, 1+\varepsilon_0]$,
 \begin{equation}
 	\label{eq: regularization omega t lip}
 	\omega_{t} \geq (1- A_1|s-1|) \omega_{ts}, 
 \end{equation}
 where $\varepsilon_0>0$ is a fixed small constant. For $|s-1|<\varepsilon_0$ we set
 \[
 \lambda_s: = \frac{|1-s|}{s}, \ \alpha_s := s(1-\lambda_s)(1-A_1|s-1|) \in ]0,1[. 
 \]
 Up to shrinking $\varepsilon_0$ we can also assume that 
 for all $|s-1|\leq \varepsilon_0$,
 \[
\gamma_s:=  \frac{\lambda_s}{1-\alpha_s} \geq \varepsilon_1,
 \]
 where $\varepsilon_1=(5+A_1)^{-1}>0$.
 
 We let $\rho \in {\rm PSH}(X,\theta)$, $\sup_X \rho=0$, be the unique bounded solution to 
 \[
(\varepsilon_1 \theta + dd^c \rho)^n = e^{c_1} gdV,
 \]
 for some normalization constant $c_1\in \mathbb{R}$ (see \cite{EGZ09}).

\begin{lemma}\label{lem: regularization of subsolutions}
	Assume that 
	$u\in \Pc(X_T)$ is a bounded pluripotential subsolution of \eqref{eq: MAF as measures}. Then there exists a uniform constant $C>0$, depending on $M_u:= \sup_{X_T} |u|$ and the data,  such that
	for every $s\in [1-\varepsilon_0,1+\varepsilon_0]$,  
	\[
(t,z) \mapsto 	v_s(t,z) := \frac{\alpha_s}{s} u(ts,x) + (1-\alpha_s) \rho(x) - C|s-1|t
	\]
	is a pluripotential subsolution of \eqref{eq: MAF as measures} in $X_{T'}$. 
\end{lemma}

 \begin{proof}
 	%Since $\varphi_0$ is bounded we can choose $C>0$ large enough (under control) such that $v_s(0,\cdot) \leq \varphi_0$. 
 %	We  show that up to enlarging $C$ the function $v_s$ is a pluripotential subsolution of \eqref{eq: MAF as measures}. 
 For notational convenience we set 
 	\[
 	\beta_{s}: = \frac{1-\lambda_s}{\alpha_s} = 1+ O(|s-1|). 
 	\]
 	Since $u$ is a pluripotential subsolution of \eqref{eq: MAF as measures}, using \eqref{eq: regularization omega t lip} we  can write 
 	\begin{flalign*}
 		(\beta_s \omega_t + s^{-1}dd^c u(st,\cdot))^n & \geq s^{-n} (\omega_{ts}+dd^c u(st,\cdot))^n\\
 		&  \geq e^{-n\log s + \partial_{\tau} u(st,x) + F(ts, x, u(st,x))} g(x) dV.  
 	\end{flalign*}
 	 By the choice of $\rho$ we also have 
 	\begin{flalign*}
 		(\gamma_s \omega_t + dd^c \rho)^n \geq (\varepsilon_1 \theta+dd^c \rho)^n = e^{c_1} gdV. 
 	\end{flalign*}
 	Combining these with Lemma \ref{lem: MA mixed} below we arrive at 
 	\begin{flalign*}
 		(\omega_t +dd^c v_s(t,\cdot))^n &= \left [ (1-\lambda_s)\omega_t + \alpha_s dd^c u(st,\cdot) + \lambda_s \omega_t + (1-\alpha_s) dd^c \rho \right ]^n\\
 		& = \left [ \alpha_s(\beta_s\omega_t + s^{-1}dd^c u(st,\cdot) ) + (1-\alpha_s) (\gamma_s \omega_t +dd^c \rho) \right ]^n\\
 		&\geq e^{ \alpha_s\partial_{\tau} u(st,x) + \alpha_s F(t,x,u(st,x)) + (1-\alpha_s) c_1 - n\alpha_s \log s} g(x) dV. 
 	\end{flalign*}
 	Since $F(t,x,r)$ is uniformly bounded on $[0,T[ \times X \times [-M_u,M_u]$ and $\alpha_s-1=O(|s-1|)$, 
 	up to enlarging $C$  
 	we infer
 	\begin{flalign*}
 		(\omega_t +dd^c v_s(t,\cdot))^n & \geq e^{\partial_t v_s(t,x) + F(t,x,v_s(t,x))} g(x) dV. 
 	\end{flalign*}
 	This concludes the proof. 
 \end{proof}
 
 We have used the following mixed inequalities :
 
 \begin{lemma}\label{lem: MA mixed}
 	Let $\theta_1,\theta_2$ be two closed smooth semi-positive $(1,1)$-forms on $X$.
 	 %whose cohomology classes are big. 
 	 Let $u_1\in {\rm PSH}(X,\theta_1)$, $u_2\in \PSH(X,\theta_2)$ be bounded and such that 
 	\[
 	(\theta_1+dd^c u_1)^n \geq e^{f_1} \mu, \ \text{and} \ (\theta_2+dd^c u_2)^n \geq e^{f_2} \mu,
 	\]
 	where $f_1,f_2$ are bounded measurable functions and $\mu=hdV \geq 0$.
 	Then, for every $\alpha\in ]0,1[$, 
 	\[
 	(\alpha( \theta_1+ dd^c u_1)  + (1-\alpha)(\theta_2 + dd^c u_2))^n \geq e^{\alpha f_1 + (1-\alpha)f_2} \mu.
 	\]
 \end{lemma}
 
 \begin{proof}
 	The proof  is identical to that of \cite[Lemma 5.9]{GLZparab1} using the
 	convexity of the exponential
 	together with the mixed Monge-Amp\`ere inequalities due to  S. Ko{\l}odziej  \cite{Kol03} (see also \cite{Diw09}). 
 \end{proof}

 Let $\chi: \mathbb{R} \rightarrow [0,+\infty[$ be a smooth function with compact support in $[-1,1]$ such that $\int_{\mathbb{R}} \chi(s)ds=1$. For $\varepsilon>0$ we set $\chi_{\varepsilon}(s):= \varepsilon^{-1}\chi(s/\varepsilon)$.

 \begin{pro}
 	\label{thm: regularization of subsolutions}
 	Assume that 
 	$u\in \Pc(X_T)$ is a bounded pluripotential subsolution of \eqref{eq: MAF as measures}. 
 	Let  $v_s$ be defined as in Lemma \ref{lem: regularization of subsolutions}. 
 	
 	If $r \mapsto F(\cdot,\cdot,r)$ is convex then there exists a uniform constant $B>0$ such that, for $\varepsilon >0$ small enough, the function 
 	\[
 	u^{\varepsilon}(t,x) := \int_{\mathbb{R}} v_s(t,x) \chi_{\varepsilon}(s-1) ds  - B \varepsilon(t+1)
 	\]
 	is a pluripotential subsolution of \eqref{eq: MAF as measures}
 	which  	is ${\mathcal C}^{1/0}$ in $t$ and such that
 	$$
 	\sup_X [u^{\e}(0,x)-u_0(x)] \stackrel{\e \rightarrow 0}{\longrightarrow} 0.
 	$$
 	
 	When $r \mapsto F(\cdot,\cdot,r)$ is merely semi-convex, the same conclusion holds if we further  assume that 
 	\begin{equation}
 		\label{eq: Lips condition on u}
 		|\partial_t u(t,x)| \leq C/t, \ \forall (t,x)\in ]0,T'] \times X.
 	\end{equation}
  \end{pro}
   
  \begin{proof}
  %The proof is a simple adaptation of \cite{GLZparab1} so we will be brief. 
  We first assume that $r \mapsto F(\cdot,\cdot,r)$ is uniformly convex and increasing. In this case we do not need the assumption \eqref{eq: Lips condition on u}. 
  Fix $\varepsilon_0$ as in Lemma \ref{lem: regularization of subsolutions}. 
  For each $|s|\leq \varepsilon_0$ the function $v_s(t,z)$ is a pluripotential subsolution to \eqref{eq: MAF as measures}.

  As in \cite[Theorem 5.7, Step 3]{GLZparab1}, we use \cite[Main Theorem]{Guedj_Lu_Zeriahi_2017subsolution} and Jensen's inequality to show that, for any $t \in ]0,T[$,
\begin{equation} \label{eq:appr-subsol}
(\omega_t + dd^c u^\e)^n \geq \exp \left(\partial_t u^\e (t,x) + B \e +  \int_\R  F (t,x,v_s (t,x)) \chi_\e (s-1) ds\right) g d V,
\end{equation}
in the weak sense on $X$.

If $F (t,x,\cdot)$ is convex for any $(t,x)$, then
$$
\int_\R  F (t,x,v_s (t,x)) \chi_\e (s-1) ds \geq F (t,x,u^\e + B \e) \geq F (t,x,u^\e),
$$
since $F$ is non decreasing in $r$. Pluging this inequality in (\ref{eq:appr-subsol}) we conclude that for any $B \geq 0$, $u^\e$ is a subsolution to the equation $(5.1)$. 

\smallskip

 If $F$ is merely semi-convex,   the function $r \longmapsto F (t,x,r) + \lambda r^2$ is convex for any $(t,x) \in X_T$, 
 for some constant $ \lambda > 0$. Thus  
\begin{equation} \label{eq:qconv}
 \int_\R  F (t,x,v_s (t,x)) \chi_\e (s-1) ds  \geq F (t,x, u^\e (t,x)) \chi_\e (s-1) ds + \lambda Q_\e (t,x),
\end{equation}
where
$$
Q_\e (t,x) :=  \int_\R  v_s (t,x)^2  \chi_\e (s-1) ds - \left(\int_R v_s (t,x)  \chi_\e (s-1) ds\right)^2.
$$

We claim that there is  $C_1 > 0$ such that  $\vert Q_\e (t,x) \vert  \leq C_1 \e$ for all $(t,x) \in [0,T] \times X$. Indeed the family of function $ s \longmapsto v_s$ is uniformly bounded in $[0,T] \times X$ by a constant $M > 0$. Hence for any $(t,x) \in [0,T] \times X$ and $\e > 0$ small enough, we have
$$
\vert Q_\e (t,x) \vert \leq 2 M  \int_\R  \vert v_s (t,x)) - v_s^\e (t,x)\vert \chi_\e (s-1) ds,
$$ 
where $v_s^\e (t,x)) := \int_R v_s (t,x)  \chi_\e (s-1) ds$.

Recall that  $v_s (t,x) := \frac{\alpha_s}{s} u (s\cdot t,x) + (1- \alpha_s) \rho (x) - C \vert s - 1\vert t$.  
The condition $(5.3)$ ensures that the function $\partial_s v_s $ is uniformly bounded in 
$ s\in [1-\e_0,1 +\e_0]$ and $(t,x) \in X_T$. 
Thus the family $ s \longmapsto v_s$ is uniformly $L$-Lipschitz in $ s \in [1-\e_0,1 + \e_0]$, 
which proves our claim with $C_1 := 2 M L$.

By (\ref{eq:qconv}) this implies that for any $(t,x) \in X_T$, 
$$
\int_\R  F (t,x,v_s (t,x)) \chi_\e (s-1) ds  \geq F (t,x, u^\e (t,x)) \chi_\e (s-1) ds  - C_1 \e.
$$
Pluging this inequality in (\ref{eq:appr-subsol}) and taking $B \geq C_1$ we see that $u^\e$ is a subsolution to the equation $(5.1)$.
Taking $B$ large enough we obtain furthermore that $u^\e (0,x) \leq u_0 (x)$ for all $x \in X$.

We let the reader adapt these arguments to the situation when $r \mapsto F(\cdot,\cdot,r)$ is merely quasi-increasing.
  \end{proof}

\section{Uniqueness} \label{sec: uniqueness}

We have shown in the previous section that the Cauchy problem for \eqref{eq:CMAF}
with bounded initial data $\f_0 \in \PSH(X,\omega_0)$
admits a pluripotential solution which is 
 locally uniformly  semi-concave in $t$. We now prove that there is only one such solution.
  % We assume in this section that $T<+\infty$. 

 \subsection{Comparison principle 1}
 
 Our goal in this section is to establish the following comparison principle :

 \begin{theorem}
	\label{thm: weak comparison principle v2}
	Fix   $\varphi,\psi \in \Pc(X_T,\omega)\cap L^{\infty}(X_T)$ and assume that
	\begin{itemize}
		\item [(a)]$\varphi$ is a pluripotential subsolution to \eqref{eq: MAF as measures};
		\item [(b)] $\psi$ is a pluripotential supersolution to \eqref{eq: MAF as measures};
		\item [(c)] $x \mapsto \varphi(\cdot,x)$ is continuous in $\Omega$ and
		%\eqref{eq: Lips condition on u}
		$|\partial_t \f(t,x)| \leq C/t,$ \; $\forall (t,x)\in X_T$;		 
		 \item [(d)] $\psi$  is locally uniformly semi-concave in $t\in ]0,T[$;
		 \item [(e)] $\f_t \rightarrow \f_0$ and $\p_t \rightarrow \p_0$ in $L^1$ as $t \rightarrow 0$.
	\end{itemize}
	If $\varphi_0\leq \psi_0$ 
	then $\varphi\leq \psi$. 
\end{theorem}

 We first establish this result under extra assumptions :
 
  \begin{pro}\label{thm: weak parabolic comparison principle v1}
 	Fix $\varphi,\psi\in \mathcal{P}(X_T)\cap L^{\infty}(X_T)$. 
 	Assume that $\varphi$ (resp. $\psi$) is a pluripotential subsolution (resp. supersolution) 
 	to \eqref{eq: MAF as measures} such that
 	\begin{itemize}
 		\item[(a)] $\varphi$ is  $\mathcal{C}^{1/0}$ in $t$ 
 		and for any $t>0$, $\varphi(t,\cdot)$ is continuous on $\Omega$;
 		\item[(b)] $\psi$ is locally uniformly semi-concave in $t$;
 		 \item[(c)]  $\f_t \rightarrow \f_0$ and $\p_t \rightarrow \p_0$ in $L^1$ as $t \rightarrow 0$;
 		\item[(d)]  the function $(t,x) \mapsto \psi(t,x)$ is continuous on $[0,T[ \times \Omega$.
 	\end{itemize}
 Then 
 	\[
 	\varphi_0 \leq \psi_0 \Longrightarrow \varphi \leq \psi \ \text{on} \ X_T. 
 	\]
 \end{pro}
 
 A particular case of this result was established in \cite[Theorem 3.1]{GLZstability}.

 %The condition $\varphi_0\leq \psi_0$ means that  $$\limsup_{t\to 0} \varphi_t(x) \leq \liminf_{t\to 0} \psi_t(x), \ \forall x \in X.$$

 \begin{proof}
 	We fix $T'\in ]0,T[$ 	and we prove that $\varphi\leq \psi$ on $[0,T'] \times X$. The result then follows by letting $T'\to T$. 
Using the invariance properties of the family of equations (see Section \ref{sec:invariance}), we can assume without loss of generality that $r \mapsto F(\cdot,\cdot,r)$ is increasing.
 	 	We proceed in several steps. 
 	
 	\smallskip
 	
 	\paragraph{\it Construction of auxiliary functions.}
 	We first introduce two auxiliary functions. 
 	Let $\phi_1 \in \PSH(X,\theta/2)$ be a $\theta/2$-psh function with analytic singularities (in particular $\phi_1$ is smooth in $\Omega$) such that $\phi_1=-\infty$ on $\partial \Omega$. % In the argument below we will consider a function having the form $h+ \delta\phi_1$, where $h$ is upper semicontinuous in $\Omega$ and $\delta>0$. Since $\phi_1$ goes to $-\infty$ on $\partial \Omega$ the supremum of $h+\delta \phi_1$ is attained at some point in $\Omega$ and we can apply the classical maximum principle.   
 	We need to use this function in order to apply the classical maximum principle in $\Omega$. 
 	The standard strategy is to replace $\varphi$ by $(1-\delta)\varphi_t + \delta \phi_1$. However,   the time derivative $\dot{\varphi}_t$ may blow up as $t\to 0$ so we need to use a second auxiliary  function. 
 	Let $\phi_2 \in \PSH(X,\theta/2)$ be the unique  solution to 
 	\[
 	(\theta/2 + dd^c \phi_2)^n = e^{c_1}gdV,
 	\]
 	normalized by $\sup_X \phi_2=0$, where 
  $c_1\in \bR$ is a normalization constant. 
  Then $\phi:= \phi_1+\phi_2$  is a $\theta$-psh function which is continuous in $\Omega$ and tends to $-\infty$ on $\partial \Omega$. 
 	
 	\medskip
 	
 	\paragraph{\it Adaptation of the arguments in \cite{GLZstability}.}  
 	 Fix $\varepsilon,\delta>0$ small enough and set
 	\[
 	w (t,x):=(1-\delta)\varphi(t,x) + \delta \phi(x)-\psi(t,x)-3\varepsilon t, \ t\in [0,T'], x\in \Omega.
 	\]
 	
	This function is upper semi-continuous on   $[0,T'] \times \Omega$ (see Proposition \ref{pro: usc}) and tends to $-\infty$ on $\partial \Omega$, hence attains a maximum at some  $(t_0,x_0) \in [0,T'] \times \Omega$. 
	
	We claim that $w (t_0,x_0)\leq 0$. Assume by contradiction that $ w (t_0,x_0) >  0$. Then 
	 $t_0>0$ and we set
	 $$
	 K := \{ x \in \Omega \setdef  w (t_0,x) = w (t_0,x_0)\}.
	 $$
 Since $w$ tends to $-\infty$ on $\partial \Omega$ it follows from upper semi-continuity of $\varphi$  that $K$ is a compact subset of $\Omega$. Since $\varphi(\cdot, x_0)$ is differentiable in ]0,T[, 
the classical maximum principle insures that for all $x \in K$,
	$$
		(1-\delta)\partial_t \varphi(t_0,x) \geq \partial_t^{-} \psi(t_0,x) +3\varepsilon.
	$$
By assumption the partial derivative $\partial_t \varphi(t,x)$ is continuous on $\Omega$. Moreover, Proposition \ref{pro: continuity of solution elliptic EGZ} ensures that $\psi_t$ is continuous on $\Omega$, for all $t\in ]0,T[$. By local semi-concavity of $t\mapsto \psi_t$ it then follows that, for $t\in ]0,T[$ fixed, $\partial_t^{-} \psi(t,x)$ is upper semi-continuous in $\Omega$. 
We thus can find $\eta>0$ small enough such that, by introducing the open set containing $K$, $$D:= \{x\in \Omega \setdef w(t_0,x) > w(t_0,x_0) - \eta\}\Subset \Omega,$$ 
we have
	\begin{equation}\label{eq: comparison principle 5.4 1}
	(1-\delta)\partial_t \f (t_0,x) > \partial_t^{-} \psi(t_0,x) +2\varepsilon, \ \forall x\in D.
	\end{equation}
	
Set $u := (1-\delta) \varphi (t_0,\cdot)+ \delta \phi$ and $v := \psi (t_0,\cdot)$. 
Since $\varphi$ is a subsolution  to \eqref{eq: MAF as measures}, using Lemma \ref{lem: MA mixed}  we infer
\begin{eqnarray*}
	(\omega_{t_0} +dd^c u)^n & \geq & \left [ (1-\delta) (\omega_{t_0} +dd^c \varphi_{t_0}) + \delta (\theta/2+ dd^c \phi_2)  \right ]^n \\
	&\geq &e^{(1-\delta) (\partial_t \varphi(t_0,\cdot) + F(t_0,x,\varphi(t_0,\cdot))) + \delta c_1} gdV. 
\end{eqnarray*}
Since $F$ is bounded on $[0,T'] \times X \times [-M,M]$ for each $M>0$ and $\varphi$ is bounded on $[0,T'] \times X$ there exists a constant $C>0$ under control such that  
$$
(\omega_{t_0} +dd^c u)^n \geq e^{(1-\delta) \partial_t \varphi(t_0,\cdot) + F(t_0,x,\varphi(t_0,\cdot)) - \delta C} gdV,
$$
in the weak sense of measures in $\Omega$.
Using  \eqref{eq: comparison principle 5.4 1} and choosing $\delta<C^{-1} \varepsilon$ we then have
	\begin{eqnarray*}
		(\omega_{t_0} + dd^c u)^n &\geq & e^{\partial_t^{-} \psi(t_0,\cdot) + F(t_0,x,\varphi(t_0,\cdot))+ \varepsilon} gdV,
	\end{eqnarray*} 
	in the weak sense of measures in $D$. 
	Since $\psi$ is a pluripotential supersolution, Lemma \ref{lem: sub super solution pointwise} ensures
	$$
	(\omega_{t_0} +dd^c \psi_{t_0})^n \leq e^{\partial_t^{-} \psi(t_0,\cdot) + F(t_0,x,\psi(t_0,\cdot))} gdV,
	$$
	in the weak sense of measures in $D$. 
The last two estimates yield
	\[
	(\omega_{t_0} + dd^c u)^n
		 \geq  e^{F(t_0,\cdot,u(\cdot))-F(t_0,\cdot,v(\cdot))+\varepsilon}(\omega_{t_0} + dd^c v)^n.
	\]
%	in the weak sense of measures in $D$.  
	Recall that $u (x) > v(x)+\varepsilon t_0$ for any $x \in K$. 
Shrinking $D$ if necessary,  we can thus assume that $u(x)>v(x)$ for all $x\in D$. 

Since $r \mapsto F(t,x,r)$ is increasing 
we thus get
	\[
	(\omega_{t_0} + dd^c u)^n \geq  e^{\varepsilon}(\omega_{t_0} + dd^c v)^n,
	\]
	in the sense of positive measures in $D$. 
	
Consider now $\tilde{u}:=u +\min_{\partial D}(v -u)$. 
Since $v \geq \tilde{u}$ on $\partial D$, the comparison principle Proposition \ref{pro: domination principle} below yields
	\begin{flalign*}
		\int_{\{v<\tilde{u}\}\cap D}e^{\varepsilon} (\omega_{t_0} + dd^c v)^n & \leq \int_{\{v<\tilde{u}\}\cap D} (\omega_{t_0} + dd^c u)^n \\
		&\leq \int_{\{v<\tilde{u}\}\cap D} (\omega_{t_0} +dd^c v)^n. 
	\end{flalign*}	
%The set $\{x \in D \ : \ v (x) < u_r(x)\}$ has therefore measure $0$ with respect to the measure $(\omega_{t_0} + dd^c v)^n$, 
It then follows that $\tilde{u} \leq v$ almost everywhere in $D$ with respect to  the measure  $(\omega_{t_0} + dd^c v)^n$, hence everywhere in $D$ by the domination principle (see Proposition \ref{pro: domination principle}).  
In particular, 
	\begin{equation}
		\label{eq: varphi r less than psi}
		u (x_0)- v(x_0) +\min_{\partial D}(v - u)=\tilde{u}(x) -v (x) \leq 0, 
	\end{equation}
Since $K\cap \partial D =\emptyset$, we infer $w(t_0,x) < w(t_0,x_0)$, for all $x\in \partial D$, hence
	$$
	 u (x) - v (x) < u (x_0) - v (x_0)
	 \text{ for all }
	 x\in \partial D ,
	 $$
	   contradicting \eqref{eq: varphi r less than psi}. 
Altogether this shows that $t_0 = 0$, thus 
\[
(1-\delta) \varphi + \delta \phi - \psi - 3\varepsilon t \leq \delta \sup_X |\varphi_0| 
\]
in $[0,T'] \times \Omega$.  Letting $\delta\to 0$ and then $\varepsilon\to 0$ we arrive at $\varphi\leq \psi$ in $[0,T'] \times \Omega$ hence in $[0,T'] \times X$.
\end{proof}

  We have  used the following semi-local version of the {\it  domination principle}:

 \begin{pro}\label{pro: domination principle}
  Fix a non-empty open subset $D\subset X$
and let $u,v$ be bounded $\theta$-psh functions on $X$ such that 
 	\[
 	\limsup_{D\ni z\to \partial D} (u-v)(z)\geq 0.
 	\]
 	Then 
 	\[
 	\int_{\{u<v\}\cap D}\theta_{v}^n \leq \int_{\{u<v\}\cap D} \theta_u^n. 
 	\]
 	Moreover, if ${\rm MA}_{\theta}(u)(\{u<v\}\cap D)=0$ then $u\geq v$ in $D$. 
 \end{pro}

We now remove assumption $(d)$ in Proposition \ref{thm: weak parabolic comparison principle v1} :

\begin{pro}
	\label{prop: weak parabolic comparison principle v1 improved}
	Fix $\varphi,\psi \in \Pc(X_T) \cap L^{\infty}(X_T)$ such that $\varphi$ (respectively $\psi$) is a pluripotential subsolution (respectively supersolution) to \eqref{eq: MAF as measures}. 
	If  the assumptions $(a), (b),(c)$ of Proposition \ref{thm: weak parabolic comparison principle v1} are satisfied then 
	$$
	\varphi_0 \leq \psi_0 \Longrightarrow \varphi \leq \psi \ \text{on}\ X_T.
	$$ 
\end{pro}

\begin{proof}
We assume without loss of generality that $r \mapsto F(\cdot,\cdot,r)$ is increasing and
use an argument similar to that of Lemma \ref{lem: regularization of subsolutions}.
	Fix $s>0$ small enough and consider 
\[
v_s(t,x) := \psi(t+s,x) + Cst +Cs -Cs \log s,
\]
and 
\[
u_s(t,x) := \alpha_s \varphi(t,x) + (1-\alpha_s) \rho -Cst -Cs,
\]
where $\alpha_s:= (1-s)(1-As) \in ]0,1[$, $A>0$ is defined in \eqref{eq:omegatLip}, $\rho$ is a quasi-psh function and $C$ is a positive constant to be specified later. The goal is to show that for $C>0$ large enough (under control) $v_s$ is a supersolution while $u_s$ is a subsolution to a parabolic equation and $u_s(0,\cdot) \leq v_s(0,\cdot)$. 
We can then invoke Proposition \ref{thm: weak parabolic comparison principle v1} and let $s\to 0$ to obtain the result. 

By considering $s$ small enough we can assume that 
\[
\beta_s:= \frac{(1-As)s}{(1-\alpha_s)} \geq \varepsilon_1>0
\]
 for some uniform constant $\varepsilon_1$. We let $\rho$ be the unique $\varepsilon_1 \theta$-psh function on $X$ such that $\sup_X \rho =0$ and $(\varepsilon_1\theta+dd^c \rho)^n =e^{c_1}gdV$. By definition one has that $\alpha_s + (1-\alpha_s) \beta_s =1-As$.  Then one can show that 
\begin{flalign*}
(\omega_{t+s} + dd^c u_s)^n & \geq 	 \left [ (1-As) \omega_t + \alpha_s dd^c \varphi_t + (1-\alpha_s)dd^c \rho \right ]^n \\
& = \left [ \alpha_s(\omega_t +dd^c \varphi_t) + (1-\alpha_s) (\beta_s \omega_t +dd^c \rho) \right ]^n\\
& \geq e^{\alpha_s (\partial_t \varphi_t + F(t,\cdot,\varphi(t,\cdot))) + (1-\alpha_s) c_1} gdV.
\end{flalign*}
where in the last line we use Lemma \ref{lem: MA mixed}. Since $\alpha_s= 1+ O(s)$ and $F$ is bounded, by choosing $C>0$ large enough (depending on $M_F,c_1$) we have 
$$
(\omega_{t+s} +dd^c u_s)^n \geq e^{\partial_t u_s + F(t,\cdot,u_s(t,\cdot))} gdV. 
$$
%In other words this means that, for $s>0$ small enough, the function $u_s: ]0,T'[\times X \rightarrow \mathbb{R}$ is a pluripotential subsolution to the following equation $$(\eta_t +dd^c v_t)^n = e^{\dot{v}_t + F(t,x,v_t(x))} gdV,$$ where $\eta_t:= \omega_{t+s}$. 
On the other hand for $C>0$ large enough (which depends on $\kappa_F$) we have 
$$
(\omega_{t+s} +dd^c v_s)^n \leq e^{\partial_t v_s -Cs + F(t+s,\cdot,v_s(t,\cdot))}gdV \leq e^{\partial_t v_s + F(t,\cdot,v_s(t,\cdot))}gdV.
$$

Up to increasing $C$ it follows from Proposition \ref{prop: estimate for super solutions} that  $v_s(0,x)\geq \psi_0(x)$. Since $u_s(0,x)\leq v_s(0,x)$ it then follows from Proposition \ref{thm: weak parabolic comparison principle v1} that 
\[
u_s(t,x) \leq  v_s(t,x), \forall (t,x) \in X_{T}. 
\]
Letting $s\to 0$ we arrive at the conclusion, finishing the proof. 
\end{proof}

 \begin{proof}[Proof of Theorem \ref{thm: weak comparison principle v2}]
Fix $T'<T$. We regularize the subsolution $\varphi$ by applying Proposition \ref{thm: regularization of subsolutions}. 
The family of subsolutions obtained  this way is denoted by $u_{\varepsilon}$.  Then $u_{\varepsilon}$ is a pluripotential subsolution to \eqref{eq: MAF as measures} and, up to enlarging the constant $B>0$ in Proposition \ref{thm: regularization of subsolutions} we can also assume that $u_{\varepsilon}(t,x)$ converges to $\varphi_0$ in $L^1(X,dV)$ as $t\to 0^+$. 

It follows moreover from Proposition \ref{pro: continuity of solution elliptic EGZ} that for $t\in ]0,T[$ fixed, $\psi_t$ is continuous in $\Omega$.  We can thus apply Proposition    \ref{thm: weak parabolic comparison principle v1} 
and obtain $u_{\varepsilon} \leq \psi$ on $]0,T'] \times X$. Letting $\varepsilon \to 0$ and then $T'\to T$ we arrive at the conclusion.
\end{proof}

\begin{cor} \label{cor:unique1}
Fix $\varphi_0$ a bounded $\omega_0$-psh function. There exists a unique function 
$\varphi  \in \Pc(X_T)\cap L^{\infty}(X_T)$ such that  
	\begin{itemize}
		\item for each $t\in ]0,T[$, $x \mapsto \varphi(t,x)$ is continuous on $\Omega$ ;
		\item 	$|\partial_t \f(t,x)| \leq C/t$ for all $(t,x)\in ]0,T[ \times X$;
		\item $t \mapsto \f(t,\cdot)$ is locally uniformly semi-concave in $]0,T[$; 
		\item $\varphi_t \to \varphi_0$ in $L^1(X)$ as $t\to 0$; 
		\item $\varphi$ solves \eqref{eq:CMAF} in $X_T$. 
	\end{itemize}
\end{cor}

In particular any smooth approximants converge towards this solution, hence the latter is independent of the approximants. 

\begin{defi}
Given a data $(F,g,\omega,\varphi_0)$ we let $\Phi(F,g,\omega,\varphi_0)$	 denotes the unique pluripotential  solution to \eqref{eq: MAF as measures} as in Corollary \ref{cor:unique1}. 
\end{defi}

\subsection{Uniqueness and stability}

We now establish a more general uniqueness result. The proof relies on a delicate comparison principle
and yields as well stability results

\subsubsection{Stability 1}

 \begin{pro} \label{thm: stability}
 Assume that $(g,F,\omega_t)$ and  $(g_j,F_j,\omega_{t,j})$ satisfy the assumptions in 
 the introduction  with uniform constants independent of $j$, and
 \begin{itemize}
 \item $g_j,F_j$ are smooth
 \item  $0 < g_j$ converge in $L^p(X)$ to $g \in L^p(X)$;
 \item $\f_{0,j}$ are  uniformly bounded $\omega_0$-psh functions which converge in $L^1(X)$ towards $\f_0 \in L^{\infty}\cap \PSH(X,\omega_0)$;
 \item $F_j$ uniformly converge to $F$  on $[0,T[ \times X \times J$, for each $J\Subset \mathbb{R}$;
 \item $\omega_{t,j}$  uniformly converges  to $\omega_t$.
 \end{itemize}
 
Let $\varphi_j$ be the unique smooth solutions to \eqref{eq: MAF as measures} with data  $(g_j,F_j,\omega_{t,j},\varphi_{0,j})$.
Then $(\varphi_j)$ converges in $L^1_{\loc}(X_T)$
%locally uniformly 
to $\varphi$, where $\f$ is 
the unique solution of \eqref{eq: MAF as measures}  with data $(g,F,\omega_t,\f_0)$
provided by Corollary \ref{cor:unique1}.
 \end{pro}

 \begin{proof}
 	The sequence $(\varphi_j)$ satisfies the conditions of Theorem  \ref{thm:sumup}. 
 	Hence, Theorem \ref{thm:conv2} ensures that  a subsequence of $(\varphi_j)$, still denoted by 
 	$(\varphi_j)$, converges in $L^1_{\loc}(X_T)$ to a function $\varphi \in \Pc(X_T,\omega)$ which 
 	\begin{itemize}
 		\item is a pluripotential solution to \eqref{eq: MAF as measures};
 		\item  is locally uniformly semi-concave in $t \in ]0,T[$; 
 		\item satisfies $|\partial_t \f| \leq C/t$ on $]0,T[\times X$.
 	\end{itemize}
 	
 	 It follows from Proposition \ref{prop: estimate for super solutions} that there exists  uniform constants  $C>0, t_0>0$ 
 	 such that for all  $(t,x) \in ]0,t_0] \times X$, 
 	$$
 	\varphi_{j}(t,x) \geq (1-t)e^{-Ct} \varphi_{0,j}(x) + C(t\log t -t), .
 	$$
 	Letting $j\to +\infty$ we obtain, for all  $(t,x) \in ]0,t_0] \times X$, 
 	$$
 	\varphi(t,x) \geq (1-t)e^{-Ct} \varphi_{0}(x) + C(t\log t -t).
 	$$
 This lower bound and Proposition \ref{prop: average} ensure that  $\varphi_t\to \varphi_0$ in $L^1(X)$ as $t\to 0^+$. 
 It finally follows from Corollary \ref{cor:unique1}  that $\varphi \in \Pc(X_T,\omega)$ is uniquely determined.
 \end{proof}

\subsubsection{Comparison principle 2}
  
We now extend the comparison principle (Theorem \ref{thm: weak comparison principle v2}),
 avoiding the space continuity assumption on the subsolution $\varphi$, as well as the Lipschitz type control at the origin :
  
  \begin{theorem}  \label{thm: weak comparison principle v3}
  	Assume that $\varphi,\psi\in \Pc(X_T)\cap L^{\infty}(X_T)$ are such that
  	\begin{itemize}
  		\item $\varphi$ is a pluripotential subsolution to \eqref{eq: MAF as measures}; 
  		\item $\psi$ is a pluripotential supersolution to \eqref{eq: MAF as measures};
  		 \item $\psi$  is locally uniformly semi-concave in $t$;
  		 \item $\f_t \rightarrow \f_0$ and $\p_t \rightarrow \p_0$ in $L^1$ as $t \rightarrow 0$.
  	\end{itemize}
  	If $\varphi_0 \leq \psi_0$ then $\varphi\leq \psi$. 
  \end{theorem}
  
\begin{proof}
We fix $T'<T$ and prove that $\varphi\leq \psi$ on $[0,T']\times X$. 
The result  follows then by letting $T'\to T$. 
Using the invariance properties of our family of equations, we can assume 
%without loss of generality 
that $r \mapsto F(\cdot,\cdot,r) $ is increasing.

 If $\tilde{\psi}$ is the unique pluripotential solution constructed by approximation 
 (see Corollary \ref{cor:unique1}),  Theorem \ref{thm: weak comparison principle v2} ensures that $\tilde{\psi}\leq \psi$ on $X_{T'}$.  We can thus assume without loss of generality that $\psi=\tilde{\psi}$.  We proceed in several steps.

 \medskip
 
 \noindent{\it Step 1.}  Assume  that the data $(g,F,\psi_0)$ is smooth, $g>0$, $\theta$ is K\"ahler, and  the time derivative $\partial_t \varphi(t,x)$ is uniformly bounded in $]0,T'] \times X$.

 Then $\psi$ is smooth since there 
 exists a unique smooth solution  (as follows from \cite{GZ13,DNL14,Dat1} and Corollary \ref{cor:unique1}). 
 
 If $x \mapsto \f(\cdot,x)$ were known to be continuous, we could invoke Theorem \ref{thm: weak comparison principle v2}
 to conclude. In absence of this extra assumption, we take a little detour inspired by viscosity techniques. 
 
 \smallskip
 
 \noindent{\it Step 1.1.}   Assume the following conditions
 \begin{itemize}
 	\item[(1)] for all $x\in X$, $\dot{\varphi}_t(x)$ exists and $t\mapsto \dot{\varphi}_t(x)$ is continuous in $t\in ]0,T']$. 
 	\item[(2)] $\varphi_t, \dot{\varphi}_t$ are uniformly quasi continuous on $  X$. This means, for any $\varepsilon>0$ there exists an open subset $U$ with $\Vol(U)<\varepsilon$ such that on the compact set $X\setminus U$ the functions $x\mapsto \varphi_t(x), \dot{\varphi}_t(x)$  are continuous for all $t\in ]0,T']$. 
 \end{itemize}
 The first condition ensures that the inequality 
 $$
 (\omega_t +dd^c \varphi_t)^n \geq e^{\dot{\varphi}_t + F(t,\cdot, \varphi_t(\cdot))} gdV,
 $$
 holds in the pluripotential sense on $X$ for all $t\in ]0,T']$. As will be shown later, the regularizing family $\varphi_{\varepsilon}(t,x)$ as constructed in Theorem \ref{thm: regularization of subsolutions} satisfies these two conditions. 
 
 We introduce the following constants: $M:=2+M_1+M_2$, where
 $$
 M_1:=  \sup_{[0,T'] \times X} |\varphi(t,x)|;
 $$
 and 
 $$
 M_2:= \sup_{[0,T'] \times X} \left ( |\dot{\varphi}_t(x)| +  |F(t,x,\varphi_t(x))| + |F(t,x,\psi_t(x))| \right). 
 $$
 Fix $\varepsilon>0$ small enough. By uniform quasi-continuity of $\varphi_t$ and $\dot{\varphi}_t$ there exists an open set $U$ such that 
 \begin{itemize}
 	\item  $\Vol(U)<e^{-4M/\varepsilon}$; 
 	\item for all $t\in ]0,T']$, the functions $\varphi_t$ and $\dot{\varphi}_t$ are continuous on $X\setminus U$. 
 \end{itemize}
 Let $\rho_{\varepsilon}$ be the unique bounded $\theta$-psh function on $X$ such that 
 $$(\theta+dd^c \rho_{\varepsilon})^n = e^{2M/\varepsilon} {\bf 1}_U gdV + a_{\varepsilon} g dV, \ \sup_X \rho_{\varepsilon}=0.$$
 Here $0 \leq a_{\varepsilon}$ is a normalization constant. The boundedness of $g$  yields 
 $$
 \int_U e^{2M/\varepsilon} gdV \leq e^{-2M/\varepsilon} \sup_X g,
 $$
 hence for $\varepsilon>0$ small enough,  $a_{\varepsilon}\geq 1/2$.  Also,  the $L^2$-norm of the density of $(\theta+dd^c \rho_{\varepsilon})^n$ is uniformly bounded hence by \cite[Proposition 2.6]{EGZ09},  $\rho_{\varepsilon}$ is uniformly bounded. 
 
Set, for  $(t,x) \in X_{T'}$,
$$
u_{\varepsilon}(t,x) := (1-\varepsilon) \varphi_t + \varepsilon \rho_{\varepsilon}.
$$
We prove that $u_{\varepsilon} - \psi -C\varepsilon t\leq 0$ in $[0,T']\times X$, where $C:= 2M + M_1\kappa_F$. 
 By contradiction assume that it were not the case. Since the function is upper semicontinuous on the compact set $[0,T']\times X$, its maximum is attained at some $(t_0,x_0)\in ]0,T'] \times X$. We then have 
 $$
 u_{\varepsilon}(t_0,x_0) -\psi(t_0,x_0) \geq C\varepsilon t_0>0, 
 $$
 hence 
\begin{equation}
	\label{eq:gen.comp.1}
	\varphi(t_0,x_0) \geq \psi(t_0,x_0) -\varepsilon M_1.
\end{equation}
By the classical maximum principle we have
\begin{equation}
	\label{eq:gen.comp.2}
(1-\varepsilon) \dot{\varphi}_{t_0}(x_0) \geq C\varepsilon + \dot{\psi}_{t_0}(x_0).
\end{equation}
By assumption $(1)$, $t\mapsto \dot{\varphi}_t(x)$  is continuous on $]0,T'[$ for all $x\in X$.  Since $\varphi$ is a subsolution to \eqref{eq: MAF as measures},  Lemma \ref{lem: sub super solution pointwise}  then ensures that 
\begin{equation*}
	%\label{eq:gen.comp.3}
	(\omega_{t_0}+dd^c \varphi_{t_0})^n \geq e^{\dot{\varphi}_{t_0} + F(t_0,\cdot, \varphi_{t_0}(\cdot))} gdV
\end{equation*}
holds in the sense of measures on $X$. By construction of $\rho_{\varepsilon}$ we also have 
\begin{equation*}
	%\label{eq:gen.comp.3}
	(\omega_{t_0}+dd^c \rho_{\varepsilon})^n \geq (\theta+dd^c \rho_{\varepsilon})^n \geq e^{f_{\varepsilon}} gdV,
\end{equation*}
where 
$$
f_{\varepsilon}:= \begin{cases}
 \frac{2M}{\varepsilon}	\ , \ x \in U , \\
	-\log 2\ , \ x\in X\setminus U.
\end{cases}
$$
It then follows from Lemma \ref{lem: MA mixed} that 
\begin{equation}
	\label{eq:gen.comp.4}
(\omega_{t_0}+dd^c u_{\varepsilon}(t_0,\cdot))^n \geq e^{h_{\varepsilon}}gdV,
\end{equation}
where 
$$
h_{\varepsilon}:= \begin{cases}
	M\ , \ x\in U,\\
	(1-\varepsilon) \dot{\varphi}_{t_0} + (1-\varepsilon) F(t_0,x,\varphi_{t_0}(x)) -\varepsilon \log 2 \ , \ x \in X\setminus U. 
\end{cases}
$$
The assumption $(2)$ ensures that $h_{\varepsilon}$ is lower semicontinuous on $X\setminus U$. The  choice of $M$ then shows that $h_{\varepsilon}$ is lower semicontinuous on $X$ and 
\begin{equation}
	\label{eq:gen.comp.5}
	h_{\varepsilon}(x) \geq (1-\varepsilon) \dot{\varphi}_{t_0}(x) + (1-\varepsilon) F(t_0,x,\varphi_{t_0}(x)) -\varepsilon \log 2, \ \forall x\in X.
\end{equation}
It then follows from Lemma \ref{lem: towards viscosity almost everywhere} that \eqref{eq:gen.comp.4} holds in the viscosity sense. The function $x\mapsto \psi_{t_0}(x) -\psi_{t_0}(x_0) + u_{\varepsilon}(t_0,x_0)$ is a smooth upper test for $u_{\varepsilon}(t_0,\cdot)$ at $x_0$, hence
$$
(\omega_{t_0}+dd^c \psi_{t_0})^n \geq e^{h_{\varepsilon}} gdV
$$
holds in the classical sense at $x_0$. Now from \eqref{eq:gen.comp.2} and \eqref{eq:gen.comp.5} we have  
\begin{equation}
	\label{eq:gen.comp.6}
	(\omega_{t_0} +dd^c \psi_{t_0})^n(x_0) \geq e^{\dot{\psi}_{t_0}(x_0)+ (C-\log 2)\varepsilon  + F(t_0,x_0,\varphi_{t_0}(x_0)) }g(x_0)dV.
\end{equation}
It follows from \eqref{eq:gen.comp.1} and the monotonicity of $r\mapsto F(t,x,r)$ that 
\begin{eqnarray*}
	F(t_0,x_0,\varphi_{t_0}(x_0)) &\geq &  F(t_0,x_0,\psi_{t_0}(x_0)-M_1\varepsilon) \\
	& \geq &  F(t_0,x_0,\psi_{t_0}(x_0)) -\kappa_F M_1 \varepsilon. 
\end{eqnarray*}
Hence
\begin{eqnarray*}
	(1-\varepsilon) F(t_0,x_0,\varphi_{t_0}(x_0)) &\geq &    F(t_0,x_0,\psi_{t_0}(x_0)) -\kappa_F M_1 \varepsilon - M\varepsilon\\
	&\geq &    F(t_0,x_0,\psi_{t_0}(x_0)) -(C-M) \varepsilon. 
\end{eqnarray*}
This together with \eqref{eq:gen.comp.6} gives a contradiction since $\psi$ is a solution to \eqref{eq: MAF as measures}. 

We thus have that $u_{\varepsilon} \leq \psi + C \varepsilon t$ on $[0,T']\times X$. Letting $\varepsilon\to 0$ we arrive at the conclusion.
%
%Assume, to seek a contradiction, that $\varphi-\psi-\varepsilon t$ attains a maximum on $[0,T']$ at some point $(t_0,x_0)$ with $t_0>0$. Then the classical comparison principle and Lemma \ref{lem: towards viscosity almost everywhere} 
%below show that 
%\begin{flalign*}
%(\omega_{t_0} +dd^c \psi_{t_0})^n & \geq e^{\dot{\varphi}_{t_0} +F(t_0,x_0,\varphi_{t_0}(x_0))} g(x_0)dV\\
%& > e^{\dot{\psi}_{t_0} + F(t_0,x_0,\psi_{t_0}(x_0))} g(x_0)dV, 	
%\end{flalign*}
%which is a contradiction. Letting $\varepsilon \to 0$ we arrive at $\varphi\leq \psi$. 
%
\smallskip

\noindent{\it Step 1.2.} We next remove the assumptions on $\varphi$ in Step 1.1.  

 Using Proposition \ref{thm: regularization of subsolutions} we can find $\varphi^{\varepsilon}$ 
 which are pluripotential subsolutions to  \eqref{eq: MAF as measures} and which satisfy the assumptions in Step 1.1. Indeed, it suffices to check that $\varphi_t$ is uniformly quasi-continuous on $X$. But this holds by quasi-continuity of $\varphi_t$ and by the Lipschitz condition. More precisely,  for fix $\varepsilon>0$, there exists an open subset $U$ with $\Vol(U) <\varepsilon$ such that $\varphi_t$ is continuous on $X\setminus U$ for all $t\in ]0,T[ \cap \mathbb{Q}$. The continuity of $\varphi_t$ on $X\setminus U$ for irrational points $t$ follows from the Lipschitz property of the family $\varphi_t$.

Thus the previous step applies and yields $\varphi^{\varepsilon} \leq \psi+O(\e)$. 
Letting $\varepsilon\to 0$ we arrive at $\varphi\leq \psi$. 

\medskip

\noindent{\it Step 2.}
We finally remove the smoothness assumption on the data and the Lipschitz condition on $\varphi$ by using 
the stability result above together with an argument from \cite{GLZstability}. 

Let $(g_j,F_j)$ be  smooth approximants of $(g,F)$.  Fix  $\varepsilon>0$ small enough  and consider 
\[
\varphi_{\varepsilon,j}(t,x):= (1-\delta_j) \varphi(t+\varepsilon, x) + \delta_j  \rho_j + n \log (1 - \delta_j) - (B \delta_j  + C\varepsilon+ \eta_j)t,
\]
where $\delta_j \in ]0,1/2[$ will be specified later, and  $\rho_j\in \PSH(X,\theta)$ is the unique 
solution to 
\[
(\theta +dd^c \rho_j)^n = \left ( a_j+ \frac{|g_j-g|}{\|g_j-g\|_p}\right ) dV,
\]
normalized by $\sup_X \rho_j=0$
for a normalization constant $a_j\geq 0$.

We are going to prove that, for suitable choices of $B,M_j$, the function $\varphi_{\varepsilon,j}$ is a pluripotential subsolution to \eqref{eq: MAF as measures} with data $(\omega_{\varepsilon,j},g_j,F_{j})$, where 
\begin{itemize}
	\item $\omega_{\varepsilon,j}(t):= \omega(t+\varepsilon) + 2^{-j} \Theta$; 
	\item $0<g_j$ is smooth and $\|g_j-g\|_p\to 0$; 
	\item $F_{j}$ is smooth in $]0,T[ \times X \times \bR$ with the same Lipschitz   and semi-convexity constants as $F$,
	 and $F_j$  locally uniformly converge to $F$.
\end{itemize}

Let $\psi_{\varepsilon,j}$ be the unique smooth solution to \eqref{eq: MAF as measures} with the above data $(\omega_{j,\varepsilon}, g_j, F_j)$ such that $\psi_{\varepsilon,j}(0,\cdot) \geq  (1-\delta_j)\varphi(\varepsilon,\cdot)$ and 
\[
\int_X\psi_{\varepsilon,j}(0,x) dV(x) \leq \int_X (1-\delta_j) \varphi(\varepsilon,x) dV(x) + 2^{-j}.
\]

It follows from Proposition \ref{thm: stability} that $\psi_{\varepsilon,j}$ converge, as $j\to +\infty$ and $\varepsilon\to 0$, to some $\tilde{\psi} \in \Pc(X_{T'})$ which is a pluripotential solution to (CMAF) with data $(\omega_t,g,F)$ and initial value $\varphi_0$. Moreover, $\tilde{\psi},\psi$ satisfy the assumptions of Theorem \ref{thm: weak comparison principle v2}, hence $\tilde{\psi}\leq \psi$. 

We now prove that $\varphi \leq \tilde{\psi}$ by showing that $\varphi_{\varepsilon,j}$ is a subsolution
to an approximate \eqref{eq: MAF as measures}.
 Since $\varphi$ is locally uniformly Lipschitz, there exists a constant $C_1$ (depending also on $\varepsilon$) such that 
\[
\sup_{[\varepsilon,T']\times X} | \dot{\varphi} |\leq C_1,
\]
hence 
\begin{flalign*}
\partial_t \varphi(t+\varepsilon,x)& \geq (1-\delta_j) \partial_t \varphi(t+\varepsilon,x) - C_1 \delta_j\\
& = \partial_t \varphi_{\varepsilon,j} (t,x) 	 + (B-C_1)\delta_j + C\varepsilon+ \eta_j,
\end{flalign*}
A direct computation yields
\begin{flalign*}
(\omega_{\varepsilon,j} + & dd^c \varphi_{\varepsilon,j})^n \geq 	e^{n\log (1-\delta_j) + \partial_t \varphi(t+\varepsilon,x) + F(t+\varepsilon,x,\varphi(t+\varepsilon, x))} gdV + \delta_j^n \frac{|g_j-g|}{\|g_j-g\|_p}dV. 
\end{flalign*}
Set  $M_{\varphi} := \sup_{X_T} |\varphi|$ and $J:= [-M_{\varphi}, M_{\varphi}]$ and
\[
\eta_{j}:= \sup \left \{  |F(t,x,r)-F_j(t,x,r)| \setdef   (t,x,r) \in [0,T'] \times X\times  J  \right \}. 
\]
Then $\eta_j\to 0$ as $j\to +\infty$.  Setting 
\[
\delta_j := e^{C_1 + M_F} \|g_j-g\|_p, \ 
\]
where
\[
M_F:= \sup\{ |F(t,x,r)| \setdef (t,x,r) \in [0,T[ \times J\},
\]
and considering $j$ large enough (so that $\delta_j\leq 1/2$), 
we obtain
\begin{flalign*}
(\omega_{\varepsilon,j} + & dd^c \varphi_{\varepsilon,j})^n \geq 	e^{n\log (1-\delta_j) + \partial_t \varphi(t+\varepsilon,x) + F(t+\varepsilon,x,\varphi(t+\varepsilon, x))} g_jdV. 
\end{flalign*}

The Lipschitz condition on $F$ ensures that
\[
F(t+\varepsilon, x, \varphi(t+\varepsilon,x)) - F(t, x, (1-\delta_j) \varphi(t+\varepsilon,x)) \geq -C_2 (\delta_j+ \varepsilon), 
\]
for a uniform constant $C_2>0$. Choosing $B,C>0$ large enough and using $\log (1-\delta_j) \geq - \delta_j$,  
we conclude that
\begin{flalign*}
(\omega_{\varepsilon,j}(t,x) +  dd^c \varphi_{\varepsilon,j}(t,x))^n & \geq 	e^{ \partial_t \varphi_{\varepsilon,j} + F(t,x,\varphi_{\varepsilon,j})+ \eta_j} g_jdV\\
& \geq e^{ \partial_t \varphi_{\varepsilon,j} + F_j(t,x,\varphi_{\varepsilon,j})} g_jdV. 
\end{flalign*}
Thus $\varphi_{\varepsilon,j}$ is a subsolution
to   \eqref{eq: MAF as measures} for the data $(g_j,F_j,\omega_{\e,j})$.

We can now apply Step 1 to obtain $\varphi_{\varepsilon,j} \leq \psi_{\varepsilon,j}$. 
Letting $j\to +\infty$ and then $\varepsilon\to 0$   eventually  shows that $\f \leq \p$.
\end{proof}

We have used the following straightforward extension 
of \cite[Proposition 1.5]{EGZ11} :

\begin{lemma}
	\label{lem: towards viscosity almost everywhere}
	Assume that $u$ is a psh function in an open set $U\subset \mathbb{C}^n$ and $(dd^c u)^n \geq e^{f} dV$ in the pluripotential sense, where $f$ is lower semicontinuous in $U$. 
	Then the inequality holds in the viscosity sense. 
\end{lemma}

  \begin{cor}
  There exists a unique solution to the Cauchy problem for \eqref{eq: MAF as measures} which is locally uniformly semi-concave in $t$.
  It is the envelope of pluripotential subsolutions.
  \end{cor}

  \subsubsection{Stability 2}

  We are now ready to prove \ref{main Thm C} of the introduction.
    We assume here that 
 \begin{itemize}
\item  $ F, G : \Hat{ X}_T := [0,T[ \times X \times \R \rightarrow \R $ are continuous;
\item $F,G$ are increasing in the last variable;
 \item $F,G$ are  uniformly Lipschitz in $r$ with Lipschitz constants $L_F,L_G$. 
 \item $0 \leq f, g \in L^p (X)$  with $p > 1$.
 \end{itemize}
 The Lipschitz assumption on $F$  means that for all $(t,x)\in X_T$,
 $$
 |F(t,x,r) - F(t,x,r')| \leq L_F|r-r'| . 
 $$
  
  \begin{theorem}
  Assume that  $\f\in \Pc(X_T,\omega)$ is a solution to the parabolic equation \eqref{eq: MAF as measures} with 
  admissible data $(F,f)$ and $\psi \in \Pc(X_T,\omega)$ is a bounded solution to   \eqref{eq: MAF as measures} with 
  admissible data $(G,g)$. 
 
  There exists  $\alpha \in ]0,1[$ and for any $\e > 0$ there  exists  $A (\e) > 0$  
  %depending on  $\e$, $X, \theta, n,  p, L$,  a uniform bound on  $ \f,\psi$, $\dot \f$, $\dot \psi$  on $[\e,T[ \times X$ 
% and a uniform bound on $\Vert f \Vert_p$ and $\Vert g\Vert_p$ 
such that 
$$
 \sup_{[\e,T[ \times X} \vert \f - \psi\vert  \leq  A (\e) \Vert \f_\e - \psi_\e\Vert_{L^1 (X)}^{\alpha} + T \, \sup_{\Hat{X}_T} \vert F- G \vert +  A (\e)  \,  \Vert g - f \Vert_p^{1 \slash n}.
 $$
 
 In particular if
 \begin{itemize}
\item $(g_j)$ are  densities which converge to $g$ in $L^p(X)$,
\item $F_j$ converges to $F$ locally uniformly.
%\item   $\omega_{t,j}$ are smooth semi-positive forms smoothly converging to $\omega_t$,
\item $\f_{0,j}$ are bounded $\omega_{0}$-psh functions converging in $L^1(X)$ to $\f_0$,
\end{itemize}
then  $\Phi(F_j,g_j,\f_{0,j})$ locally uniformly converges to 
$\Phi(F,g,\f_{0})$.
  \end{theorem}
  
  We denote here by $\Phi (F, g, \f_0)$ the solution to the Cauchy problem for  the admissible data $(F,g, \f_0)$.

 \begin{proof}
 Set $\Phi^j=\Phi(F_j,g_j,\f_{0,j})$ and $\Phi=\Phi(F,g,\f_{0})$.
 The quantitative estimate is a simple consequence of Proposition \ref{pro:stabestimate} below.  The norm $||\Phi^j_{\e}-\Phi_{\e}||_{L^1(X)}$ is controlled by $\|\Phi^j-\Phi\|_{L^1([\varepsilon,T] \times X)}$ as follows from Lemma \ref{lem:L1Slice-L1}. By Proposition \ref{pro:stab}, $\Phi^j$ converges in $L^{1}_{\loc}(X_T)$ to $\Phi$, hence the last statement of the theorem follows.
 \end{proof}

  The   stability result is a consequence of  the following   quantitative version of the comparison principle :

  \begin{pro} \label{pro:stabestimate}
   Assume  $\f \in \Pc(X_T,\omega)$ is a subsolution to \eqref{eq: MAF as measures} with data $(F,f)$,
  $\psi  \in \Pc(X_T,\omega)$ is a supersolution to \eqref{eq: MAF as measures} with data $(G,g)$.
   
 Fix $\varepsilon > 0$. There exists $\alpha, A,B >0$ such that for all $(t,x) \in [\e,T[ \times X$,
 \begin{flalign*}
\f (t,x) - \psi (t,x)  \leq B  \Vert (\f_\e - \psi_\e)_+ \Vert_{L^1 (X)}^{\alpha} 
 + T \, \sup_{\Hat{X}_T} (G-F)_+  + A \, \Vert (g - f)_+ \Vert_p^{1 \slash n},
 \end{flalign*}
 where $A, B>0$ depend on $X, \theta, n,  p$ and a uniform bound on 
 $\dot \f $, $\dot \psi$, $ \f ,\psi$ on the set $[\e,T[ \times X$,  $\sup_{X_T} G(t,x, \sup_{X_T} \varphi)$ and $L_G$.   
  \end{pro}

  \begin{proof}
  We use a perturbation argument as in \cite{GLZstability} which goes back to the work of  Ko{\l}odziej  \cite{Kolodziej_1996Monge}. For convenience we normalize $\theta$ so that $\int_X d V = \int_X \theta^n = 1$. Set 
  $$
  m_0:= \inf_{X_T} \varphi, \; \; 
  m_1(\varepsilon) := \inf_{[\varepsilon,T[} \dot{\varphi}
   \; \; \text{ and } \; \; 
  M:= \sup_{\Hat{X}_T} (G-F)_+.
  $$
  
 We first assume that $\Vert (g - f)_+ \Vert_p > 0$. It follows from  \cite{EGZ09} (see also \cite{Kol98} in the K\"ahler case)
that  there exists  $\rho \in {\rm PSH} (X,\theta) \cap L^{\infty} (X)$, normalized by  $\max_X \rho = 0$, such that
 \begin{equation} \label{eq:yaubis}
 (\theta + dd^c \rho)^n = \left(a +  \frac{(g - f)_+}{\Vert (g - f)_+ \Vert_p}\right) d V, 
 \end{equation}
where $a \geq 0$ is a normalizing constant given by
 $$
 a := 1 - \frac{\Vert (g - f)_+\Vert_1}{\Vert (g - f)_+ \Vert_p} \in [0,1].
 $$
 
 We moreover have a uniform bound on $\rho$  which only depends on the $L^p$ norm of the density
 of $(\theta + dd^c \rho)^n$ which is here bounded from above by $2$,
 \begin{equation} \label{eq:est0}
 \Vert \rho \Vert_{\infty} \leq C_0 (a + 1) \leq 2 C_0,
 \end{equation}
 where $C_0 > 0$ is a uniform constant depending only on $(X,\theta, p)$.

For $0 < \delta < 1 $ and $(t,x) \in X_T$ we set
 $$
 \f_\delta (t,x) := (1 - \delta) \f (t,x) + \delta  \rho + n \log (1 - \delta) - B \delta t - M t.
 $$
 The plan is to choose $B>0$  in such a way that $\f_{\delta}$ is a subsolution to \eqref{eq: MAF as measures} on $[\varepsilon,T[$ with data $(G,g)$. The conclusion will then follow
 from the comparison principle (Theorem \ref{thm: weak comparison principle v3}).

 Observe that for almost all $t \in [\varepsilon,T[$ fixed, $\f_\delta(t,\cdot) $ is $\omega_t$-plurisubharmonic on $X$ and 
 $$
 (\omega_t + dd^c \f_{\delta} (t,\cdot))^n \geq (1 - \delta)^n (\omega_t + dd^c \f_t)^n + \delta^n (\theta + dd^c \rho)^n.
 $$
Using that $\f$ is a subsolution to \eqref{eq: MAF as measures} with data $(F,f)$, we infer
 \begin{equation} \label{eq:subest}
 (\omega_t + dd^c \f_{\delta} (t,\cdot))^n \geq   e^{\dot \f_t + F (t,\cdot,\f_t) + n \log (1- \delta)} f d V + \delta^n \frac{(g - f)_+}{\Vert (g - f)_+ \Vert_p} d V.
 \end{equation}
Noting that $\f \geq \f_{\delta} + \delta \f$ and recalling that $G$ is increasing in the last variable, we obtain
 \begin{eqnarray*}
&& \dot{\f} (t,x)  + F (t,x ,\f (t,x)) + n \log (1-\delta) \\
&\geq & \dot{\f}_\delta (t,x) + \delta \dot{\f} (t,x) + G \left(t,x, \f (t,x)\right) -M +
  n \log (1-\delta) + B \delta +M \\
 &\geq & \dot{\f}_\delta (t,x) + \delta \dot{\f} (t,x) + G \left(t,x,\f_\delta (t,x) + \delta \f (t,x)\right) +
  n \log (1-\delta) + B \delta \\
 & \geq & \dot{\f}_\delta (t,x)  + G \left(t,x,\f_\delta (t,x)  + \delta m_0 \right)  + \delta m_1 (\e) + n \log (1-\delta) + B \delta. 
 \end{eqnarray*}
The Lipschitz condition on $G$ yields, writing $L=L_G$,
 \begin{eqnarray*}
&&\dot{\f} (t,x)  + F (t,x ,\f (t,x)) + n \log (1-\delta) \\
&\geq & \dot{\f}_\delta  (t,x) + G (t,x,\f_\delta (t,x))
+ B \delta - L \delta |m_0| + \delta m_1 (\e) +  n \log (1-\delta).
 \end{eqnarray*}
Using the elementary inequality $\log (1 - \delta) \geq - 2 (\log 2 ) \delta$ for  $0 < \delta \leq 1 \slash 2$, 
it follows that for $0 < \delta \leq 1 \slash 2$,
 $$
 B \delta - L \delta |m_0|  + m_1 (\e) \delta + n \log (1-\delta) \geq  (B  - L  |m_0| + m_1 (\e) - 2 n \log 2) \delta.
 $$
  
\smallskip  
  
We now choose $B :=  L  |m_0|  +  2 n \log 2 - m_1 (\e)$ so that
  $$
  \dot{\f} (t,x)  + F (t,x ,\f (t,x)) + n \log (1-\delta) \geq  \dot{\f}_\delta  (t,x) + G (t,x,\f_\delta (t,x)),
  $$
   which, together with  (\ref{eq:subest}),  yields
\begin{equation}  \label{eq: subest2} 
 (\omega_t + dd^c \f_{\delta} (t,\cdot))^n \geq   e^{\dot{\f}_\delta  (t,\cdot) + G (t,\cdot,\f_\delta (t,\cdot))} f d V + \delta^n \frac{(g - f)_+}{\Vert (g - f)_+ \Vert_p}.
 \end{equation} 
On the other hand,  if we set 
$$
M_1 (\e) := \sup_{[\e,T[ \times X}  \dot\f , \;
M_0 := \sup_{X_T} \f
\text{ and }
M_2 := \sup_{X_T} G (t,x,M_0),
$$
then the Lipschitz property of $G$ ensures, for $(t,x) \in [\e,T[ \times X$, 
\begin{eqnarray*}
\dot{\f}_\delta (t,x) + G (t,x,\f_\delta (t,x)) & \leq & (1- \delta) \sup_{X_T} \dot{\f} + \sup_{X_T} G (t,x, (1-\delta) \f (t,x))  \\
& \leq & (1- \delta) M_1 (\e) + \sup_{X_T} G (t,x, (1-\delta)M_0) \\
& \leq & (1- \delta) M_1 (\e) +  M_0 L \delta + M_2 \\
& \leq & M_2 + \max \{L M_0, M_1 (\e)\}. 
\end{eqnarray*}
Using  \eqref{eq: subest2} we conclude that for $0 < \delta < 1 \slash 2$, $x\in X$, and almost all $t\in [\varepsilon,T[$,
\begin{equation} \label{eq:subest3}
(\omega_t + dd^c \f_{\delta} (t,\cdot))^n \geq   e^{\dot{\f}_\delta  (t,\cdot) + G (t,\cdot,\f_\delta (t,\cdot))} \left(f  + \delta^n e^{ - M_3 (\e)}\frac{(g- f)_+}{\Vert (g - f)_+ \Vert_p}\right) d V,
\end{equation}
 where $M_3 (\e) :=  M_2 + \max \{L M_0, M_1 (\e)\}$.

To conclude that $\f_{\delta}$ is a subsolution, we finally set
 \begin{equation} \label{eq:est1}
 \delta  := \Vert (g - f)_+\Vert^{1 \slash n}_p e^{ M_3 (\e)\slash n}\cdot
 \end{equation}
 Assume first that $\Vert (g - f)_+\Vert_p \leq 2^{- n} e^{ - M_3 (\e)}$  so that  $\delta \leq 1 \slash 2$. 
 It follows from (\ref{eq:subest3}) that, for almost all $t \in [\e,T[$,
 \begin{eqnarray*}
 (\omega_t + dd^c \f_{\delta} (t,\cdot))^n & \geq &   e^{\dot{\f}_\delta  (t,\cdot) + G (t,\cdot,\f_\delta (t,\cdot))} (f  +  (g - f)_+) d V \\
 & \geq & e^{\dot{\f}_\delta  (t,\cdot) + G (t,\cdot,\f_\delta (t,\cdot))} g d V,
 \end{eqnarray*}
 hence  $\f_\delta$ is a subsolution to\eqref{eq: MAF as measures} for the data   $(G,g)$ on $[\varepsilon,T[$.
 The comparison principle 
 %(Theorem \ref{thm:CP}) 
 ensures that for all $(t,x) \in [\e,T[ \times X$,
 $$
 \f_{\delta} (t,x) -\psi (t,x) \leq \max_X (\f_\delta (\e,\cdot) - \psi (\e,\cdot))_+.
 $$
Together with (\ref{eq:est0}) and (\ref{eq:est1}) we obtain,  for $(t,x) \in [\e,T[ \times X$,
\begin{equation*} %\label{eq:fest}
\f (t,x)- \psi (t,x) \leq \max_X (\f (\e,\cdot) - \psi (\e,\cdot))_+ + T M +  A_1 (\e)
 \Vert (g - f)_+\Vert^{1 \slash n}_p, 
\end{equation*}
where 
$$ 
A_1(\varepsilon) :=  (M_0  + 2 n \log 2 + B  T) e^{M_2 (\e) \slash n}\cdot
$$ 
 
 When $\Vert (g - f)_+\Vert_p >  2^{- n} e^{- M_2 (\e)}$, we   choose  $A_2 (\e)> 0$ so that  
 $$
\sup_{X_T} (\f (t,x)- \psi (t,x) ) \leq \max_X (\f _\e - \psi_\e)_+ + A_2 (\e)2^{- n} e^{- M_2 (\e)}.
 $$ 
We eventually set $A (\e) = \max \{A_1 (\e),A_2 (\e)\}$.

\smallskip
 
Assume finally that $\Vert (g-f)_+\Vert_p = 0$ which means that $g \leq f$ almost everywhere in $X$.
 In this case, the function $(t,x) \mapsto \varphi(t,x) - Mt$ is a subsolution to \eqref{eq: MAF as measures} with data $(G,g)$ and the conclusion follows from  the comparison principle. 

Now observe that $\psi_{\varepsilon}$ is a supersolution to the degenerate elliptic equation
$$
(\omega_\e + dd^c \psi_\e)^n \leq e^{D (\e)} g d V,
$$
where $D (\e)$ is an upper bound of $\dot{\psi} (t,x) + G (t,x,\psi (t,x))$ on $[\e,T[ \times X$.

By the $L^{\infty}$-$L^{1}$ semi-stability theorem of \cite[Proposition 3.3]{EGZ09}, it follows that there exits $\alpha \in ]0,1[$ and a constant $C (\e) > 0$ depending on $D(\varepsilon)$, $p$, $\theta$, and $\|g\|_p$ such that 
$$
\max_X (\f _\e - \psi_\e)_+ \leq C (\e) \Vert (\f _\e - \psi_\e)_+\Vert_{L^1 (X)}^{\alpha},
$$
concluding the proof.
  \end{proof}

  \section{Geometric applications}  \label{sec:lt}
  
  In this section we show that our hypotheses are satisfied when studying the K\"ahler-Ricci flow on a 
  compact K\"ahler variety with 
log terminal singularities. We prove the existence and study the long term behaviour of the normalized K\"ahler-Ricci flow (NKRF for short) on such varieties
starting from an arbitray closed positive current with bounded potential.

The definition and study of the K\"ahler-Ricci flow on mildly singular projective varieties has been undertaken by Song and Tian in \cite{ST12,ST17}. A different viscosity approach has been developed by  Eyssidieux-Guedj-Zeriahi in \cite{EGZ16,EGZ16b}.

Our approach allows one to avoid any projectivity assumption on the varieties, to deal with more general singularities, to avoid any continuity 
assumption on the data, and also provide more general uniqueness and stability results.
%Since these are extension of results from \cite{EGZ16,EGZ16b} which dealt with varieties having canonical singularities,
%we make the exposition brief, merely emphasizing the new arguments and difficulties that need to be overcome.
The whole discussion extends to the case of Kawamata log terminal pairs but we leave this discussion for later works.

 \subsection{Analytic approach to the Minimal Model Program}
 
  \subsubsection{Log terminal singularities}
 
 Let $Y$ be an irreducible compact K\"ahler normal complex analytic space with only terminal singularities.  
 Let $\pi: X\to Y$
be a log-resolution, i.e.: $X$ is a compact K\"ahler manifold, $\pi$ is a bimeromorphic projective morphism
and $\mathrm{Exc}(\pi)$ is a  divisor with simple normal crossings. Denote by $\{E\}_{E\in \mathcal{E}}$ the family
of the irreducible components of  $\mathrm{Exc}(\pi)$.  With this notation, one has furthermore
$$
K_X\equiv \pi^* K_Y + \sum_E a_E  E  
$$
where $-1< a_E \in \Q$,  $K_Y$ denote the first Chern class in Bott-Chern cohomology of the $\Q$-line bundle  $O_Y(K_Y)$ on $Y$,  whose restriction to the smooth locus is 
the 
line bundle whose sections are holomorphic top dimensional  forms (canonical forms), $K_X$ the canonical class of $X$ 
and $E$ also denotes with a slight abuse of language the cohomology class of $E$ (we refer to \cite{KM} for more details). 

The log terminal condition $a_E>-1$ means that for every non-vanishing locally defined multivalued canonical form $\eta$ defined over $Y$, 
the holomorphic multivalued canonical form $\pi^* \eta$ on $X$ has poles or  zeroes of order $a_E$ along $E$,
so that the corresponding volume form decomposes as 
$$
\pi^*(c_n \eta \wedge \overline{\eta})=e^{w^+-w^-} dV(x),
$$
where $w^+=\sum_{a_E>0} a_E \log |s_E|_{h_E}$ and $w^-=\sum_{0>a_E>-1} a_E \log |s_E|_{h_E}$
are quasi-\psh with $e^{w^+}$ continuous and $e^{-w^{-}} \in L^p$ for some $p>1$ whose precise value
depends on $\min_E (a_E+1)$.

\subsubsection{The (normalized) K\"ahler-Ricci flow}
\label{sec:checkgeom}

The   K\"ahler-Ricci flow  is the following evolution equation of K\"ahler forms on $Y$
$$
\frac{\partial {\theta}_t}{\partial t} =-\Ric({\theta}_t),
$$
starting from an initial K\"ahler form ${\theta}_0$.
These can be written as
$$
{\theta}_t=\chi_t +dd^c \phi_t,
\; \text{ with } \;
\chi_t=  \theta_0+t\chi,
$$
where $\chi \in c_1(K_Y)$.

One can pull-back these forms via a log-resolution of singularities and consider the corresponding 
forms $\omega_t=\pi^* \chi_t$ which are big and semi-positive (they vanish along $\mathrm{Exc}(\pi)$). 
The latter satisfy our main assumptions:

\begin{lem}
Assume that $\pi:X \rightarrow Y$ is a proper holomorphic map onto a compact normal K\"ahler space $Y$,
with $\omega_t=\pi^*\chi_t$ and $\theta=\pi^* \omega_Y$, where 
 $\omega_Y$ is a K\"ahler form on $Y$.
Then there exists $A>0$ such that 
$$
\theta/A \leq \omega_t, \; \; 
-A \omega_t \leq \dot{\omega}_t  \leq A \omega_t ,
\; \;
\text{ and } \; \; 
\ddot{\omega}_t \leq A \omega_t.
$$
\end{lem}

\begin{proof}
The corresponding inequalities are valid on $Y$ since $\omega_Y,\theta_0$ are K\"ahler
and $t \mapsto \chi_t$ is smooth, as long as we work on a finite interval of time (which is
implicit).
One can then transpose the inequalities from $Y$ to $X$ by 
 the holomorphic mapping $\pi$.
\end{proof}

One can similarly consider the
 normalized K\"ahler-Ricci flow on $Y$,
$$
\frac{\partial \theta_t}{\partial t}=-\rm{Ric}(\theta_t)-\lambda \theta_t,
$$
starting from an  initial data $\theta_0 = \chi_0 + dd^c \phi_0$ with $\phi_0$ being a bounded  potential which is plurisubharmonic 
with respect to the given K\"ahler form $\chi_0$ on $Y$,
and where $\lambda \in \R$. 

By rescaling one can reduce to the cases $\lambda=1,0,-1$.
To simplify the discussion we restrict to the case $\lambda=1$.
At the cohomological level, this yields
a first order ODE showing that the cohomology class of $\theta_t$ evolves as
$$
\{ \theta_t \} =e^{-t} \{ \theta_0 \} +(1-e^{-t} ) K_Y.
$$
We thus  define  by
$$
T_{max} :=\sup\{ t>0   \setdef e^{-t}\{\theta_0\}+(1-e^{-t}) K_{Y} \in \mathcal{K}(Y) \}
$$
the maximal (cohomological) time of existence of the flow. 

\smallskip

Denote by $\mathcal{K}(Y) \subset H^{1}(Y, \mathcal{PH}_Y)$ the
open convex cone of K\"ahler classes and
let $\chi_0$ be a smooth K\"ahler representative of the K\"ahler class $\{\theta_0\}$.
Assume $h$ is a 
smooth hermitian metric on the holomorphic $\Q$-line bundle underlying $O_Y(K_Y)$. 
Then 
$
\chi:=-dd^c \log h
$
is a smooth representative of $K_Y\in H^{1}(Y, \mathcal{PH}_Y)$ and we set
$$
\chi_t= e^{-t} \chi_0+ (1-e^{-t}) \chi.
$$

The solution to the normalized K\"ahler-Ricci flow can be written as 
$\theta_t = \chi_t + dd^c \phi_t$, with $\pi^* \phi  \in \Pc(X_T)$. 
We now define 
%$$\omega_{NKRF}:=\pi^*\omega \in H^0(X, \mathcal{Z}^{1,1}_{X_T/[0,T[})$$ and  
$$
\mu_{NKRF}= c_n \frac{\pi^*\eta \wedge \overline{ \pi^*\eta}}{ \pi^* \| \eta \|^2_h} \in C^0(X,\Omega^{n,n}_X)
$$
which we view as a  continuous element of $C^0(X_T, \Omega^{n,n}_{X_T/[0,T[} )$ and $c_n$ is the unique
 complex number of modulus $1$ such that the expression is positive. 
As the notation suggests, $\mu_{NKRF}$ is independent of the auxiliary multivalued holomorphic form $\eta$ but depends on $h$.
% an auxiliary smooth metric on $Y$.
%When necessary  we shall write $\mu_{NKRF}(h)$ to emphasize this dependence.

In local coordinates $\mu_{NKRF}$ has density of the form 
$$
v_{NKRF}=\prod_E |f_E|^{2a_E} v
$$ 
where $v>0$ is smooth and  $f_E$ is an equation of $E$ in these local coordinates.

\begin{theo}\label{thm:nkrf}
 The Cauchy problem with initial data $ S_0 := \chi_0+dd^c\phi_0$
for the normalized K\"ahler-Ricci flow on $Y$ admits a unique   pluripotential solution defined
on $ [0,T_{max}[ \times Y$.
\end{theo}

\begin{proof}
Fix $T < T_{max}$. Since for any $t \in [0,T]$, $e^{-t}\{\omega_0\}+(1-e^{-t}) K_{Y} \in  \mathcal{K}(Y)$, 
there exists a smooth family of K\"ahler forms  $(\chi_t)_{0 \leq t\leq T}\in \mathcal{K}(Y)$
 such that for any $t \in [0,T]$, $\{\chi_t\} = \{\omega_t\}$.
%Observe that if $\mathcal{K}_Y$ is semi-ample then $T_{max} = + \infty$ and we can take $\chi_t := e^{-t} \chi_0 +(1-e^{-t})\chi$, where $\chi$ is a smooth semi-positive representative of the canonical class $ \mathcal{K}_Y$.

We can write $\theta_t=\chi_t + dd^c \phi_t$, where $\phi$ is a solution to the corresponding Monge-Amp\`ere flow at the level of potentials,
\begin{equation*} %\label{eq:PMAF}
(\chi_t  +dd^c \phi_t)^n=e^{\partial_t \phi+\phi_t} v_Y,
\end{equation*}
on $Y_T$ for some admissible volume form $v_Y$ on $Y$,  or equivalently
$$
(\omega_t+dd^c \f_t)^n=e^{\partial_t \f +\f_t} \mu_{NKRF}
=e^{\partial_t \f+F(t,x,\f_t)}g dV_X,
$$
on a log resolution $\pi:X \rightarrow Y$, where $\mu_{NKRF}$ is a volume form on $X$ 
which can be locally written 
$$
\mu_{NKRF}=\Pi_E |f_E|^{2a_E} dV_X=g dV_X,
$$
where 
%$$
%\p=   \sum_{a_E \geq 0} a_E \log |f_E|^2-\sum_{-1<a_E  <0} a_E \log |f_E|^2=w^+-w^-
%$$
%is a difference of quasi-\psh functions $w^{\pm}$, with 
%$\p \geq w^+-C=:w$ and 
$g= \Pi_E |f_E|^{2a_E} \in L^{p}$
for some $p>1$, since $-1<a_E $ for all $E$,
%Observe that $g$ is lower semi-continuous.
and $g>0$ almost everywhere.

We write here $\f := \pi^* \phi$ and $\omega_t := \pi^* \chi_t$.
Since $(\chi_t)_{0 \leq t\leq T}$ is a smooth family of K\"ahler forms on $Y$, it follows that the family of semi-positive 
forms $[0,T[ \ni t \longmapsto \theta_t $ 
satisfies all our requirements.

 Theorem \ref{thm:existence} can thus be applied (with $F(t,x,r) \equiv r$) and guarantees the existence
 of  a unique pluripotential
 solution to the Monge-Amp\`ere flow on $X_T$ for any fixed $T < T_{max}$ starting at $\f_0$. 
 By uniqueness all these solutions glue into a unique solution of the Monge-Amp\`ere flow on $[0,T_{max}[ \times X$ starting at $\f_0$. Pushing this solution down to $Y$ we obtain a solution to the NKRF starting at $S_0$.
\end{proof}

  \subsubsection{Song-Tian program}
  
  A natural and difficult problem is to understand the asymptotic behavior of $\omega_t$ as 
$t \rightarrow T_{max}$. 
Song and Tian have proposed in \cite{ST17} an ambitious program,
combining the Minimal Model Program and Hamilton-Perelman approach to
the Poincar\'e conjecture.

We focus here on the case when $X$ has non-negative Kodaira dimension.
One would ideally like to proceed as follows :

\begin{itemize}
\item[{\it Step 1.}] Show that $(Y,\omega_t)$ converges to a midly singular K\"ahler variety $(Y_1,S_1)$
equipped with a singular {\it K\"ahler  current} $S_1$, as $t \rightarrow T_{1,max}$;

\item[{\it Step 2.}]  Restart the NKRF on $Y_1$ with initial data $S_1$;

\item[{\it Step 3.}] Repeat  finitely many times to reach a minimal model $Y_r$ ($K_{Y_r}$ is {\it nef});

\item[{\it Step 4.}]  Study  the long term behavior of the NKRF 
and show that $(Y_r,\omega_t)$ converges to
a canonical model $(Y_{can},\omega_{can})$, as $t \rightarrow +\infty$.
\end{itemize}

This program is more or less complete in dimension $\leq 2$ (see the lecture notes by Song-Weinkove in \cite{SW13}
or Tosatti \cite{TosAFST} and references therein).
It is largely open in dimension $\geq 3$, but for {\it Step 2} which has been completed in 
\cite{ST17,GZ13,DNL14,EGZ16,Dat1}.
  
In the sequel we focus on the final {\it Step 4}, i.e. we assume that $T_{max}=+\infty$, so that  $Y$ is a minimal model with log terminal singularities.
The normalized K\"ahler-Ricci flow is then well defined for all times $t>0$, and
our goal is to understand its asymptotic behavior, as $t \rightarrow +\infty$.

 \subsection{Convergence of the NKRF}

 \subsubsection{Convergence of the NKRF on l.t. varieties of general type}  
 
 Let $Y$ be a compact K\"ahler variety with terminal singularities and 
 assume $K_Y$ is big and nef. It has been shown in \cite{EGZ09} that there exists
 a unique positive  closed  current $\omega_{KE}$ on $Y$ such that 
 \begin{itemize}
\item  $ \omega_{KE} \in C_1(K_Y)$  and it has bounded potentials;
\item $\omega_{KE}$ is smooth in $Amp(K_Y)$ where it satisfies ${\rm Ric}(\omega_{KE})=-\omega_{KE}$.
 \end{itemize}
 
The current $\omega_{KE}$ is called the {\it singular K\"ahler-Einstein current}.

 \begin{thm} \label{thm:typegeneral}
Fix $S_0$ a positive closed current with bounded potentials,
 whose cohomology class is K\"ahler.
 The normalized K\"ahler-Ricci flow 
 % exists for all times $t>0$, and 
 continuously deforms $S_0$ towards 
 %the unique  singular K\"ahler-Einstein current 
 $\omega_{KE}$, as $t \rightarrow +\infty$, 
 at an exponential speed.
 \end{thm}
 
 \begin{proof}
 It is classical that the problem boils down to solving and studying the longterm behavior
 of the parabolic scalar equation
 \begin{equation} \label{eq:type}
  (\chi_t+dd^c \f_t)^n=e^{\partial_t \f_t+\f_t} v_Y,
 \end{equation}
 with initial data $\f_0$, where $T_0=\chi_0+dd^c \f_0$,
 and $\chi_t=e^{-t}\chi_0+(1-e^{-t})\chi$.
 Here $\chi$ is a K\"ahler form representing $c_1(K_Y)$.
 
 The existence of the unique maximal solution $\f_t$ has been explained in Theorem \ref{thm:nkrf},
 so the problem is to show that $\f_t \rightarrow \f_{KE}$, as $t \rightarrow +\infty$,
 where  $\omega_{KE}=\chi+dd^c  \f_{KE}$.
 We let the reader check that 
 $$
 u(t,x)=e^{-t} \f_0+(1-e^{-t})  \f_{KE} +h(t)e^{-t}
 $$
 is a subsolution to (\ref{eq:type}), where 
 $$
 h(t)=n(e^t-1)\log (e^t-1)-ne^t \log e^t=O(t).
 $$
 The computations are the same as that of \cite[Theorem 4.3, Step 1]{EGZ16}.
 The comparison principle (Theorem \ref{thm: weak comparison principle v3})  yields
 $$
 \f_{KE}(x)-C(t+1)e^{-t} \leq u(t,x) \leq \f(t,x),
 $$
 for some uniform constant $C>0$.
 
 The proof for the upper bound is similar.
 Since $\chi$ is K\"ahler, we can fix $B>0$ such that $\omega_0 \leq (1+B) \chi$,
 thus $\chi_t \leq (1+Be^{-t}) \chi$ for all $t$.
   We set
 $$
 v_t(x):=(1+Be^{-t})\f_{KE}+Ce^{-t},
 $$
 where $C$ is chosen so that $v_0 \geq \f_0$. The function
 $v$ is a supersolution to the Cauchy problem for the parabolic equation 
 $$
   ([1+Be^{-t}]\chi+dd^c v_t)^n=e^{\partial_t v_t+v_t+n \log[1+Be^{-t}]} v_Y
   \leq e^{\partial_t v_t+v_t+n Be^{-t}} v_Y
 $$
 with initial data $\f_0$, while 
 $w(t,x)=\f(t,x)-nBte^{-t}$ is a subsolution to this equation since
 $$
    ([1+Be^{-t}]\chi+dd^c w)^n \geq    (\chi_t+dd^c \f_t)^n=e^{\partial_t \f_t+\f_t} v_Y
    =e^{\partial_t w_t+w_t+nBe^{-t}} v_Y.
 $$
 
   The comparison principle thus yields
 $$
 \f(t,x) \leq \f_{KE}(x)+C'(t+1)e^{-t}.
 $$
 The conclusion follows.
 \end{proof}

 \subsubsection{Convergence of the KRF on l.t. $\Q$-Calabi-Yau varieties}
 
In this section we study the K\"ahler-Ricci flow on a $\Q$-Calabi-Yau variety $Y$
(i.e. a Gorenstein K\"ahler space of finite index with trivial first Chern class and log-terminal singularities), and prove Theorem D 
of the introduction.

  \begin{thm} \label{thm:calabiyau}
 Fix $S_0$ a positive closed current with bounded potentials,
 whose cohomology class is K\"ahler.
 The  weak K\"ahler-Ricci flow 
 $$
 \frac{\partial \omega_t}{\partial t}=-{\rm Ric}(\omega_t)
 $$
 exists for all times $t>0$, and deforms $S_0$ towards 
the unique Ricci flat K\"ahler-Einstein current 
 $\omega_{KE}$ cohomologous to $S_0$, as $t \rightarrow +\infty$.
 \end{thm}
 
 The existence of the singular Ricci flat K\"ahler-Einstein current 
 $\omega_{KE}$ has been shown in \cite{EGZ09}, generalizing Yau's celebrated solution to the Calabi conjecture \cite{Yau78}.
 
 \begin{proof}
 It is classical that the problem boils down to solving and studying the longterm behavior
 of the parabolic scalar equation
 \begin{equation} \label{eq:CY}
  (\theta_0+dd^c \f_t)^n=e^{\partial_t \f_t} v_Y,
 \end{equation}
 with initial data $\f_0$, where $S_0=\theta_0+dd^c \f_0$.
 
 The existence of the unique semi-concave solution $\f_t$ has been explained in Theorem \ref{thm:nkrf}.
 We are going to show
that $\f_t$ uniformly converges to $\f_{KE}$, as $t \rightarrow +\infty$,
 where  $\omega_{KE}=\theta_0+dd^c  \f_{KE}$ with
 $$
 (\theta_0+dd^c \f_{KE})^n=v_Y,
 $$
 and the bounded $\theta_0$-\psh function $\f_{KE}$ is properly normalized.
 We proceed in several steps.
 
 \smallskip
 
\noindent  {\it Step 1: ${\mathcal C}^0$-bounds and normalization.}
  It follows  from the comparison principle that $(\f_t)$ 
 remains uniformly bounded : indeed $\f_{KE} - C$ (resp. $\f_{KE}+C$) provides a static subsolution (resp. supersolution)
  to the Cauchy problem if 
 $C>0$ is so large  that $\f_{KE}-C \leq \f_0$ (resp. $\f_{KE}+C \geq \f_0$).
 
 We assume without loss of generality that $v_Y$ and $\theta_0$ are normalized by
 $$
 \int_Y \theta_0^n=v_Y(Y)=1.
 $$
 The concavity of the logarithm insures that 
 \begin{eqnarray*}
 \int_Y \partial_t \f_t \, v_Y &=& \int_Y \log \left[ (\theta_0+dd^c \f_{t})^n / v_Y \right] v_Y \\
&\leq & \log \left[  \int_Y  (\theta_0+dd^c \f_{t})^n \right] =0,
 \end{eqnarray*}
hence $t \mapsto \int_Y \f_t \, v_Y$ is decreasing.
We therefore impose the normalization
 $$
 \int_Y \f_{KE} \, v_Y=\lim_{t \rightarrow +\infty} \int_Y \f_{t} \, v_Y.
 $$

\smallskip

\noindent  {\it Step 2: Monotonicity of the Monge-Amp\`ere Energy along the flow.}
We now observe that
 $t \mapsto E(\f_t)$ is increasing, where
 $$
 E(\f_t):=\frac{1}{n+1} \sum_{j=0}^n \int_Y \f_t (\theta_0+dd^c \f_t)^j \wedge \theta_0^{n-j}.
 $$
 More precisely :
 
 \begin{lemma}
	\label{lem: E differentiable}
	The function $t\mapsto E(\varphi_t)$ is differentiable almost everywhere  with
	%and the derivative is given by 
	$$
	\frac{d}{dt} E(\varphi_t) = \int_Y \dot{\varphi}_t (\theta_0+dd^c \varphi_t)^n \geq 0,
	$$
	for almost every $t\in ]0,+\infty[$. 
\end{lemma}

\begin{proof}
It is straightforward to check that $t\mapsto E(\varphi_t)$ is locally Lipschitz, its differentiability  almost everywhere  
thus follows from Rademacher theorem.  Our goal is now to compute its derivative. 

	By the Lipschitz property of $t\mapsto\varphi_t$ we can find a subset $I\subset ]0,T[$ with $]0,+\infty[\setminus I$ having measure zero such that for every $t_0\in  I$ fixed the function $t\mapsto \varphi(t,x)$ is differentiable at $t_0$ for almost every $x\in Y$. By the first observation we can also assume that $t\mapsto E(\varphi_t)$ is differentiable at  every $t\in I$. 
The semi concavity property of $\varphi_t$ in $t$  moreover ensures that for every $t\in I$  and
almost every $x\in Y$,
	$$
	\dot{\varphi}_t^{+}(x) = \dot{\varphi}_t^{-}(x).
	$$
	
	The semi concavity property of $t\mapsto \varphi_t$ ensures that, for $x\in Y$ fixed, 
	the function $t\mapsto \dot{\varphi}_t^{+}(x)$ is lower semicontinuous, while 
	$t\mapsto \dot{\varphi}_t^{-}(x)$ is upper semicontinuous in $]0,+\infty[$. 
	In particular, for $t_0\in I$ fixed,
	$$
	\liminf_{t\to t_0} \dot{\varphi_t}^+(x) \geq  \dot{\varphi}_{t_0}^{+}(x) =\dot{\varphi}_{t_0}^{-}(x) \geq \limsup_{t\to t_0} \dot{\varphi}_{t}^{-}(x), 
	$$
	for almost every $x\in Y$. 
	
Fix $t_0 \in I$ and  $t\in I, t>t_0$. By concavity of the Monge-Amp\`ere energy  (see \cite[Proposition 2.1]{BBGZ}) we obtain
	$$
	\int_Y \frac{\varphi_t-\varphi_{t_0}}{t-t_0} (\theta_0+dd^c \varphi_{t})^n \leq \frac{E(\varphi_t)-E(\varphi_{t_0})}{t-t_0} \leq \int_Y \frac{\varphi_t-\varphi_{t_0}}{t-t_0} (\theta_0+dd^c \varphi_{t_0})^n.
	$$
Using that $\varphi_t$ is a  solution to \eqref{eq:CY} and $\dot{\varphi}_t^+(x)=\dot{\varphi}_t^-(x)$ a.e., we get
	$$
	\int_Y \frac{\varphi_t-\varphi_{t_0}}{t-t_0}e^{\dot{\varphi}_t} gdV \leq \frac{E(\varphi_t)-E(\varphi_{t_0})}{t-t_0} 
	\leq \int_Y \frac{\varphi_t-\varphi_{t_0}}{t-t_0} e^{\dot{\varphi}_{t_0}}gdV .
	$$
Letting $I\ni t\to t_0$ and using Lebesgue dominated convergence theorem we arrive at the desired formula for the  derivative of $t\mapsto E(\varphi_t)$ at $t_0$. 
	
	It remains to check that $\frac{d}{dt} E(\varphi_t) \geq 0$. This follows from 
	Jensen inequality,
 $$
 \frac{d}{d t} E(\f_t)=\int_Y \dot{\f}_t (\theta_0+dd^c \f_t)^n
 \geq -\log \int_Y v_Y=0.
 $$
\end{proof}

 \smallskip

\noindent  {\it Step 3: Asymptotic behavior of $\dot{\f}_t(x)$.}
We claim that  there exists a constant $C>0$ such that for all $t \geq 1$ and $x \in Y$,
	 $$
	 |\dot{\varphi}_t(x)| \leq C.
	 $$
	 Indeed, since $t\mapsto \varphi_t(x)$ is locally uniformly Lipschitz (away from $t=0$),
	there is  $C>0$ such that  $|\varphi_{s+1} -\varphi_1| \leq  Cs$, for every $s\in [0,1]$. Fix such $s$ and consider,
	 for $ t>0$ and $x\in Y$,
	 $$
	 u_{t}(x):= \varphi(s+t+1,x) -Cs. 
	 $$
	 Observe that $u_0\leq \varphi_{1}$ and 
	 $$
	 (\theta_0+dd^c u_t)^n = e^{\dot{u}_t} v_Y.
	 $$
	 Since the function $(t,x) \mapsto \varphi(t+1,x)$ solves the above equation, it follows from  Theorem \ref{thm: weak comparison principle v3} that $u_t \leq \varphi_{t+1}$, for all $t>0$.  Thus
	 $$
	 \f_{s+t+1} \leq \f_{t+1}+Cs,
	 $$
	 and letting $s \rightarrow 0$
 yields a uniform upper bound for $\dot{\varphi}_t$. The lower bound follows similarly.

\smallskip
 
 We now claim that there is a sequence of times
 $t_j \rightarrow +\infty$ such that
$\dot{\varphi}_{t_j}(x) \rightarrow 0$ for almost every $x \in Y$.
Indeed observe that 
the functional
 $
t \mapsto  {\mathcal F}(\f_t):=E(\f_t)-\int_Y \f_t \, v_Y
 $
 is increasing along the flow :  for a.e. $t \geq 1$,
	 $$
	 \frac{d}{dt} \mathcal{F}(\varphi_t) = \int_Y \dot{\varphi}_t (e^{\dot{\varphi}_t}-1) dv_Y  
	  \geq C^{-1} \int_Y |\dot{\varphi}_t |^2 dv_Y \geq 0. 
	 $$
	 
	  Since $\mathcal{F}$ is uniformly bounded along the flow  there is   $t_j \to +\infty$ such that 
	 $$
	 \int_Y |\dot{\varphi}_{t_j}|^2 dv_Y \to 0
	 $$
Since the time derivative $\dot{\varphi}_t$ is uniformly bounded for $t\geq 1$,
	  it follows that 
	  $$
	  e^{\dot{\varphi}_{t_j}} \to  1
	  $$
	    in $L^{q}(Y,dv_Y)$ for all $1<q<2$,
	    and $\dot{\varphi}_{t_j}(x) \rightarrow 0$ for almost every $x \in Y$
	    (up to extracting and relabelling).
It follows from the elliptic $L^1$-$L^{\infty}$ stability \cite[Theorem C]{GZ12}  that
$\varphi_{t_j}$ uniformly converges to some $\p$ which satisfies
	 $$
	 (\theta_0+dd^c \p)^n =v_Y
	 $$
	 and $\int_Y \p dv_Y=\int_Y \f_{KE} dv_Y$,
	since $\int_X \varphi_t  dv_Y$ decreases to $\int_Y \f_{KE} dv_Y$.
	The uniquess of the normalized
	K\"ahler-Einstein potential \cite{EGZ09} now ensures that $\p=\f_{KE}$,
	i.e. $\varphi_{t_j}$ uniformly converges to $\f_{KE}$.
	  
	   \smallskip

\noindent  {\it Step 4: The semi-group property.}
The conclusion follows now from the fact that our equation is invariant under translations in time : observe that
for all $s>0$, the function $(t,x) \mapsto \p(t,x)=\f(t+s,x)$ is again a bounded parabolic potential, solution to the equation
$$
(\theta_0+dd^c \p_t)^n=e^{\partial_t{\p}_t} dv_Y.
$$

Fix $\varepsilon>0$ and $j$ large enough so that
$$
\sup_X |\f_{t_j}-\f_{KE}| <\e.
$$
The function $\p(t,x)=\f_{KE}(x)-\e$ is  a subsolution to the Cauchy problem for the above equation with initial data 
$\f_{t_j}$. Similarly $\f_{KE}(x)+\e$ is  a supersolution to the same Cauchy problem.
The comparison principle (Theorem \ref{thm: weak comparison principle v3}) therefore yields, for all $t \geq 0$ and $x \in X$,
$$
\f_{KE}(x)-\e \leq \f(t+t_j,x) \leq \f_{KE}(x)+\varepsilon.
$$
Letting $t \rightarrow +\infty$ and then $\e \rightarrow 0$ yields the conclusion.
 \end{proof}

  \subsubsection{Minimal models of intermediate Kodaira dimension}

We finally just say a few words of the more delicate volume collapsing case.
We assume here $Y$ is an abundant minimal model of Kodaira dimension
$1 <\kappa= kod(Y) <n$, i.e.
$K_Y$ is a semi-ample $\Q$-line bundle with $K_Y=f^*K_{Y_{can}}$,
where 
 $f: Y \rightarrow Y_{can}$ is the Iitaka fibration,   with $K_{Y_{can}}$ ample.
 
  A generic fiber $X_y=f^{-1}(y)$ is a $\Q$-Calabi-Yau variety.
We fix $h_A$ a positive hermitian metric of $A$ with curvature form $\theta_A$,
and $\eta$ local (multivalued) non-vanishing holomorphic section of $K_Y$.
It occurs that
$$
v(h_A)=c_n\frac{\eta \wedge \overline{\eta}}{||\eta ||^2_{f^*{h}_A}}
$$ 
is a globally well defined volume form on $Y$
such that the measure $f_*v(h_A)$ has density in $L^{1+\e}$ w.r.t to $\theta_A^{\kappa}$.

Generalizing   \cite{ST12,ST17}, it has been shown 
in \cite{EGZ16b} that there exists a unique bounded $\theta_A$-psh function $\f_{can}$ on $Y_{can}$ s.t.
 
 \begin{itemize}
 \item $(\theta_A+dd^c \f_{can})^{\kappa}=e^{\f_{can}} f_*(v(h_A))$;
 \item the current $\omega_{can}=\theta_A+dd^c \f_{can}$ is  {independent of $h_A$};
 \item it is  {smooth in $Y_{can}^{reg} \setminus \text{critical values  of } f$}.  
\item it satisfies $ \rm{Ric}(\omega_{can})=-\omega_{can}+\omega_{WP}$ in $Y_{can}^{reg} \setminus \rm{critical \text{ } values}$.
 \end{itemize}

The Weil-Petersson term $\omega_{WP}$ is a semi-positive $(1,1)$-form which measures the change of complex structures in the fibers 
of the Iitaka fibration. The current $T_{can}=f^* \omega_{can}$ is an important birational invariant s.t.
 $$
 T_{can}^{\kappa} \wedge \omega_{SF}^{n-\kappa} =e^{\f_{can} \circ f} v(h_A).
 $$
 Here $\omega_{SF}=\omega_0+dd^c  {\rho}$ denotes the fiberwise family of Ricci flat KE metrics,
 $$
 {\omega_{SF}}_{| X_y }=\text{unique Ricci flat metric in } {\{\omega_0 \} }_{| X_y } ,
 $$
  whose existence has been obtained in \cite{EGZ09}.

Extending the main result of \cite{EGZ16b}, the tools developed in this article allow one to establish  the following :

\begin{thm}
If $\dim_{\C} Y \leq 3$ then the normalized K\"ahler-Ricci flow    deforms 
%the K\"ahler form
$\omega_0$ towards the canonical current $T_{can}$, as $t \rightarrow +\infty$.
\end{thm}

\begin{proof}
For a suitable choice of the normalizing constants, the normalized K\"ahler-Ricci flow is equivalent 
the following parabolic complex Monge-Amp\`ere flow of potentials,
\begin{equation*} %\label{eq:flotfinal}
\frac{(\omega_t+dd^c \f_t)^n}{C_n^{\kappa} e^{-(n-\kappa)t}}=  e^{\partial_t \f+\f_t} v(h),
\end{equation*}
starting from an initial bounded potential $\f_0 \in \PSH(X,\omega_0)$.
We have normalized here both sides so that the volume of the left hand side converges to
$1$ as $t \rightarrow +\infty$. Here $C_n^k$ denotes the binomial coefficient
$
C_n^k=\binom{n}{k}.
$ 
It follows from Theorem \ref{thm:nkrf} that this flow admits a unique bounded  pluripotential solution.

Once the objects are well defined, the proof is then identical to that in \cite[Theorem D]{EGZ16b}. 
The restriction on $\dim_{\C} Y$   is related to a regularity
 issue for some families of  Ricci flat metrics.
\end{proof}

\end{document}